\documentclass[11pt,thmsa]{article}
\usepackage{amssymb}
\usepackage{amsfonts}
\usepackage{amsmath}

\numberwithin{equation}{section}
\newtheorem{theorem}{Theorem}[section]
\newtheorem{proposition}[theorem]{Proposition}

\newtheorem{remark}{Remark}[section]
\newtheorem{notation}{Notation}[section]
\newtheorem{lemma}[theorem]{Lemma}

\newenvironment{proof}[1][Proof]{\noindent\textbf{#1.} }{\hfill $\square$}

\begin{document}

\author{Guofeng Che and Haibo Chen \\
School of Mathematics and Statistics,\\
Central South University, Changsha, Hunan, 410083, P. R. China \and %
Tsung-fang Wu$^{*}$ \\
Department of Applied Mathematics \\
National University of Kaohsiung, Kaohsiung, Taiwan}
\title{Existence and multiplicity of positive solutions for fractional
Laplacian systems with nonlinear coupling }
\date{}
\maketitle

\begin{abstract}
It is well known that a single nonlinear fractional Schr\"{o}dinger equation
with a potential $V(x)$ may have a positive solution that is concentrated at
the nondegenerate minimum point of $V(x)$ as the positive parameter $%
\varepsilon $ sufficiently small (see \cite{CZ, DPW, FMV}). While in this
paper, we can find two different positive solutions for weakly coupled
fractional Schr\"{o}dinger systems with two potentials $V_{1}(x)$ and $%
V_{2}(x)$ having the same minimum point and these positive solutions are
concentrated at this minimum point. In fact that by using the energy
estimates, Nehari manifold technique and the Lusternik-Schnirelmann theory
of critical points, we obtain the multiplicity results for a class of weakly
coupled fractional Schr\"{o}dinger system. Furthermore, the existence and
nonexistence of least energy positive solutions are also explored.
\end{abstract}

\footnotetext{%
G. Che was supported by the Fundamental Research Funds for the Central
Universities of Central South University (Grant No. 2017zzts058). H. Chen
was supported by the National Natural Science Foundation of China (Grant No.
11671403). T.-F. Wu was supported in part by the Ministry of Science and
Technology, Taiwan (Grant No. 106-2115-M-390-001-MY2) and the National Center for
Theoretical Sciences, Taiwan.
\par
$^{\ast }$ Corresponding author. ~\textit{E-mail addresses}: tfwu@nuk.edu.tw
(T.-F. Wu).
\par
~~\textit{E-mail addresses}: cheguofeng222@163.com (G. Che),
math\_chb@163.com (H. Chen).} \vskip2mm

\textbf{Keywords:} Fractional Laplacian systems; Lusternik--Schnirelmann
theory; Positive solutions; Variational methods \vskip2mm

\textbf{2010 Mathematics Subject Classification.} Primary: 35J50, 35A01;
Secondary: \newline
35B38

\section{Introduction}

In this paper, we are concerned with the following nonlinear systems of two
weakly coupled fractional Schr\"{o}dinger equations with nonconstant
potentials:
\begin{equation}
\left\{
\begin{array}{ll}
\varepsilon ^{2\alpha }(-\Delta )^{\alpha }u+V_{1}\left( x\right) u=\left(
\mu _{1}|u|^{2p}+\beta \left( x\right) |u|^{q-1}|v|^{q+1}\right) u, & \text{%
in }\mathbb{R}^{N}, \\
\varepsilon ^{2\alpha }(-\Delta )^{\alpha }v+V_{2}\left( x\right) v=\left(
\mu _{2}|v|^{2p}+\beta \left( x\right) |v|^{q-1}|u|^{q+1}\right) v, & \text{%
in }\mathbb{R}^{N}\,, \\
0\leq u,v\in H^{\alpha }(\mathbb{R}^{N}), & \text{in }\mathbb{R}^{N},%
\end{array}%
\right.  \tag{$P_{\varepsilon }$}
\end{equation}%
where $\alpha \in (0,1)$, $0<q\leq p\leq 1$ if $N\leq 4\alpha $ and $0<q\leq
p<\frac{2\alpha }{N-2\alpha }$ if $N>4\alpha $, $0<\varepsilon \ll 1$ is a
small parameter and $\mu _{l},~l=1,2,$ are positive constants. Throughout
this paper, we assume that the potentials $V_{1}(x),V_{2}(x)$ and the
coupling function $\beta (x)$ satisfy the following hypotheses:

\begin{itemize}
\item[$\left( D_{1}\right) $] $V_{1}\left( x\right) $ and $V_{2}\left(
x\right) $ are bounded and uniformly continuous functions on $\mathbb{R}%
^{N}; $

\item[$\left( D_{2}\right) $] the potentials $V_{1}\left( x\right) $ and $%
V_{2}\left( x\right) $ have isolated gobal minimum points at $%
z_{1,1},z_{1,2},\ldots ,z_{1,k}$ and $z_{2,1},z_{2,2},\ldots ,z_{2,\ell }$,
respectively. Here, an isolated gobal minimum point $z_{1,i}$ of $%
V_{1}\left( x\right) $ (similarly to $z_{2,j}$ of $V_{2}(x)$) means that $%
z_{1,i}$ is the unique minimum point of $V_{1}\left( x\right) $ in $%
B_{r_{0}}(z_{1,i})$ (a neighborhood of $z_{1,i})$, where $r_{0}$ is a
positive constant. Moreover,%
\begin{equation*}
V_{1}\left( z_{1,i}\right) =\lambda _{1}\equiv \min \left\{ V_{1}\left(
x\right) :x\in \mathbb{R}^{N}\right\} >0\text{ for all }i=1,2\ldots ,k
\end{equation*}%
and%
\begin{equation*}
V_{2}\left( z_{2,j}\right) =\lambda _{2}\equiv \min \left\{ V_{2}\left(
x\right) :x\in \mathbb{R}^{N}\right\} >0\text{ for all }j=1,2\ldots ,\ell ;
\end{equation*}

\item[$\left( D_{3}\right) $] $\beta \left( x\right) $ is a nonpositive,
bounded and uniformly continuous function on $\mathbb{R}^{N};$

\item[$\left( D_{4}\right) $] there exist a positive integer $1\leq m\leq
\min \left\{ k,\ell \right\} $ and a constant $c_{0}>0$ such that $%
z_{1,i}=z_{2,j}$ and $\beta \left( x\right) \leq -c_{0}$ for all $x\in
B_{r_{0}}(z_{1,i})$ for all $1\leq i,j\leq m,$ where $r_{0}>0$ as in
condition $\left( D_{2}\right) .$
\end{itemize}

In recent years great interest has been devoted to the study of elliptic
equations involving the fractional Laplacian operator $(-\triangle)^{%
\alpha},~\alpha\in (0,1)$. This type of operators appears in a quite natural
way in many different applications, such as, fractional quantum mechanics,
geophysical fluid dynamics, population dynamics, phase transitions, flames
propagation, anomalous diffusion, crystal dislocation, materials science,
water waves and soft thin films. For more details about the applications, we
refer the readers to \cite{FQT, Las1,Las2} and the references therein.

Much attention has been paid to the single fractional Laplacian equation:
\begin{equation}
\varepsilon ^{2\alpha }(-\triangle )^{\alpha }u+V(x)u=f(x,u),\quad \text{in}~%
\mathbb{R}^{N}.  \label{1.1}
\end{equation}%
Eq. $(\ref{1.1})$ is related to the stationary analogue of the fractional
Schr\"{o}dinger equation:
\begin{equation*}
i\varepsilon \frac{\partial u}{\partial t}=\varepsilon ^{2\alpha
}(-\triangle )^{\alpha }u+V(x)u-f(x,u)\quad \text{in}~\mathbb{R}^{N},
\end{equation*}%
which was introduced by \cite{Las1,Las2} through expanding the Feynman path
integral from the Brownian--like to the L\'{e}vy--like quantum mechanical
paths. In the past several decades, with the aid of variational methods, the
existence, nonexistence, multiplicity, uniqueness, regularity and the
asymptotic decay properties of solutions for Eq. $(\ref{1.1})$ have been
obtained under various hypotheses on the potential $V(x)$ and the
nonlinearity $f(x,u)$, see \cite{CZ, DPW, Fal, FLS, FMV, Sec} and the
references therein. For instance, Frank, Lenzmann and Silvestre \cite{FLS}
studied the following single fractional Schr\"{o}dinger equation:
\begin{equation}
\left\{
\begin{array}{ll}
(-\triangle )^{\alpha }u+\lambda _{l}u=\mu _{l}|u|^{2p}u, & \text{in }%
\mathbb{R}^{N}, \\
u\in H^{\alpha }(\mathbb{R}^{N}),\quad u>0, & \text{in }\mathbb{R}^{N}\,,%
\end{array}%
\right.  \tag{$\widehat{E}_{l}$}
\end{equation}%
where $l=1,2$, and obtained the following properties of the ground states
for Eq. $(\widehat{E}_{l})$.

\begin{proposition}
\label{P1.1}Let $N\geq 1,~\alpha \in (0,1)$ and $0<p<\frac{2_{\alpha }^{\ast
}-2}{2}$, where
\begin{equation*}
2_{\alpha }^{\ast }=\left\{
\begin{array}{ll}
\frac{2N}{N-2\alpha }, & \text{\textrm{if }}N>2\alpha , \\
+\infty , & \text{\textrm{if }}N\leq 2\alpha .%
\end{array}%
\right.
\end{equation*}
\end{proposition}

Then the following hold:\newline
$(i)$ (Uniqueness) The ground state solution $U\in H^{\alpha }(\mathbb{R}%
^{N})$ for Eq. $(\widehat{E}_{l})$ is unique.\newline
$(ii)$ (Symmetry, regularity and decay) $U(x)$ is radial, positive and
strictly decreasing in $|x|$. Moreover, the function $U(x)$ belongs to $%
H^{2\alpha +1}(\mathbb{R}^{N})\cap C^{\infty }(\mathbb{R}^{N})$ and
satisfies
\begin{equation*}
\frac{C_{1}}{1+|x|^{N+2\alpha }}\leq U(x)\leq \frac{C_{2}}{1+|x|^{N+2\alpha }%
},~\forall ~x\in \mathbb{R}^{N},
\end{equation*}%
with some constants $C_{2}\geq C_{1}>0$.\newline
$(iii)$ (Non-degeneracy) The linearized operator $L_{0}=(-\triangle
)^{\alpha }+1-(2p+1)U^{2p}$ is non-degenerate, i.e., its kernel is given by
\begin{equation*}
kerL_{0}=span\{\partial _{x_{1}}U,\cdots ,\partial _{x_{N}}U\}.
\end{equation*}

Very recently, many researchers have paid attention to the following Schr%
\"{o}dinger system with constant potentials and $\beta $ is a real constant,
i.e.
\begin{equation}
\left\{
\begin{array}{ll}
-\varepsilon ^{2}\Delta u+u=\left( |u|^{2p}+\beta |u|^{p-1}|v|^{p+1}\right)
u, & \text{in }\mathbb{R}^{N}, \\
-\varepsilon ^{2}\Delta v+\omega ^{2}v=\left( |v|^{2p}+\beta
|v|^{p-1}|u|^{p+1}\right) v, & \text{in }\mathbb{R}^{N},%
\end{array}%
\right.  \label{1.3}
\end{equation}%
which arises as a model for propagation of polarized laser beams in
briefringent Kerr medium in nonlinear optics (see \cite{Ber,Fib,Lak}). With
the help of the critical point theory and the variational methods, there
have been lots of results about the existence and multiplicity of nontrivial
solutions for system (\ref{1.3}), see \cite{Cip,Fig1,Mai1} and the
references therein. In the case of $p=1,N=2,3,\beta \in \mathbb{R}$ and the
potentials may be not constants, we may get the following system:
\begin{equation}
\left\{
\begin{array}{ll}
-\varepsilon ^{2}\Delta u+V_{1}\left( x\right) u=\mu _{1}u^{3}+\beta uv^{2},
& \text{in }\mathbb{R}^{N}, \\
-\varepsilon ^{2}\Delta v+V_{2}\left( x\right) v=\mu _{2}v^{3}+\beta vu^{2},
& \text{in }\mathbb{R}^{N},%
\end{array}%
\right.  \label{1.4}
\end{equation}%
which appears in the Hartree--Fock theory for a double condensate, i.e., a
binary mixture of Bose--Einstein condensates in two different hyperfine
states $|1\rangle $ and $|2\rangle $ (cf. \cite{Esr, Hal, Tim}). Physically,
$u$ and $v$ are the corresponding condensate amplitudes, $\varepsilon ^{2}=%
\frac{h^{2}}{2m}$ and $\mu _{j}=-(N_{j}-1)U_{jj}$, where $h$ is Planck
constant, $m$ is the atom mass, $N_{j}$ is a fixed number of atoms in the
hyperfine state $|j\rangle $. Recently, various results on the existence and
concentration of solutions for system (\ref{1.4}) have been obtained, we
refer the readers to \cite{I, LW,lw1, Lin1,Pen1} and the references therein.
Peng and Li \cite{Pen1} obtained the existence of multi--spike vector
solutions for system system (\ref{1.4}) with $\beta \neq 0$. Moreover, the
attractive phenomenon for $\beta <0$ and the repulsive phenomenon for $\beta
>0$ are also explored by the authors. Lin and Wu \cite{Lin1}, use energy
estimates and category theory to prove the nonuniqueness theorem, provided $%
\beta <0$. Lin and Wei \cite{LW} proved that there exists $\beta _{0}\in (0,%
\sqrt{\mu _{1}\mu _{2}})$ such that system (\ref{1.4}) possesses a least
energy solution whenever $\beta \in (-\infty ,\beta _{0})$. When the
potentials $V_{1}(x)$ and $V_{2}(x)$ satisfy:\newline
$(\overline{V})$ $0<\inf\limits_{x\in \mathbb{R}^{N}}V_{l}(x)<\lim%
\limits_{|x|\rightarrow \infty }V_{l}(x)\leq \infty \text{ for }l=1,2.$%
\newline
Lin and Wei \cite{lw1} studied the minimization of the functional $%
E_{\varepsilon }$ for system (\ref{1.4}) on $\mathbb{R}^{N},~N=2,3$. When $%
\beta <0$ and $\varepsilon >0$ sufficiently small, problem (\ref{1.4}) has a
least energy solution $(u_{\varepsilon ,1},u_{\varepsilon ,2}),$ such that%
\begin{equation*}
\varepsilon ^{-N}E_{\varepsilon }(u_{\varepsilon ,1},u_{\varepsilon
,2})\rightarrow \alpha _{\lambda _{1},\mu _{1}}+\alpha _{\lambda _{2},\mu
_{2}},
\end{equation*}%
as $\varepsilon $ goes to zero (up to a subsequence), where $\alpha
_{\lambda _{l},\mu _{l}}$ is defined by
\begin{equation*}
\alpha _{\lambda _{l},\mu _{l}}=\inf \left\{ I_{\lambda _{l},\mu _{l}}\left(
u\right) \ |\ u\in \mathbf{M}_{\lambda _{l},\mu _{l}}\right\} ,
\end{equation*}%
\begin{eqnarray*}
I_{\lambda _{l},\mu _{l}}\left( u\right) &=&\frac{1}{2}\int_{\mathbb{R}%
^{N}}\left( \left\vert \nabla u\right\vert ^{2}+\lambda _{l}u^{2}\right)
\mathrm{d}x-\frac{\mu _{l}}{4}\int_{\mathbb{R}^{N}}u^{4}\mathrm{d}x,\quad
\forall ~u\in H^{1}\left( \mathbb{R}^{N}\right) , \\
\mathbf{M}_{\lambda _{l},\mu _{l}} &=&\left\{ u\in H^{1}\left( \mathbb{R}%
^{N}\right) \backslash \left\{ 0\right\} \ |\ \left\langle I_{\lambda
_{l},\mu _{l}}^{\prime }\left( u\right) ,u\right\rangle =0\right\} ,
\end{eqnarray*}%
for $l=1,2$. Furthermore, if $V_{1}(x)$ and $V_{2}(x)$ are of $C^{2}$
functions with nondegenerate minimum points at $z_{1}$ and $z_{2}$,
respectively, then $u_{\varepsilon ,l}$ has only one maximum point $%
z_{l}^{\varepsilon }$ that satisfies
\begin{eqnarray*}
V_{l}(z_{l}^{\varepsilon }) &\rightarrow &\inf\limits_{x\in \mathbb{R}%
^{N}}\,V_{l}(x),\quad l=1,2, \\
U_{\epsilon ,l}(y) &=&u_{\varepsilon ,l}(z_{l}^{\epsilon }+\epsilon
y)\rightarrow \omega _{\lambda _{l},\mu _{l}}(y),\quad l=1,2, \\
\frac{|z_{1}^{\varepsilon }-z_{2}^{\varepsilon }|}{\varepsilon }
&\rightarrow &+\infty ,
\end{eqnarray*}%
as $\varepsilon \rightarrow 0$ (up to a subsequence), where $\omega
_{\lambda _{l},\mu _{l}}\left( x\right) =\omega _{\lambda _{l},\mu
_{l}}\left( \left\vert x\right\vert \right) $ is the energy minimizer of the
minimum $\alpha _{\lambda _{l},\mu _{l}}>0$ (cf. \cite{Wi}) and is the
unique solution (cf.~\cite{K}) of
\begin{equation}
-\Delta u+\lambda _{l}u=\mu _{l}u^{3}\text{ in }\mathbb{R}^{N},\quad u>0%
\text{ in }\mathbb{R}^{N}.  \tag{$E_{\lambda _{l},\mu _{l}}$}
\end{equation}%
Consequently, the $u_{\varepsilon ,l}$'s satisfy
\begin{equation}
\varepsilon ^{-N}\int_{B^{N}\left( z_{l}^{\varepsilon };\varepsilon R\right)
}u_{\varepsilon ,l}^{4}\geq d>0,  \label{1.9}
\end{equation}%
as $\varepsilon \rightarrow 0$ (up to a subsequence), where $d$ and $R$ are
positive constants independent of $\varepsilon $, $B^{N}\left(
z_{l}^{\varepsilon };\varepsilon R\right) $ is an $N$ dimensional ball with
a radius $\varepsilon R$ and a center at $z_{l}^{\varepsilon }$. Hereafter,
the point $z_{l}^{\varepsilon }$ is defined as a concentration point of $%
u_{\varepsilon ,l}$ if and only if~(\ref{1.9}) holds.

In the nonlocal case, that is, when $\alpha \in (0,1)$, even in the power
type nonlinearities case, there are very few results for the fractional
Laplacian systems. In \cite{Guo1}, Guo and He studied the following
fractional Schr\"{o}dinger system with nonconstant potentials:
\begin{equation}
\left\{
\begin{array}{l}
\varepsilon ^{2\alpha }(-\Delta )^{\alpha }u+P_{1}(x)u=\left( |u|^{2p}+\beta
|u|^{p-1}|v|^{p+1}\right) u,\quad \text{in}~\mathbb{R}^{N}, \\
\varepsilon ^{2\alpha }(-\Delta )^{\alpha }v+P_{2}(x)v=\left( |v|^{2p}+\beta
|v|^{p-1}|u|^{p+1}\right) v,\quad ~\text{in}~\mathbb{R}^{N},%
\end{array}%
\right.  \label{1.10}
\end{equation}%
where $\alpha \in (0,1)$, $0<p<\frac{2\alpha }{N-2\alpha }$, $\varepsilon >0$
is a small parameter and coupling function $\beta \equiv b>0$. Under some
appropriate hypotheses on the potentials $P_{1}(x)$ and $P_{2}(x)$, the
authors obtained the existence of nontrivial nonnegative solutions for
system $(\ref{1.10})$ which concentrate around local minima of the
potentials. When $\varepsilon =1$, $P_{1}(x)=1$ and $P_{2}(x)=\omega
^{2\alpha },~\omega >0$, Guo and He \cite{Guo} obtained the existence of a
least energy solution for system $(\ref{1.10})$ on the Nehari manifold.
Furthermore, the existence of least energy positive solution with both
nontrivial components was also established. In \cite{Lu}, L\"{u} and Peng
studied system $(\ref{1.10})$ with $P_{1}(x)=P_{2}(x)=\varepsilon =1,~\left(
|u|^{2p}+\beta |u|^{p-1}|v|^{p+1}\right) u$ and $\left( |v|^{2p}+\beta
|v|^{p-1}|u|^{p+1}\right) v$ being replaced by $f(u)+\beta v$ and $%
g(v)+\beta u$, respectively. Under very weak assumptions on the nonlinear
terms $f$ and $g$, they obtained the existence of positive vector solutions
and vector ground state solutions for system $(\ref{1.10})$. Moreover, the
asymptotic behavior of the solutions as $\beta \rightarrow 0$ was also
analyzed by them. For the other related results about the fractional
Laplacian system, we refer the readers to \cite{Che1, Cho, Dip} and the
references therein.

Motivated by \cite{Guo1, LW, lw1, Lin1, Lu}, it is very natural for us to
pose some questions, in particular, such as:

$(I)$ As pointed out in \cite{Guo1,Lu}, when the coupled nonlinear terms are
replaced by $\beta \left( x\right) |u|^{q-1}u|v|^{q+1}$ and $\beta \left(
x\right) |u|^{q+1}|v|^{q-1}v$ for the coupling function $\beta <0$ in $%
\mathbb{R}^{N}$ and $0<q\leq p,$ is the existence of positive solutions for
system $(\ref{1.10})$ which concentrate around local minima of the
potentials still true?

$\left( II\right) $ Can the relationship of the minimum points of the
potentials $V_{1}(x)$ and $V_{2}(x)$ affects the number of positive
solutions for system $(P_{\varepsilon })?$

$(III)$ Under our assumptions $(D_{1})$ and $(D_{2})$, can one prove that
the positive solution of problem $(P_{\varepsilon })$ is a least energy
solution? If not, can one give some appropriate conditions on the potentials
to assure that the positive solution is a least energy solution of problem $%
(P_{\varepsilon })?$

In the present paper, by using the Nehari manifold technique, the energy
estimates and the Lusternik--Schnirelmann theory of critical points, we
study how the relationship of the minimum points of the potentials $V_{1}(x)$
and $V_{2}(x)$ affects the number of positive solutions for system $%
(P_{\varepsilon })$, provided $\beta $ is a nonpositive, bounded and
uniformly continuous function on $\mathbb{R}^{N}$, and will give answers to
Questions $(I)-(III)$. Moreover, whether the positive solution for problem $%
\left( P_{\varepsilon }\right) $ is a least energy solution depends on the
relationship between $\liminf\limits_{|x|\rightarrow \infty }V_{l}(x)$,$%
\lim\limits_{|x|\rightarrow \infty }V_{l}(x)$ and $\lambda _{l},~l=1,2.$

Before we describe the main results, we need some known techniques. The
energy functional $J_{\varepsilon }$ we consider that corresponds to problem
$\left( P_{\varepsilon }\right) $ is given by, for each $\left( u,v\right)
\in H=H^{\alpha }\left( \mathbb{R}^{N}\right) \times H^{\alpha }\left(
\mathbb{R}^{N}\right) ,$%
\begin{eqnarray}
J_{\varepsilon }\left( u,v\right) &=&\frac{1}{2}\left\Vert \left( u,v\right)
\right\Vert _{H}^{2}-\frac{1}{2p+2}\left( \int_{\mathbb{R}^{N}}\mu
_{1}\left\vert u^{+}\right\vert ^{2p+2}\mathrm{d}x+\int_{\mathbb{R}^{N}}\mu
_{2}\left\vert v^{+}\right\vert ^{2p+2}\mathrm{d}x\right)  \notag \\
&&-\frac{1}{q+1}\int_{\mathbb{R}^{N}}\beta (x)\left\vert u^{+}\right\vert
^{q+1}\left\vert v^{+}\right\vert ^{q+1}\mathrm{d}x,  \label{1.14}
\end{eqnarray}%
where with $u^{+}=\max\{u,0\}$, $v^{+}=\max\{v,0\}$ and the norm $\Vert
\cdot \Vert _{H}, $ given by
\begin{eqnarray*}
\left\Vert \left( u,v\right) \right\Vert _{H}^{2} &=&\left\Vert u\right\Vert
_{V_{1}}^{2}+\left\Vert v\right\Vert _{V_{2}}^{2} \\
&=&\int_{\mathbb{R}^{N}}\left( \varepsilon ^{2\alpha }|(-\triangle )^{\frac{%
\alpha }{2}}u|^{2}+V_{1}(x)u^{2}\right) \mathrm{d}x+\int_{\mathbb{R}%
^{N}}\left( \varepsilon ^{2\alpha }|(-\triangle )^{\frac{\alpha }{2}%
}v|^{2}+V_{2}(x)v^{2}\right) \mathrm{d}x.
\end{eqnarray*}%
Note that the fractional Sobolev space $H^{\alpha }\left( \mathbb{R}%
^{N}\right) $ is given by:
\begin{equation*}
H^{\alpha }\left( \mathbb{R}^{N}\right) =\left\{ u\in L^{2}\left( \mathbb{R}%
^{N}\right) :\int_{\mathbb{R}^{N}}\int_{\mathbb{R}^{N}}\frac{|u(x)-u(y)|^{2}%
}{|x-y|^{N+2\alpha }}\mathrm{d}x\mathrm{d}y<+\infty \right\} .
\end{equation*}%
From Proposition 3.4, 3.6 of \cite{DPV}, we know that
\begin{equation*}
\int_{\mathbb{R}^{N}}\int_{\mathbb{R}^{N}}\frac{|u(x)-u(y)|^{2}}{%
|x-y|^{N+2\alpha }}\mathrm{d}x\mathrm{d}y=2C(N,\alpha )\int_{\mathbb{R}%
^{N}}|(-\triangle )^{\frac{\alpha }{2}}u|^{2}\mathrm{d}x,
\end{equation*}%
where
\begin{equation*}
C(N,\alpha )=\left( \int_{\mathbb{R}^{N}}\frac{1-cos\zeta _{1}}{|\zeta
|^{N+2\alpha }}\mathrm{d}\zeta \right) ,~~\zeta =\left( \zeta _{1},\zeta
_{2},\cdots ,\zeta _{N}\right) .
\end{equation*}%
It is well known that the energy functional $J_{\varepsilon }$ is of class $%
C^{1}$ in $H$ and the nonnegative solutions of problem $\left(
P_{\varepsilon }\right) $ are the critical points of the energy functional $%
J_{\varepsilon }.$ As the energy functional $J_{\varepsilon }$ is not
bounded below on $H$ and to prove the existence of nontrivial critical
points of $J_{\varepsilon }$, it is useful to consider the functional on the
Nehari manifold%
\begin{equation*}
\mathbf{N}_{\varepsilon }=\left\{ (u,v)\in \mathbf{T}:%
\begin{array}{l}
\left\Vert u\right\Vert _{V_{1}}^{2}=\mu _{1}\int_{\mathbb{R}^{N}}\left\vert
u^{+}\right\vert ^{2p+2}\mathrm{d}x+\int_{\mathbb{R}^{N}}\beta(x)\left\vert
u^{+}\right\vert ^{q+1}\left\vert v^{+}\right\vert ^{q+1}\mathrm{d}x, \\
\left\Vert v\right\Vert _{V_{2}}^{2}=\mu _{2}\int_{\mathbb{R}^{N}}\left\vert
v^{+}\right\vert ^{2p+2}\mathrm{d}x+\int_{\mathbb{R}^{N}}\beta(x)\left\vert
u^{+}\right\vert ^{q+1}\left\vert v^{+}\right\vert ^{q+1}\mathrm{d}x%
\end{array}%
\right\} ,
\end{equation*}%
where%
\begin{equation*}
\mathbf{T}=\left\{ \left( u,v\right) \in H:u\not\equiv 0\text{ and }%
v\not\equiv 0\right\} .
\end{equation*}%
Furthermore, we consider the minimization problem:
\begin{equation}
c_{\varepsilon }=\inf \left\{ J_{\varepsilon }\left( u,v\right) :\left(
u,v\right) \in \mathbf{N}_{\varepsilon }\right\} ,  \label{1.15}
\end{equation}%
we call the nontrivial critical point $(u,v)\in H$ of $J_{\varepsilon }$ is
a least energy solution of problem $(P_{\varepsilon })$ if $J_{\varepsilon
}(u,v)=c_{\varepsilon }$ and $(u,v)\in \mathbf{N}_{\varepsilon }$. Note that
if there exists a nontrivial solution $(u,v)\in \mathbf{N}_{\varepsilon }$
of problem $(P_{\varepsilon })$ such that $J_{\varepsilon
}(u,v)>c_{\varepsilon }$, then we call the solution $(u,v)$ is a higher
energy solution of problem $(P_{\varepsilon })$.

Now we state our main results.

\begin{theorem}
\label{t1.1}$\left( i\right) $ Assume that the conditions $\left(
D_{1}\right) -\left( D_{4}\right) $ hold. Then there exists $\varepsilon
_{0}>0$ such that for every $\varepsilon \in \left( 0,\varepsilon
_{0}\right) ,$ problem $(P_{\varepsilon })$ has at least $k\times \ell +m$
positive solutions.\newline
$\left( ii\right) $ Let $(\widehat{u}_{\varepsilon ,i},\widehat{v}%
_{\varepsilon ,j})$ be a positive
solution of problem $(P_{\varepsilon })$ as in part $\left( i\right) .$ Then there exist $z_{1,i}$ and $z_{2,j}$ which are isolated
gobal minimum points of potentials $V_{1}\left( x\right) $ and $V_{2}\left(
x\right) ,$ respectively such that $(\widehat{u}_{\varepsilon ,i},\widehat{v}%
_{\varepsilon ,j})$ concentrating at $(z_{1,i},z_{2,j})$ as $\varepsilon
\rightarrow 0.$
\end{theorem}

\begin{theorem}
\label{t1.2}$\left( i\right) $ If $\liminf\limits_{|x|\rightarrow \infty
}V_{l}(x)\equiv V_{l,\infty }>\lambda _{l}$ for all $l=1,2,$ then there
exists $0<\varepsilon _{\ast \ast }\leq \varepsilon _{0}$ such that for
every $\varepsilon <\varepsilon _{\ast \ast },$ we can find at least one
least energy solution in the these solutions of Theorem \ref{t1.1} $\left(
i\right) $.\newline
$\left( ii\right) $ If $\lim\limits_{|x|\rightarrow \infty }V_{l}(x)=\lambda
_{l}$ for all $l=1,2,$ then all of the solutions of Theorem \ref{t1.1} $%
\left( i\right) $ are higher energy.
\end{theorem}

\begin{remark}
Compared with the local operator $-\Delta $, the operator $(-\Delta
)^{\alpha }$ with $\alpha \in (0,1)$ on $\mathbb{R}^{N}$ is nonlocal, which
can be expressed as follows: the quantity $(-\Delta )^{\alpha }u(x)$ depend
on not only the values of $u$ in a neighborhood of $x$ (as is the case for
the Laplacian), but also the values of $u$ at any point $y\in \mathbb{R}^{N}$%
, and it is expected that the standard techniques for $-\Delta $ cannot be
used directly.
\end{remark}

\begin{remark}
The main difficulty when dealing with problem $(P_{\varepsilon })$ lies in
the lack of compactness of the embedding from $H^{\alpha }(\mathbb{R}^{N})$
into $L^{r}(\mathbb{R}^{N}),~2\leq r<2_{\alpha }^{\ast }$, which prevents us
from using the variational methods in a standard way. We solve this
difficulty by using the Nehari manifold technique and the energy estimates,
see Section 3 for details.
\end{remark}

\begin{notation}
Throughout this paper, we shall denote $|\cdot |_{r}$ the $L^{r}$-norm for $%
1\leq r\leq +\infty $ and $C$ various positive generic constants, which may
vary from line to line. $B^{N}(x,r)$ denotes a ball centered at $x$ with
radius $r$ in $\mathbb{R}^{N}$. Also if we take a subsequence of a sequence $%
\{(u_{n},v_{n})\}$ we shall denote it again by $\{(u_{n},v_{n})\}$. We use $%
o(1)$ to denote any quantity which tends to zero as $n\rightarrow \infty $
and $o_{\varepsilon }(1)$ to denote any quantity which tends to zero as $%
\varepsilon \rightarrow 0$.
\end{notation}

The remainder of this paper is as follows. In Section 2, we give some
preliminaries. In Section 3, we construct the Palais--Smale (PS) sequences.
In Section 4, we prove Theorem \ref{t1.1}. In Section 5, we prove Theorem %
\ref{t1.2}.

\section{Preliminaries}

First of all, it is easy to see that if we make the change of variables $%
x=\varepsilon z$, then we can rewrite Eq. $(P_{\varepsilon })$ as the
following equivalent equation:
\begin{equation}
\left\{
\begin{array}{ll}
(-\Delta )^{\alpha }u+V_{1}\left( \varepsilon x\right) u=\left( \mu
_{1}|u|^{2p}+\beta \left( \varepsilon x\right) |u|^{q-1}|v|^{q+1}\right) u,
& \text{in }\mathbb{R}^{N}, \\
(-\Delta )^{\alpha }v+V_{2}\left( \varepsilon x\right) v=\left( \mu
_{2}|v|^{2p}+\beta \left( \varepsilon x\right) |v|^{q-1}|u|^{q+1}\right) v,
& \text{in }\mathbb{R}^{N}\,, \\
u,v\in H^{\alpha }(\mathbb{R}^{N}),\quad u,v>0, & \text{in }\mathbb{R}^{N}.%
\end{array}%
\right.  \tag{$\widetilde{P}_{\varepsilon }$}
\end{equation}

Now we present some related results about Eq. $\left( \widehat{E}_{l}\right)
$. It is obvious that Eq. $\left( \widehat{E}_{l}\right) $ is variational,
and its solutions are the critical points of the functional $\widehat{I}%
_{l}\left( u\right) $ defined in $H^{\alpha }(\mathbb{R}^{N})$ as
\begin{equation*}
\widehat{I}_{l}\left( u\right) =\frac{1}{2}\int_{\mathbb{R}^{N}}\left(
|(-\triangle )^{\frac{\alpha }{2}}u|^{2}+\lambda _{l}u^{2}\right) \mathrm{d}%
x-\frac{\mu _{l}}{2p+2}\int_{\mathbb{R}^{N}}\left\vert u^{+}\right\vert
^{2p+2}\mathrm{d}x,~~l=1,2.
\end{equation*}%
Furthermore, one can see that $\widehat{I}_{l}$ is a $C^{1}$ functional with
the derivative given by%
\begin{equation*}
\left\langle \widehat{I}_{l}^{\prime }(u),\varphi \right\rangle =\int_{%
\mathbb{R}^{N}}\left( (-\triangle )^{\frac{\alpha }{2}}u(-\triangle )^{\frac{%
\alpha }{2}}v+\lambda _{l}u\varphi \right) \mathrm{d}x-\mu _{l}\int_{\mathbb{%
R}^{N}}\left\vert u^{+}\right\vert ^{2p}u^{+}\varphi \mathrm{d}x,~~l=1,2,
\end{equation*}%
for all $\varphi \in H^{\alpha }(\mathbb{R}^{N})$, where $\widehat{I}%
_{l}^{\prime }$ denotes the Fr\'{e}chet derivative of $\widehat{I}_{l}.$

Define the Nehari manifold
\begin{equation*}
\widehat{\mathbf{M}}_{l}:=\left\{ u\in H^{\alpha }(\mathbb{R}^{N})\backslash
\{0\}:\left\langle \widehat{I}_{l}^{\prime }\left( u\right) ,u\right\rangle
=0\right\} \text{ for }l=1,2.
\end{equation*}%
Then, by Frank, Lenzmann and Silvestre \cite{FLS}, we may assume that Eq. $%
\left( \widehat{E}_{l}\right) $ has a unique least energy positive solution $%
\widehat{\omega }_{l}$ (which up to translation) such that
\begin{equation*}
\widehat{\alpha }_{l}=\inf_{u\in \widehat{\mathbf{M}}_{l}}\widehat{I}%
_{l}\left( u\right) =\widehat{I}_{l}\left( \widehat{\omega }_{l}\right)
\text{ for }l=1,2.
\end{equation*}%
Furthermore, $\widehat{\omega }_{l}$ is radial, i.e., $\widehat{\omega }%
_{l}(x)=\widehat{\omega }_{l}(|x|)$.

To prove the main results, we will show some technical lemmas, whose proofs
follow with the same type of arguments found in \cite{LW}. However for the
readers' convenience we will write their proofs. we need the following
lemmas.

\begin{lemma}
\label{L2.1}Assume that $\varepsilon >0$ and $(u,v)\in \mathbf{N}%
_{\varepsilon }.$ Then
\begin{equation}
\mu _{1}\int_{\mathbb{R}^{N}}\left\vert u^{+}\right\vert ^{2p+2}\mathrm{d}%
x\geq \int_{\mathbb{R}^{N}}\left( \varepsilon ^{2\alpha }\left\vert
(-\triangle )^{\frac{\alpha }{2}}u\right\vert ^{2}+V_{1}(x)u^{2}\right)
\mathrm{d}x\geq \frac{2p+2}{p}\varepsilon ^{N}\widehat{\alpha }_{1}
\label{2.1}
\end{equation}%
and
\begin{equation}
\mu _{2}\int_{\mathbb{R}^{N}}\left\vert v^{+}\right\vert ^{2p+2}\mathrm{d}%
x\geq \int_{\mathbb{R}^{N}}\left( \varepsilon ^{2\alpha }\left\vert
(-\triangle )^{\frac{\alpha }{2}}v\right\vert ^{2}+V_{2}(x)v^{2}\right)
\mathrm{d}x\geq \frac{2p+2}{p}\varepsilon ^{N}\widehat{\alpha }_{2}.
\label{2.2}
\end{equation}%
Furthermore, if $\int_{\mathbb{R}^{N}}\beta(x) \left\vert u^{+}\right\vert
^{q+1}\left\vert v^{+}\right\vert ^{q+1}\mathrm{d}x<0,$ then all the
inequalities of $\left( \ref{2.1}\right) $ and $\left( \ref{2.2}\right) $
become strict, i.e.
\begin{equation*}
\mu _{1}\int_{\mathbb{R}^{N}}\left\vert u^{+}\right\vert ^{2p+2}\mathrm{d}%
x>\int_{\mathbb{R}^{N}}\left( \varepsilon ^{2\alpha }\left\vert (-\triangle
)^{\frac{\alpha }{2}}u\right\vert ^{2}+V_{1}(x)u^{2}\right) \mathrm{d}x>%
\frac{2p+2}{p}\varepsilon ^{N}\widehat{\alpha }_{1}
\end{equation*}%
and
\begin{equation*}
\mu _{2}\int_{\mathbb{R}^{N}}\left\vert v^{+}\right\vert ^{2p+2}\mathrm{d}%
x>\int_{\mathbb{R}^{N}}\left( \varepsilon ^{2\alpha }\left\vert (-\triangle
)^{\frac{\alpha }{2}}v\right\vert ^{2}+V_{2}(x)v^{2}\right) \mathrm{d}x>%
\frac{2p+2}{p}\varepsilon ^{N}\widehat{\alpha }_{2}.
\end{equation*}
\end{lemma}

\begin{proof}
Let $\left( u,v\right) \in\mathbf{\ N}_{\varepsilon }$. Set $u_{\varepsilon
}\left( x\right) =u\left( \varepsilon x\right) $ and $v_{\varepsilon }\left(
x\right) =v\left( \varepsilon x\right) $ for $x\in \mathbb{R}^{N}$. Due to $%
\beta(x) \leq 0$ in $\mathbb{R}^{N}$ and $\left( u,v\right) \in \mathbf{N}%
_{\varepsilon } $, it is obvious that
\begin{eqnarray}
\int_{\mathbb{R}^{N}}\left( \left\vert (-\triangle )^{\frac{\alpha }{2}%
}u_{\varepsilon }\right\vert ^{2}+\lambda _{1}u_{\varepsilon }^{2}\right)
\mathrm{d}x &\leq &\int_{\mathbb{R}^{N}}\left( \left\vert (-\triangle )^{%
\frac{\alpha }{2}}u_{\varepsilon }\right\vert ^{2}+V_{1}(\varepsilon
x)u_{\varepsilon }^{2}\right) \mathrm{d}x  \label{2.3} \\
&=&\mu _{1}\int_{\mathbb{R}^{N}}\left\vert u_{\varepsilon }^{+}\right\vert
^{2p+2}\mathrm{d}x+\int_{\mathbb{R}^{N}}\beta \left( \varepsilon x\right)
\left\vert u_{\varepsilon }^{+}\right\vert ^{q+1}\left\vert v_{\varepsilon
}^{+}\right\vert ^{q+1}\mathrm{d}x  \notag \\
&\leq &\mu _{1}\int_{\mathbb{R}^{N}}\left\vert u_{\varepsilon
}^{+}\right\vert ^{2p+2}\mathrm{d}x  \notag
\end{eqnarray}%
and
\begin{eqnarray}
\int_{\mathbb{R}^{N}}\left( \left\vert (-\triangle )^{\frac{\alpha }{2}%
}v_{\varepsilon }\right\vert ^{2}+\lambda _{2}v_{\varepsilon }^{2}\right)
\mathrm{d}x &\leq &\int_{\mathbb{R}^{N}}\left( \left\vert (-\triangle )^{%
\frac{\alpha }{2}}v_{\varepsilon }\right\vert ^{2}+V_{2}(\varepsilon
x)v_{\varepsilon }^{2}\right) \mathrm{d}x  \label{2.4} \\
&=&\mu _{2}\int_{\mathbb{R}^{N}}\left\vert v_{\varepsilon }^{+}\right\vert
^{2p+2}\mathrm{d}x+\int_{\mathbb{R}^{N}}\beta \left( \varepsilon x\right)
\left\vert u_{\varepsilon }^{+}\right\vert ^{q+1}\left\vert v_{\varepsilon
}^{+}\right\vert ^{q+1}\mathrm{d}x  \notag \\
&\leq &\mu _{2}\int_{\mathbb{R}^{N}}\left\vert v_{\varepsilon
}^{+}\right\vert ^{2p+2}\mathrm{d}x.  \notag
\end{eqnarray}%
Let
\begin{equation*}
s_{\varepsilon }=\left( \frac{\int_{\mathbb{R}^{N}}\left( \left\vert
(-\triangle )^{\frac{\alpha }{2}}u_{\varepsilon }\right\vert ^{2}+\lambda
_{1}u_{\varepsilon }^{2}\right) \mathrm{d}x}{\mu _{1}\int_{\mathbb{R}%
^{N}}\left\vert u_{\varepsilon }^{+}\right\vert ^{2p+2}\mathrm{d}x}\right)
^{1/2p}\text{and }~t_{\varepsilon }=\left( \frac{\int_{\mathbb{R}^{N}}\left(
\left\vert (-\triangle )^{\frac{\alpha }{2}}v_{\varepsilon }\right\vert
^{2}+\lambda _{2}v_{\varepsilon }^{2}\right) \mathrm{d}x}{\mu _{2}\int_{%
\mathbb{R}^{N}}\left\vert v_{\varepsilon }^{+}\right\vert ^{2p+2}\mathrm{d}x}%
\right) ^{1/2p}.
\end{equation*}%
Then by (\ref{2.3}) and (\ref{2.4}), it is easy to verify that $%
0<s_{\varepsilon },t_{\varepsilon }\leq 1$, $s_{\varepsilon }u_{\varepsilon
}\in \widehat{\mathbf{M}}_{1}$ and $t_{\varepsilon }v_{\varepsilon }\in
\widehat{\mathbf{M}}_{2}$. Hence,
\begin{equation}
\widehat{\alpha }_{1}\leq \widehat{I}_{1}\left( s_{\varepsilon
}u_{\varepsilon }\right) \leq \frac{p}{2p+2}\int_{\mathbb{R}^{N}}\left(
\left\vert (-\triangle )^{\frac{\alpha }{2}}u_{\varepsilon }\right\vert
^{2}+\lambda _{1}u_{\varepsilon }^{2}\right) \mathrm{d}x  \label{2.5}
\end{equation}%
and
\begin{equation}
\widehat{\alpha }_{2}\leq \widehat{I}_{2}\left( t_{\varepsilon
}v_{\varepsilon }\right) \leq \frac{p}{2p+2}\int_{\mathbb{R}^{N}}\left(
\left\vert (-\triangle )^{\frac{\alpha }{2}}v_{\varepsilon }\right\vert
^{2}+\lambda _{2}v_{\varepsilon }^{2}\right) \mathrm{d}x.  \label{2.6}
\end{equation}%
Thus by $\left( \ref{2.3}\right) -\left( \ref{2.6}\right) $, we obtain~ (\ref%
{2.1}) and (\ref{2.2}). Similarly, one may follow the above argument to show
that all the inequalities of $\left( \ref{2.1}\right) $ and (\ref{2.2})
become strict if $\int_{\mathbb{R}^{N}}\beta(x) \left\vert u^{+}\right\vert
^{q+1}\left\vert v^{+}\right\vert ^{q+1}\mathrm{d}x<0$. Therefore, we may
complete the proof.
\end{proof}

The integral $\int_{\mathbb{R}^{N}}\beta(x) \left\vert u^{+}\right\vert
^{q+1}\left\vert v^{+}\right\vert ^{q+1}\mathrm{d}x$ may play an important
role in the quantity of the energy functional $J_{\varepsilon }$. Here we
state the crucial energy estimates, given by

\begin{lemma}
\label{L2.2}Let $\varepsilon ,\sigma >0$ and $\left( u,v\right) \in \mathbf{N%
}_{\varepsilon }.$ If $\int_{\mathbb{R}^{N}}\beta (x)\left\vert
u^{+}\right\vert ^{q+1}\left\vert v^{+}\right\vert ^{q+1}\mathrm{d}x\leq
-\sigma \varepsilon ^{N},$ then
\begin{equation}
\varepsilon ^{-N}J_{\varepsilon }\left( u,v\right) >\widehat{\alpha }_{1}+%
\widehat{\alpha }_{2}+\overline{\delta },  \label{2.7}
\end{equation}%
where
\begin{equation}
\overline{\delta }=\frac{p\sigma }{p+1}\min \left\{ 1,\frac{1}{2}\min
\left\{ \frac{S_{1}^{p+1}}{\mu _{1}\left( \frac{2p+2}{p}\widehat{\alpha }%
_{1}+\sigma \right) ^{p}},\frac{S_{2}^{p+1}}{\mu _{2}\left( \frac{2p+2}{p}%
\widehat{\alpha }_{2}+\sigma \right) ^{p}}\right\} \right\} ,  \label{2.8}
\end{equation}%
and
\begin{equation}
S_{l}=\inf_{w\in H^{\alpha }\left( \mathbb{R}^{N}\right) \backslash \left\{
0\right\} }\frac{\int_{\mathbb{R}^{N}}(|(-\triangle )^{\frac{\alpha }{2}%
}w|^{2}+\lambda _{l}w^{2})\mathrm{d}x}{\left( \int_{\mathbb{R}%
^{N}}\left\vert w^{+}\right\vert ^{2p+2}\mathrm{d}x\right) ^{1/p+1}}>0,\quad
l=1,2.  \label{2.9}
\end{equation}%
Furthermore, if $0<\sigma <(2p+2)\sqrt{\widehat{\alpha }_{1}\widehat{\alpha }%
_{2}},$ then
\begin{eqnarray}
&&p^{2}\mu _{1}\mu _{2}\int_{\mathbb{R}^{N}}\left\vert u^{+}\right\vert
^{2p+2}\mathrm{d}x\int_{\mathbb{R}^{N}}\left\vert v^{+}\right\vert ^{2p+2}%
\mathrm{d}x-\left( \int_{\mathbb{R}^{N}}\beta \left( x\right) \left\vert
u^{+}\right\vert ^{q+1}\left\vert v^{+}\right\vert ^{q+1}\mathrm{d}x\right)
^{2}  \notag \\
&>&\varepsilon ^{2N}\left( 4(p+1)^{2}\widehat{\alpha }_{1}\widehat{\alpha }%
_{2}-\sigma ^{2}\right) >0  \label{2.10}
\end{eqnarray}%
for all $\left( u,v\right) \in \mathbf{N}_{\varepsilon }$ and $\varepsilon
^{-N}J_{\varepsilon }\left( u,v\right) \leq \widehat{\alpha }_{1}+\widehat{%
\alpha }_{2}+\overline{\delta }.$
\end{lemma}

\begin{proof}
Assume $\left( u,v\right) \in \mathbf{N}_{\varepsilon }$ and $\sigma >0$
such that
\begin{equation}
\int_{\mathbb{R}^{N}}\beta (x)\left\vert u^{+}\right\vert ^{q+1}\left\vert
v^{+}\right\vert ^{q+1}\mathrm{d}x\leq -\sigma \varepsilon ^{N}.
\label{2.11}
\end{equation}%
Let $\widetilde{u}\left( x\right) =u\left( \varepsilon x\right) $ and $%
\widetilde{v}\left( x\right) =v\left( \varepsilon x\right) .$ Then
\begin{equation*}
\varepsilon ^{-N}J_{\varepsilon }(u,v)=\widetilde{J}_{\varepsilon }\left(
\widetilde{u},\widetilde{v}\right) ,
\end{equation*}%
where
\begin{eqnarray}
\widetilde{J}_{\varepsilon }\left( \widetilde{u},\widetilde{v}\right) &=&%
\frac{1}{2}\int_{\mathbb{R}^{N}}\left( |(-\triangle )^{\frac{\alpha }{2}}%
\widetilde{u}|^{2}+V_{1}\left( \varepsilon x\right) \widetilde{u}^{2}\right)
\mathrm{d}x+\frac{1}{2}\int_{\mathbb{R}^{N}}\left( |(-\triangle )^{\frac{%
\alpha }{2}}\widetilde{v}|^{2}+V_{2}\left( \varepsilon x\right) \widetilde{v}%
^{2}\right) \mathrm{d}x  \notag \\
&&-\frac{1}{2p+2}\left( \int_{\mathbb{R}^{N}}\mu _{1}\left\vert \widetilde{u}%
^{+}\right\vert ^{2p+2}+\int_{\mathbb{R}^{N}}\mu _{2}\left\vert \widetilde{v}%
^{+}\right\vert ^{2p+2}\right) \mathrm{d}x  \label{2.12} \\
&&-\frac{1}{q+1}\int_{\mathbb{R}^{N}}\beta \left( \varepsilon x\right)
\left\vert \widetilde{u}^{+}\right\vert ^{q+1}\left\vert \widetilde{v}%
^{+}\right\vert ^{q+1}\mathrm{d}x.  \notag
\end{eqnarray}%
Hence, by (\ref{2.11}), (\ref{2.12}) and $\left( u,v\right) \in \mathbf{N}%
_{\varepsilon }$, we have
\begin{equation}
\int_{\mathbb{R}^{N}}\beta \left( \varepsilon x\right) \left\vert \widetilde{%
u}^{+}\right\vert ^{q+1}\left\vert \widetilde{v}^{+}\right\vert ^{q+1}%
\mathrm{d}x\leq -\sigma ,  \label{2.13}
\end{equation}%
\begin{eqnarray}
\int_{\mathbb{R}^{N}}\left( \left\vert (-\triangle )^{\frac{\alpha }{2}}%
\widetilde{u}\right\vert ^{2}+\lambda _{1}\widetilde{u}^{2}\right) \mathrm{d}%
x &\leq &\int_{\mathbb{R}^{N}}\left( \left\vert (-\triangle )^{\frac{\alpha
}{2}}\widetilde{u}\right\vert ^{2}+V_{1}\left( \varepsilon x\right)
\widetilde{u}^{2}\right) \mathrm{d}x  \notag \\
&=&\mu _{1}\int_{\mathbb{R}^{N}}\left\vert \widetilde{u}^{+}\right\vert
^{2p+2}\mathrm{d}x+\int_{\mathbb{R}^{N}}\beta \left( \varepsilon x\right)
\left\vert \widetilde{u}^{+}\right\vert ^{q+1}\left\vert \widetilde{v}%
^{+}\right\vert ^{q+1}\mathrm{d}x,  \label{2.14}
\end{eqnarray}%
and
\begin{eqnarray}
\int_{\mathbb{R}^{N}}\left( \left\vert (-\triangle )^{\frac{\alpha }{2}}%
\widetilde{v}\right\vert ^{2}+\lambda _{2}\widetilde{v}^{2}\right) \mathrm{d}%
x &\leq &\int_{\mathbb{R}^{N}}\left( \left\vert (-\triangle )^{\frac{\alpha
}{2}}\widetilde{v}\right\vert ^{2}+V_{2}\left( \varepsilon x\right)
\widetilde{v}^{2}\right) \mathrm{d}x  \notag \\
&=&\mu _{2}\int_{\mathbb{R}^{N}}\left\vert \widetilde{v}^{+}\right\vert
^{2p+2}\mathrm{d}x+\int_{\mathbb{R}^{N}}\beta \left( \varepsilon x\right)
\left\vert \widetilde{u}^{+}\right\vert ^{q+1}\left\vert \widetilde{v}%
^{+}\right\vert ^{q+1}\mathrm{d}x.  \label{2.15}
\end{eqnarray}%
Moreover, it follows from (\ref{2.12}), (\ref{2.14}) and (\ref{2.15}) that%
\begin{eqnarray}
\widetilde{J}_{\varepsilon }\left( \widetilde{u},\widetilde{v}\right) &\geq &%
\frac{1}{2}\int_{\mathbb{R}^{N}}\left( |(-\triangle )^{\frac{\alpha }{2}}%
\widetilde{u}|^{2}+V_{1}\left( \varepsilon x\right) \widetilde{u}^{2}\right)
\mathrm{d}x+\frac{1}{2}\int_{\mathbb{R}^{N}}\left( |(-\triangle )^{\frac{%
\alpha }{2}}\widetilde{v}|^{2}+V_{2}\left( \varepsilon x\right) \widetilde{v}%
^{2}\right) \mathrm{d}x  \notag \\
&&-\frac{1}{2p+2}\int_{\mathbb{R}^{N}}\mu _{1}\left\vert \widetilde{u}%
^{+}\right\vert ^{2p+2}-\frac{1}{2p+2}\int_{\mathbb{R}^{N}}\mu
_{2}\left\vert \widetilde{v}^{+}\right\vert ^{2p+2}\mathrm{d}x  \notag \\
&&-\frac{1}{p+1}\int_{\mathbb{R}^{N}}\beta \left( \varepsilon x\right)
\left\vert \widetilde{u}^{+}\right\vert ^{q+1}\left\vert \widetilde{v}%
^{+}\right\vert ^{q+1}\mathrm{d}x  \label{2.16} \\
&=&\frac{p}{2p+2}\left[ \int_{\mathbb{R}^{N}}\left( \left\vert (-\triangle
)^{\frac{\alpha }{2}}\widetilde{u}\right\vert ^{2}+V_{1}\left( \varepsilon
x\right) \widetilde{u}^{2}\right) \mathrm{d}x+\int_{\mathbb{R}^{N}}\left(
\left\vert (-\triangle )^{\frac{\alpha }{2}}\widetilde{v}\right\vert
^{2}+V_{2}\left( \varepsilon x\right) \widetilde{v}^{2}\right) \mathrm{d}x%
\right] .  \notag
\end{eqnarray}%
Set
\begin{equation*}
\widetilde{s}=\left( \frac{\int_{\mathbb{R}^{N}}\left( |(-\triangle )^{\frac{%
\alpha }{2}}\widetilde{u}|^{2}+\lambda _{1}\widetilde{u}^{2}\right) \mathrm{d%
}x}{\mu _{1}\int_{\mathbb{R}^{N}}\left\vert \widetilde{u}^{+}\right\vert
^{2p+2}\mathrm{d}x}\right) ^{1/2p}\text{and}\quad \widetilde{t}=\left( \frac{%
\int_{\mathbb{R}^{N}}\left( |(-\triangle )^{\frac{\alpha }{2}}\widetilde{v}%
|^{2}+\lambda _{2}\widetilde{v}^{2}\right) \mathrm{d}x}{\mu _{2}\int_{%
\mathbb{R}^{N}}\left\vert \widetilde{v}^{+}\right\vert ^{2p+2}\mathrm{d}x}%
\right) ^{1/2p}.
\end{equation*}%
Then, by $\left( \ref{2.13}\right) -\left( \ref{2.15}\right) $, we have
\begin{equation}
\widetilde{s}^{2p}=\frac{\int_{\mathbb{R}^{N}}\left( |(-\triangle )^{\frac{%
\alpha }{2}}\widetilde{u}|^{2}+\lambda _{1}\widetilde{u}^{2}\right) \mathrm{d%
}x}{\mu _{1}\int_{\mathbb{R}^{N}}\left\vert \widetilde{u}^{+}\right\vert
^{2p+2}\mathrm{d}x}\leq 1+\frac{\int_{\mathbb{R}^{N}}\beta \left(
\varepsilon x\right) \left\vert \widetilde{u}^{+}\right\vert
^{q+1}\left\vert \widetilde{v}^{+}\right\vert ^{q+1}\mathrm{d}x}{\mu
_{1}\int_{\mathbb{R}^{N}}\left\vert \widetilde{u}^{+}\right\vert ^{2p+2}%
\mathrm{d}x}\leq 1-\frac{\sigma }{\mu _{1}\int_{\mathbb{R}^{N}}\left\vert
\widetilde{u}^{+}\right\vert ^{2p+2}\mathrm{d}x}  \label{2.17}
\end{equation}%
and
\begin{equation}
\widetilde{t}^{2p}=\frac{\int_{\mathbb{R}^{N}}\left( |(-\triangle )^{\frac{%
\alpha }{2}}\widetilde{v}|^{2}+\lambda _{2}\widetilde{v}^{2}\right) \mathrm{d%
}x}{\mu _{2}\int_{\mathbb{R}^{N}}\left\vert \widetilde{v}^{+}\right\vert
^{2p+2}\mathrm{d}x}\leq 1+\frac{\int_{\mathbb{R}^{N}}\beta \left(
\varepsilon x\right) \left\vert \widetilde{u}^{+}\right\vert
^{q+1}\left\vert \widetilde{v}^{+}\right\vert ^{q+1}\mathrm{d}x}{\mu
_{2}\int_{\mathbb{R}^{N}}\left\vert \widetilde{v}^{+}\right\vert ^{2p+2}%
\mathrm{d}x}\leq 1-\frac{\sigma }{\mu _{2}\int_{\mathbb{R}^{N}}\left\vert
\widetilde{v}^{+}\right\vert ^{2p+2}\mathrm{d}x}.  \label{2.18}
\end{equation}%
Besides, $\widetilde{s}\widetilde{u}\in \widehat{\mathbf{M}}_{1}$ and $%
\widetilde{t}\widetilde{v}\in \widehat{\mathbf{M}}_{2}$ for all $\varepsilon
>0$, $0<q\leq p\leq 1$ if $N\leq 4\alpha $ and $0<q\leq p\leq \frac{2\alpha
}{N-2\alpha }$ if $N>4\alpha $. It follows from (\ref{2.16}), (\ref{2.17})
and (\ref{2.18}) that
\begin{eqnarray}
\widetilde{J}_{\varepsilon }\left( \widetilde{u},\widetilde{v}\right) &\geq &%
\frac{p}{2p+2}\left[ \int_{\mathbb{R}^{N}}\left( \left\vert (-\triangle )^{%
\frac{\alpha }{2}}\widetilde{u}\right\vert ^{2}+\lambda _{1}\widetilde{u}%
^{2}\right) \mathrm{d}x+\int_{\mathbb{R}^{N}}\left( \left\vert (-\triangle
)^{\frac{\alpha }{2}}\widetilde{v}\right\vert ^{2}+\lambda _{2}\widetilde{v}%
^{2}\right) \mathrm{d}x\right]  \notag \\
&\geq &\frac{p}{2p+2}\left[ \widetilde{s}^{2p}\int_{\mathbb{R}^{N}}\left(
\left\vert (-\triangle )^{\frac{\alpha }{2}}\widetilde{u}\right\vert
^{2}+\lambda _{1}\widetilde{u}^{2}\right) \mathrm{d}x+\widetilde{t}%
^{2p}\int_{\mathbb{R}^{N}}\left( \left\vert (-\triangle )^{\frac{\alpha }{2}}%
\widetilde{v}\right\vert ^{2}+\lambda _{2}\widetilde{v}^{2}\right) \mathrm{d}%
x\right]  \notag \\
&&+\frac{p\sigma }{2p+2}\left[ \frac{\int_{\mathbb{R}^{N}}\left(
|(-\triangle )^{\frac{\alpha }{2}}\widetilde{u}|^{2}+\lambda _{1}\widetilde{u%
}^{2}\right) \mathrm{d}x}{\mu _{1}\int_{\mathbb{R}^{N}}\left\vert \widetilde{%
u}^{+}\right\vert ^{2p+2}\mathrm{d}x}+\frac{\int_{\mathbb{R}^{N}}\left(
|(-\triangle )^{\frac{\alpha }{2}}\widetilde{v}|^{2}+\lambda _{2}\widetilde{v%
}^{2}\right) \mathrm{d}x}{\mu _{2}\int_{\mathbb{R}^{N}}\left\vert \widetilde{%
v}^{+}\right\vert ^{2p+2}\mathrm{d}x}\right]  \label{2.19}
\end{eqnarray}%
\begin{eqnarray}
&\geq &\frac{p\sigma }{2p+2}\left[ \frac{\int_{\mathbb{R}^{N}}\left(
|(-\triangle )^{\frac{\alpha }{2}}\widetilde{u}|^{2}+\lambda _{1}\widetilde{u%
}^{2}\right) \mathrm{d}x}{\mu _{1}\int_{\mathbb{R}^{N}}\left\vert \widetilde{%
u}^{+}\right\vert ^{2p+2}\mathrm{d}x}+\frac{\int_{\mathbb{R}^{N}}\left(
|(-\triangle )^{\frac{\alpha }{2}}\widetilde{v}|^{2}+\lambda _{2}\widetilde{v%
}^{2}\right) \mathrm{d}x}{\mu _{2}\int_{\mathbb{R}^{N}}\left\vert \widetilde{%
v}^{+}\right\vert ^{2p+2}\mathrm{d}x}\right]  \notag \\
&&+\widehat{\alpha }_{1}+\widehat{\alpha }_{2}.  \notag
\end{eqnarray}%
Now we wish to claim (\ref{2.7}). Suppose both
\begin{equation*}
\left( \frac{2p+2}{p}\widehat{\alpha }_{1}+\sigma \right) ^{p}\int_{\mathbb{R%
}^{N}}\left( \left\vert (-\triangle )^{\frac{\alpha }{2}}\widetilde{u}%
\right\vert ^{2}+\lambda _{1}\widetilde{u}^{2}\right) \mathrm{d}%
x<S_{1}^{p+1}\int_{\mathbb{R}^{N}}\left\vert \widetilde{u}^{+}\right\vert
^{2p+2}\mathrm{d}x
\end{equation*}%
and
\begin{equation*}
\left( \frac{2p+2}{p}\widehat{\alpha }_{2}+\sigma \right) ^{p}\int_{\mathbb{R%
}^{N}}\left( \left\vert (-\triangle )^{\frac{\alpha }{2}}\widetilde{v}%
\right\vert ^{2}+\lambda _{2}\widetilde{v}^{2}\right) \mathrm{d}%
x<S_{2}^{p+1}\int_{\mathbb{R}^{N}}\left\vert \widetilde{v}^{+}\right\vert
^{2p+2}\mathrm{d}x.
\end{equation*}%
Then by (\ref{2.9}), we have
\begin{eqnarray*}
&&\left( \frac{2p+2}{p}\widehat{\alpha }_{1}+\sigma \right) ^{p}\int_{%
\mathbb{R}^{N}}\left( \left\vert (-\triangle )^{\frac{\alpha }{2}}\widetilde{%
u}\right\vert ^{2}+\lambda _{1}\widetilde{u}^{2}\right) \mathrm{d}%
x<S_{1}^{p+1}\int_{\mathbb{R}^{N}}\left\vert \widetilde{u}^{+}\right\vert
^{2p+2}\mathrm{d}x \\
&\leq &\left( \int_{\mathbb{R}^{N}}\left( \left\vert (-\triangle )^{\frac{%
\alpha }{2}}\widetilde{u}\right\vert ^{2}+\lambda _{1}\widetilde{u}%
^{2}\right) \mathrm{d}x\right) ^{p+1}
\end{eqnarray*}%
and
\begin{eqnarray*}
&&\left( \frac{2p+2}{p}\widehat{\alpha }_{2}+\sigma \right) ^{p}\int_{%
\mathbb{R}^{N}}\left( \left\vert (-\triangle )^{\frac{\alpha }{2}}\widetilde{%
v}\right\vert ^{2}+\lambda _{2}\widetilde{v}^{2}\right) \mathrm{d}%
x<S_{2}^{p+1}\int_{\mathbb{R}^{N}}\left\vert \widetilde{v}^{+}\right\vert
^{2p+2}\mathrm{d}x \\
&\leq &\left( \int_{\mathbb{R}^{N}}\left( \left\vert (-\triangle )^{\frac{%
\alpha }{2}}\widetilde{v}\right\vert ^{2}+\lambda _{2}\widetilde{v}%
^{2}\right) \mathrm{d}x\right) ^{p+1}.
\end{eqnarray*}%
Hence,
\begin{equation}
\int_{\mathbb{R}^{N}}\left( \left\vert (-\triangle )^{\frac{\alpha }{2}}%
\widetilde{u}\right\vert ^{2}+\lambda _{1}\widetilde{u}^{2}\right) \mathrm{d}%
x>\frac{2p+2}{p}\widehat{\alpha }_{1}+\sigma  \label{2.20}
\end{equation}%
and
\begin{equation}
\int_{\mathbb{R}^{N}}\left( \left\vert (-\triangle )^{\frac{\alpha }{2}}%
\widetilde{u}\right\vert ^{2}+\lambda _{2}\widetilde{v}^{2}\right) \mathrm{d}%
x>\frac{2p+2}{p}\widehat{\alpha }_{2}+\sigma .  \label{2.21}
\end{equation}%
Combining (\ref{2.16}), (\ref{2.20}) and (\ref{2.21}), we obtain
\begin{equation*}
\widetilde{J}_{\varepsilon }\left( \widetilde{u},\widetilde{v}\right) >%
\widehat{\alpha }_{1}+\widehat{\alpha }_{2}+\frac{p\sigma }{p+1},
\end{equation*}%
thus (\ref{2.7}) holds. On the other hand, suppose either
\begin{equation*}
\left( \frac{2p+2}{p}\widehat{\alpha }_{1}+\sigma \right) ^{p}\int_{\mathbb{R%
}^{N}}\left( \left\vert (-\triangle )^{\frac{\alpha }{2}}\widetilde{u}%
\right\vert ^{2}+\lambda _{1}\widetilde{u}^{2}\right) \mathrm{d}x\geq
S_{1}^{p+1}\int_{\mathbb{R}^{N}}\left\vert \widetilde{u}^{+}\right\vert
^{2p+2}\mathrm{d}x
\end{equation*}%
or
\begin{equation*}
\left( \frac{2p+2}{p}\widehat{\alpha }_{1}+\sigma \right) ^{p}\int_{\mathbb{R%
}^{N}}\left( \left\vert (-\triangle )^{\frac{\alpha }{2}}\widetilde{v}%
\right\vert ^{2}+\lambda _{2}\widetilde{v}^{2}\right) \mathrm{d}x\geq
S_{2}^{p+1}\int_{\mathbb{R}^{N}}\left\vert \widetilde{v}^{+}\right\vert
^{2p+2}\mathrm{d}x.
\end{equation*}%
Then by (\ref{2.19}), we obtain
\begin{eqnarray*}
\widetilde{J}_{\varepsilon }\left( \widetilde{u},\widetilde{v}\right) &\geq &%
\frac{p\sigma }{2p+2}\left[ \frac{\int_{\mathbb{R}^{N}}\left( |(-\triangle
)^{\frac{\alpha }{2}}\widetilde{u}|^{2}+\lambda _{1}\widetilde{u}^{2}\right)
\mathrm{d}x}{\mu _{1}\int_{\mathbb{R}^{N}}\left\vert \widetilde{u}%
^{+}\right\vert ^{2p+2}\mathrm{d}x}+\frac{\int_{\mathbb{R}^{N}}\left(
|(-\triangle )^{\frac{\alpha }{2}}\widetilde{v}|^{2}+\lambda _{2}\widetilde{v%
}^{2}\right) \mathrm{d}x}{\mu _{2}\int_{\mathbb{R}^{N}}\left\vert \widetilde{%
v}^{+}\right\vert ^{2p+2}\mathrm{d}x}\right] \\
&&+\widehat{\alpha }_{1}+\widehat{\alpha }_{2} \\
&>&\frac{p\sigma }{2p+2}\min \left\{ \frac{S_{1}^{p+1}}{\mu _{1}\left( \frac{%
2p+2}{p}\widehat{\alpha }_{1}+\sigma \right) ^{p}},\frac{S_{2}^{p+1}}{\mu
_{2}\left( \frac{2p+2}{p}\widehat{\alpha }_{2}+\sigma \right) ^{p}}\right\} +%
\widehat{\alpha }_{1}+\widehat{\alpha }_{2},
\end{eqnarray*}%
i.e.
\begin{equation}
\widetilde{J}_{\varepsilon }\left( \widetilde{u},\widetilde{v}\right) >%
\widehat{\alpha }_{1}+\widehat{\alpha }_{2}+\overline{\delta },  \label{2.22}
\end{equation}%
where $\overline{\delta }$ is defined in (\ref{2.8}). Therefore, by (\ref%
{2.22}), we may complete the proof of (\ref{2.7}). Since (\ref{2.7}) holds
for all $\left( u,v\right) \in \mathbf{N}_{\varepsilon }$ satisfying (\ref%
{2.13}), then we may conclude that
\begin{equation}
-\varepsilon ^{N}\sigma <\int_{\mathbb{R}^{N}}\beta (x)\left\vert
u^{+}\right\vert ^{q+1}\left\vert v^{+}\right\vert ^{q+1}\mathrm{d}x\leq 0,
\label{2.23}
\end{equation}%
for all $\left( u,v\right) \in \mathbf{N}_{\varepsilon }$ with $\varepsilon
^{-N}J_{\varepsilon }\left( u,v\right) \leq \widehat{\alpha }_{1}+\widehat{%
\alpha }_{2}+\overline{\delta }.$ Hence, by (\ref{2.23}) and Lemma~\ref{2.1}%
, we may obtain (\ref{2.10}) for all $\left( u,v\right) \in \mathbf{N}%
_{\varepsilon }$ with $\varepsilon ^{-N}J_{\varepsilon }\left( u,v\right)
\leq \widehat{\alpha }_{1}+\widehat{\alpha }_{2}+\overline{\delta }$,
provided $0<\sigma <(2p+2)\sqrt{\widehat{\alpha }_{1}\widehat{\alpha }_{2}}$%
. This completes the proof.
\end{proof}

To prove Theorems \ref{t1.1} and \ref{t1.2}, we need another lemma as
follows:

\begin{lemma}
\label{L2.3} Let $c_{\varepsilon }$ be as in (\ref{1.15}). Then $\varepsilon
^{-N}c_{\varepsilon }\geq \widehat{\alpha }_{1}+\widehat{\alpha }_{2}$ for $%
\varepsilon >0$ that is sufficiently small. Furthermore, if $\varepsilon
^{-N}c_{\varepsilon }=\widehat{\alpha }_{1}+\widehat{\alpha }_{2},$ then
problem $(P_{\varepsilon })$ does not have a least energy solution.
\end{lemma}

\begin{proof}
By $q\leq p$ and Lemma~\ref{L2.1}, we have
\begin{equation*}
J_{\varepsilon }\left( u,v\right) \geq \frac{p}{2p+2}\left\Vert \left(
u,v\right) \right\Vert _{H}^{2}\geq \varepsilon ^{N}\left( \widehat{\alpha }%
_{1}+\widehat{\alpha }_{2}\right) \text{ for all }\left( u,v\right) \in
\mathbf{N}_{\varepsilon }\text{ and }\varepsilon >0.
\end{equation*}%
Hence, $\varepsilon ^{-N}c_{\varepsilon }\geq \widehat{\alpha }_{1}+\widehat{%
\alpha }_{2}$ for all $\varepsilon >0$. Now we want to prove that problem $%
\left( P_{\varepsilon }\right) $ does not have a least energy solution if $%
\varepsilon ^{-N}c_{\varepsilon }=\widehat{\alpha }_{1}+\widehat{\alpha }%
_{2} $. Argue by contradiction that there exists $\left( u_{0},v_{0}\right)
\in \mathbf{N}_{\varepsilon }$ that is a least energy solution of problem $%
\left( P_{\varepsilon }\right) $ such that
\begin{equation}
J_{\varepsilon }\left( u_{0},v_{0}\right) =c_{\varepsilon }=\varepsilon
^{N}\left( \widehat{\alpha }_{1}+\widehat{\alpha }_{2}\right) .  \label{2.24}
\end{equation}%
Hence by the maximum principle for the fractional Laplacian \cite{Sil}, $%
u_{0},v_{0}>0$ in $\mathbb{R}^{N}$. So
\begin{equation*}
\int_{\mathbb{R}^{N}}\beta(x) \left\vert u_{0}\right\vert ^{q+1}\left\vert
v_{0}\right\vert ^{q+1}\mathrm{d}x<0,
\end{equation*}%
and then by Lemma \ref{L2.1} and $q\leq p$, we obtain
\begin{equation*}
J_{\varepsilon }\left( u_{0},v_{0}\right) \geq \frac{p}{2p+2}\left\Vert
\left( u_{0},v_{0}\right) \right\Vert _{H}^{2}>\varepsilon ^{N}\left(
\widehat{\alpha }_{1}+\widehat{\alpha }_{2}\right)
\end{equation*}%
which would contradict (\ref{2.24}). Therefore, we may complete the proof.
\end{proof}

Now, we need to introduce a generalized barycenter map. By this we mean a
continuous map $\Phi :L^{2}\left( \mathbb{R}^{N}\right) \backslash \left\{
0\right\} \rightarrow \mathbb{R}^{N}$, which is equivariant with respect to
the action of the group of euclidian motions in $\mathbb{R}^{N}$, that is,
for every $x\in \mathbb{R}^{N}$, every orthogonal $N\times N$ matrix $A$ and
every $u\in L^{2}\left( \mathbb{R}^{N}\right) \backslash \left\{ 0\right\} $%
, one has
\begin{equation}
\Phi \left( \xi \ast u\right) =\xi +\Phi \left( u\right) \text{ and }\Phi
\left( u\circ A^{-1}\right) =A\Phi \left( u\right)  \label{1.11}
\end{equation}%
and
\begin{equation}
\Phi \left( u\circ \varepsilon \right) =\varepsilon ^{-1}\Phi \left(
u\right) \text{ for all }u\in L^{2}\left( \mathbb{R}^{N}\right) \backslash
\left\{ 0\right\} \text{ and }\varepsilon >0,  \label{1.12}
\end{equation}%
where $(\xi \ast u)(x)=u(x-\xi )$ and $(u\circ \varepsilon )(x)=u\left(
\varepsilon x\right) $. This property is easily built into the construction.
Indeed, if $\Phi _{1}$ satisfies (\ref{1.11}), then $\Phi $ defined by $%
|u|_{2}^{\frac{2}{N}}\Phi _{1}(u\circ |u|_{2}^{\frac{2}{N}})$ satisfies (\ref%
{1.11}) and (\ref{1.12}). Note that the map
\begin{equation*}
u\mapsto \frac{\int_{\mathbb{R}^{N}}xu^{2}\mathrm{d}x}{\int_{\mathbb{R}%
^{N}}u^{2}\mathrm{d}x}
\end{equation*}%
has the invariance properties (\ref{1.11}) and (\ref{1.12}), but it is
neither well defined on $L^{2}(\mathbb{R}^{N})\backslash \{0\}$ nor on $%
H^{\alpha }(\mathbb{R}^{N})\backslash \{0\}$.

The conditions $\left( D_{1}\right) $ and $\left( D_{2}\right) $ may imply $%
z_{1,i}^{\varepsilon }\rightarrow P_{1,i}$ and $z_{2,j}^{\varepsilon
}\rightarrow P_{2,j}$ as $\varepsilon \rightarrow 0$ (up to a subsequence),
where $P_{1,i}$ and $P_{2,j}$ are global minimum points of $V_{1}\left(
x\right) $ and $V_{2}\left( x\right) $, respectively. Generically, $%
(P_{1,i},P_{2,j})$ may not be equal to $(z_{1,i},z_{2,j}),$ because $%
V_{1}\left( x\right) $ and $V_{2}\left( x\right) $ may have multiple minimum
points. To find the positive solutions $(\widehat{u}_{\varepsilon ,i},%
\widehat{v}_{\varepsilon ,j})$ concentrating at $(z_{1,i},z_{2,j})$, we may
consider the minimization problem of $J_{\varepsilon }$ over the subset $%
N_{i,j}(\varepsilon )$ of $\mathbf{N}_{\varepsilon }$, where $%
N_{i,j_{2}}\left( \varepsilon \right) $ is defined by
\begin{equation*}
N_{i,j}\left( \varepsilon \right) =\left\{ \left( u,v\right) \in \mathbf{N}%
_{\varepsilon }:\Phi \left( u\right) \in C_{s}\left( z_{1,i}\right) \text{
and }\Phi \left( v\right) \in C_{s}\left( z_{2,j}\right) \right\} ,
\end{equation*}%
where $C_{s}\left( x\right) $ is a cube defined by $C_{s}\left( x\right) =%
\underset{n=1}{\overset{N}{\Pi }}\left( x_{n}-s,x_{n}+s\right) $ with the
boundary $\partial C_{s}\left( x\right) $ for $0<s<r_{0}$ ($r_{0}$ is given
in $(D_{2})$ and $\left( D_{4}\right) $), and $x=\left( x_{1},\ldots
,x_{N}\right) \in \mathbb{R}^{N}$ such that $\overline{C_{s}\left(
z_{1,i}\right) }\subset B_{r_{0}}(z_{1,i}),\overline{C_{s}\left(
z_{2,j}\right) }\subset B_{r_{0}}(z_{2,j})$ and
\begin{eqnarray*}
V_{1}\left( x\right) &>&V_{1}\left( z_{1,i}\right) \text{ for all }x\in
\partial C_{s}\left( z_{1,i}\right) \text{ and }i=1,2,\ldots ,k, \\
V_{2}\left( x\right) &>&V_{2}\left( z_{2,j}\right) \text{ for all }x\in
\partial C_{s}\left( z_{2,j}\right) \text{ and }j=1,2,\ldots ,\ell .
\end{eqnarray*}%
Moreover, by condition $\left( D_{4}\right) ,$%
\begin{equation}
\beta \left( x\right) \leq -c_{0}\text{ for all }x\in \overline{C_{s}\left(
z_{1,i}\right) }\text{ for all }1\leq i\leq m.  \label{2.34}
\end{equation}%
Next, we consider the boundary of $N_{i,j}\left( \varepsilon \right) $ as
follows:%
\begin{equation*}
O_{i,j}\left( \varepsilon \right) =\left\{ \left( u,v\right) \in \mathbf{N}%
_{\varepsilon }:\Phi \left( u\right) \in \partial C_{s}\left( z_{1,i}\right)
\text{ or }\Phi \left( v\right) \in \partial C_{s}\left( z_{2,j}\right)
\right\} .
\end{equation*}%
Now we consider the minimization of the functional $J_{\varepsilon }$ over $%
N_{i,j}\left( \varepsilon \right) $ and $O_{i,j}\left( \varepsilon \right) $%
, respectively, and denote the corresponding minima as
\begin{equation}
\gamma _{i,j}\left( \varepsilon \right) =\inf_{\left( u,v\right) \in
N_{i,j}\left( \varepsilon \right) }J_{\varepsilon }\left( u,v\right) \text{
and }\widetilde{\gamma }_{i,j}\left( \varepsilon \right) =\inf_{\left(
u,v\right) \in O_{i,j}\left( \varepsilon \right) }J_{\varepsilon }\left(
u,v\right) .  \label{2.25}
\end{equation}%
The upper bound of $\gamma _{i,j}\left( \varepsilon \right) $ is given by:

\begin{lemma}
\label{L2.4}For each $\delta >0$, there exists $\varepsilon _{\delta }>0$
such that for $\varepsilon \in \left( 0,\varepsilon _{\delta }\right) $,
there holds
\begin{equation*}
\varepsilon ^{-N}\gamma _{i,j}\left( \varepsilon \right) <\widehat{\alpha }%
_{1}+\widehat{\alpha }_{2}+\delta \text{ for all }1\leq i\leq k\text{ and }%
1\leq j\leq \ell.
\end{equation*}
\end{lemma}

\begin{proof}
First, we define test functions $u_{\varepsilon ,i}$ and $v_{\varepsilon ,j}$
by
\begin{eqnarray}
&&\left( u_{\varepsilon ,i}\left( x\right) ,v_{\varepsilon ,j}\left(
x\right) \right)  \label{2.26} \\
&=&\left( \widehat{\omega }_{1}\left( \frac{x-z_{1,i}+x^{\varepsilon }}{%
\varepsilon }\right) \psi _{\varepsilon }\left( \frac{x-z_{1,i}+x^{%
\varepsilon }}{\varepsilon }\right) ,\widehat{\omega }_{2}\left( \frac{%
x-z_{2,j}-x^{\varepsilon }}{\varepsilon }\right) \psi _{\varepsilon }\left(
\frac{x-z_{2,j}-x^{\varepsilon }}{\varepsilon }\right) \right) \,,  \notag
\end{eqnarray}%
where $x^{\varepsilon }=\frac{\sqrt{\varepsilon }}{2}e$, $e\in
S^{N-1}=\left\{ x\in \mathbb{R}^{N}:\left\vert x\right\vert =1\right\} $ and
$\widehat{\omega }_{l}$ is the unique positive radial solution of Eq. $%
\left( \widehat{E}_{l}\right) $ for $l=1,2.$ Notice that $\widehat{I}%
_{l}\left( \omega _{l}\right) =\widehat{\alpha }_{l}$ for $l=1,2.$ Moreover,
for $0<\varepsilon <1$, the function $\psi _{\varepsilon }\in C^{1}\left(
\mathbb{R}^{N},\left[ 0,1\right] \right) $ with compact support satisfies
\begin{equation*}
\psi _{\varepsilon }\left( x\right) =\left\{
\begin{array}{ll}
1, & \left\vert x\right\vert <\frac{1}{3\sqrt{\varepsilon }}-1, \\
0, & \left\vert x\right\vert >\frac{1}{3\sqrt{\varepsilon }},%
\end{array}%
\right.
\end{equation*}%
and $\left\vert \nabla \psi _{\varepsilon }\right\vert \leq 2$ in $\mathbb{R}%
^{N}$. Obviously,
\begin{equation}
\int_{\mathbb{R}^{N}}\beta(x) \left\vert u_{\varepsilon ,i}^{+}\right\vert
^{q+1}\left\vert v_{\varepsilon ,j}^{+}\right\vert ^{q+1}\mathrm{d}x=0\quad
\text{for }\varepsilon \text{ sufficiently small.}  \label{2.27}
\end{equation}%
Then by (\ref{2.26}), (\ref{2.27}) and Lemma~3.2 of~\cite{lw1}, it is easy
to find two positive numbers $t_{\varepsilon,i },s_{\varepsilon, j }$ such
that $\left( t_{\varepsilon,i }u_{\varepsilon ,i},s_{\varepsilon,j
}v_{\varepsilon ,j}\right) \in \mathbf{N}_{\varepsilon }$ when $\varepsilon $
is sufficiently small and $\left( t_{\varepsilon,i },s_{\varepsilon,j
}\right) \rightarrow \left( 1,1\right) $ as $\varepsilon \rightarrow 0^{+},$
uniformly for $e\in S^{N-1}$. Moreover, by (\ref{1.11}), (\ref{1.12}) and $%
\left( t_{\varepsilon,i },s_{\varepsilon,j }\right) \rightarrow (1,1),$ as $%
\varepsilon \rightarrow 0^{+}$, we have
\begin{eqnarray*}
\Phi \left( t_{\varepsilon,i }u_{\varepsilon ,i}\right)
&=&z_{1,i}-x^{\varepsilon }+\varepsilon \Phi \left( \widehat{\omega }%
_{1}\psi _{\varepsilon }\right) +o_{\varepsilon }\left( 1\right) \\
&=&z_{1,i}+o_{\varepsilon }\left( 1\right) \,.
\end{eqnarray*}%
Similarly, $\Phi \left( s_{\varepsilon,j }v_{\varepsilon ,j}\right)
=z_{2,j}+o_{\varepsilon }\left( 1\right) .$ Thus, $\Phi \left(
t_{\varepsilon,i }u_{\varepsilon ,i}\right) \in C_{s}\left( z_{1,i}\right) $
and $\Phi \left( s_{\varepsilon,j }v_{\varepsilon ,j}\right) \in C_{s}\left(
z_{2,j}\right) $ for $\varepsilon $ sufficiently small, and $\left(
t_{\varepsilon,i }u_{\varepsilon ,i},s_{\varepsilon,j }v_{\varepsilon
,j}\right) \in N_{i,j}\left( \varepsilon \right) $. On the other hand, by (%
\ref{2.26}), (\ref{2.27}) and $\left( t_{\varepsilon,i },s_{\varepsilon,j
}\right) \rightarrow (1,1)$ as $\varepsilon \rightarrow 0^{+}$, it is easy
to check that
\begin{equation}
\varepsilon ^{-N}J_{\varepsilon }\left( t_{\varepsilon,i }u_{\varepsilon
,i},s_{\varepsilon,j }v_{\varepsilon ,j}\right) \rightarrow \widehat{\alpha }%
_{1}+\widehat{\alpha }_{2}\text{ as }\varepsilon \rightarrow 0^{+}\text{
uniformly in }e\in S^{N-1}.  \label{2.28}
\end{equation}%
Here we have used the facts that $V_{1}(z_{1,i})=\lambda _{1}$ and $%
V_{2}(z_{2,j})=\lambda _{2}$. Therefore, by (\ref{2.25}) and (\ref{2.28}),
we may complete the proof of Lemma~\ref{L2.4}.
\end{proof}

On the other hand, we may describe the lower bound of $\widetilde{\gamma }%
\left( \varepsilon \right) $ as follows:

\begin{lemma}
\label{L2.5}There exist positive numbers $\widetilde{\delta }$ and $%
\varepsilon _{\widetilde{\delta }}$ such that for every $\varepsilon \in
\left( 0,\varepsilon _{\widetilde{\delta }}\right) $
\begin{equation*}
\varepsilon ^{-N}\widetilde{\gamma }_{i,j}\left( \varepsilon \right) >%
\widehat{\alpha }_{1}+\widehat{\alpha }_{2}+\widetilde{\delta }\,\text{for
all }1\leq i\leq k\text{ and }1\leq j\leq \ell .
\end{equation*}
\end{lemma}

\begin{proof}
We shall prove Lemma~\ref{L2.5} by contradiction. Suppose there exist $1\leq
i\leq k,1\leq j\leq \ell $ and a sequence $\left\{ \varepsilon _{n}\right\}
_{n=1}^{\infty }\subset \mathbb{R}^{+}$ such that $\varepsilon
_{n}\rightarrow 0$ and
\begin{equation*}
\varepsilon _{n}^{-N}\widetilde{\gamma }_{i,j}\left( \varepsilon _{n}\right)
\rightarrow \widehat{\alpha }_{1}+\widehat{\alpha }_{2},
\end{equation*}%
as $n\rightarrow \infty $. Then there exists $\left\{ \left(
u_{n},v_{n}\right) \right\} _{n=1}^{\infty }$ such that $\left(
u_{n},v_{n}\right) \in O_{i,j}\left( \varepsilon _{n}\right) $ for $n\in
\mathbb{N}$, and
\begin{equation}
\varepsilon _{n}^{-N}J_{\varepsilon _{n}}\left( u_{n},v_{n}\right)
\rightarrow \widehat{\alpha }_{1}+\widehat{\alpha }_{2}\text{ as }%
n\rightarrow \infty .  \label{2.29}
\end{equation}%
Let $\widetilde{u}_{n}\left( x\right) =u_{n}\left( \varepsilon _{n}x\right) $
and $\widetilde{v}_{n}\left( x\right) =v_{n}\left( \varepsilon _{n}x\right) $%
. From (\ref{2.29}), there holds
\begin{equation}
\widetilde{J}_{\varepsilon _{n}}\left( \widetilde{u}_{n},\widetilde{v}%
_{n}\right) \rightarrow \widehat{\alpha }_{1}+\widehat{\alpha }_{2}\text{ as
}n\rightarrow \infty .  \label{2.30}
\end{equation}%
It follows from (\ref{2.30}) and Lemma \ref{L2.2} that
\begin{equation}
\int_{\mathbb{R}^{N}}\beta \left( \varepsilon _{n}x\right) \left\vert
\widetilde{u}_{n}^{+}\right\vert ^{q+1}\left\vert \widetilde{v}%
_{n}^{+}\right\vert ^{q+1}\mathrm{d}x\rightarrow 0\text{ as }n\rightarrow
\infty .  \label{2.31}
\end{equation}%
Since $\left( u_{n},v_{n}\right) \in O_{i,j}\left( \varepsilon _{n}\right) $
for $n\in \mathbb{N}$, then from (\ref{2.31}), we obtain
\begin{equation}
\int_{\mathbb{R}^{N}}\left( \left\vert (-\triangle )^{\frac{\alpha }{2}}%
\widetilde{u}_{n}\right\vert ^{2}+V_{1}\left( \varepsilon _{n}x\right)
\widetilde{u}_{n}^{2}\right) \mathrm{d}x=\mu _{1}\int_{\mathbb{R}%
^{N}}\left\vert \widetilde{u}_{n}^{+}\right\vert ^{2p+2}\mathrm{d}x+o\left(
1\right)  \label{2.32}
\end{equation}%
and
\begin{equation}
\int_{\mathbb{R}^{N}}\left( \left\vert (-\triangle )^{\frac{\alpha }{2}}%
\widetilde{v}_{n}\right\vert ^{2}+V_{2}\left( \varepsilon _{n}x\right)
\widetilde{v}_{n}^{2}\right) \mathrm{d}x=\mu _{2}\int_{\mathbb{R}%
^{N}}\left\vert \widetilde{v}_{n}^{+}\right\vert ^{2p+2}\mathrm{d}x+o\left(
1\right) .  \label{2.33}
\end{equation}%
Therefore, by~(\ref{2.32}), (\ref{2.33}) and a similar argument to the proof
of Lemma 3.3 in \cite{LWW} (or see \cite[lemma 3.1]{CN}), we may identify a
contradiction. This completes the proof.
\end{proof}

\section{Palais--Smale Sequences}

Fix $0<\sigma <(2p+2)\sqrt{\widehat{\alpha }_{1}\widehat{\alpha }_{2}}$
arbitrarily. Then by Lemma \ref{L2.2},
\begin{equation*}
-\sigma \varepsilon ^{N}<\int_{\mathbb{R}^{N}}\beta(x) \left\vert
u^{+}\right\vert ^{q+1}\left\vert v^{+}\right\vert ^{q+1}\mathrm{d}x\leq 0\,,
\end{equation*}%
and
\begin{eqnarray*}
&&\left[ p^{2}\mu _{1}\mu _{2}\int_{\mathbb{R}^{N}}\left\vert
u^{+}\right\vert ^{2p+2}\mathrm{d}x\int_{\mathbb{R}^{N}}\left\vert
v^{+}\right\vert ^{2p+2}\mathrm{d}x-\left( \int_{\mathbb{R}^{N}}\beta(x)
\left\vert u^{+}\right\vert ^{q+1}\left\vert v^{+}\right\vert ^{q+1}\mathrm{d%
}x\right) ^{2}\right] \\
&>&\varepsilon ^{2N}\left( 4(p+1)^{2}\widehat{\alpha }_{1}\widehat{\alpha }%
_{2}-\sigma ^{2}\right) >0\,,
\end{eqnarray*}%
for all $\left( u,v\right) \in N_{i,j}\left( \varepsilon \right) $ with $%
\varepsilon ^{-N}J_{\varepsilon }\left( u,v\right) \leq \widehat{\alpha }%
_{1}+\widehat{\alpha }_{2}+\delta _{0}$, where $\delta _{0}=\delta
_{0}(\sigma )>0$. Moreover, by Lemmas \ref{L2.4} and \ref{L2.5}, there exist
$\delta _{0}>0$ and $\varepsilon _{0}=\varepsilon _{0}\left( \delta
_{0}\right) >0$ such that
\begin{equation}
\gamma \left( \varepsilon \right) <\varepsilon ^{N}\left( \widehat{\alpha }%
_{1}+\widehat{\alpha }_{2}+\delta _{0}\right) <\widetilde{\gamma }%
_{i,j}\left( \varepsilon \right) \,\text{for all }\varepsilon \in \left(
0,\varepsilon _{0}\right) .  \label{3.2}
\end{equation}%
Here, we may set $0<\delta _{0}\leq \min \left\{ \overline{\delta },%
\widetilde{\delta },\widehat{\alpha }_{1},\widehat{\alpha }_{2}\right\} /2$,
then we have the following results.

\begin{lemma}
\label{L3.1}For each $\left( u,v\right) \in N_{i,j}\left( \varepsilon
\right) $ with
\begin{equation*}
J_{\varepsilon }\left( u,v\right) <\min \left\{ \varepsilon ^{N}\left(
\widehat{\alpha }_{1}+\widehat{\alpha }_{2}+\delta _{0}\right) ,\widetilde{%
\gamma }_{i,j}\left( \varepsilon \right) \right\} ,
\end{equation*}%
there exist $b>0$ and a differentiable function
\begin{equation*}
\left( s,t\right) :B\left( 0;b\right) \equiv \left\{ \left( \widetilde{u},%
\widetilde{v}\right) \in H:\Vert \left( \widetilde{u},\widetilde{v}\right)
\Vert _{H}<b\right\} \rightarrow \mathbb{R}^{+}\times \mathbb{R}^{+}
\end{equation*}%
such that $\left( s\left( 0,0\right) ,t\left( 0,0\right) \right) =\left(
1,1\right) $, the function $\left( s\left( \widetilde{u},\widetilde{v}%
\right) \left( u-\widetilde{u}\right) ,t\left( \widetilde{u},\widetilde{v}%
\right) \left( v-\widetilde{v}\right) \right) \in N_{i,j}\left( \varepsilon
\right) $ and
\begin{equation*}
J_{\varepsilon }\left( s\left( \widetilde{u},\widetilde{v}\right) \left( u-%
\widetilde{u}\right) ,t\left( \widetilde{u},\widetilde{v}\right) \left( v-%
\widetilde{v}\right) \right) <\min \left\{ \varepsilon ^{N}\left( \widehat{%
\alpha }_{1}+\widehat{\alpha }_{2}+\delta _{0}\right) ,\widetilde{\gamma }%
_{i,j}\left( \varepsilon \right) \right\}
\end{equation*}%
for all $\left( \widetilde{u},\widetilde{v}\right) \in B\left( 0;b\right) .$
Moreover,
\begin{equation}
\left(
\begin{array}{c}
\left\langle s^{\prime }\left( 0,0\right) ,\left( \phi ,\psi \right)
\right\rangle \\
\left\langle t^{\prime }\left( 0,0\right) ,\left( \phi ,\psi \right)
\right\rangle%
\end{array}%
\right) =\frac{\left(
\begin{array}{c}
\Psi _{1}\left( u,v,\phi ,\psi \right) K_{2}(u,v)-(q+1)\Psi _{2}\left(
u,v,\phi ,\psi \right) L_{\beta }\left( u,v\right) \\
-(q+1)\Psi _{1}\left( u,v,\phi ,\psi \right) L_{\beta }\left( u,v\right)
+\Psi _{2}\left( u,v,\phi ,\psi \right) K_{1}(u,v)%
\end{array}%
\right) }{\Gamma (u,v)},  \label{3.3}
\end{equation}%
where
\begin{equation*}
L_{\beta }\left( u,v\right) =\int_{\mathbb{R}^{N}}\beta (x)\left\vert
u^{+}\right\vert ^{q+1}\left\vert v^{+}\right\vert ^{q+1}\mathrm{d}x,
\end{equation*}%
\begin{equation*}
K_{1}(u,v)=2p\mu _{1}\int_{\mathbb{R}^{N}}\left\vert u^{+}\right\vert ^{2p+2}%
\mathrm{d}x+(q-1)L_{\beta }\left( u,v\right) ,
\end{equation*}%
\begin{equation*}
K_{2}(u,v)=2p\mu _{2}\int_{\mathbb{R}^{N}}\left\vert v^{+}\right\vert ^{2p+2}%
\mathrm{d}x+(q-1)L_{\beta }\left( u,v\right) ,
\end{equation*}%
\begin{eqnarray*}
\Psi _{1}\left( u,v,\phi ,\psi \right) &=&2\int_{\mathbb{R}^{N}}\left(
\varepsilon ^{2\alpha }(-\triangle )^{\frac{\alpha }{2}}u(-\triangle )^{%
\frac{\alpha }{2}}\phi +V_{1}(x)u\phi \right) \mathrm{d}x \\
&&-(q+1)\int_{\mathbb{R}^{N}}\beta (x)\left( \left\vert u^{+}\right\vert
^{q-1}u^{+}\left\vert v^{+}\right\vert ^{q+1}\phi +\left\vert
u^{+}\right\vert ^{q+1}\left\vert v^{+}\right\vert ^{q-1}v^{+}\psi \right)
\mathrm{d}x \\
&&-(2p+2)\int_{\mathbb{R}^{N}}\mu _{1}\left\vert u^{+}\right\vert
^{2p}u^{+}\phi \mathrm{d}x,~\forall ~\left( \phi ,\psi \right) \in H,
\end{eqnarray*}%
\begin{eqnarray*}
\Psi _{2}\left( u,v,\phi ,\psi \right) &=&2\int_{\mathbb{R}^{N}}\left(
\varepsilon ^{2\alpha }(-\triangle )^{\frac{\alpha }{2}}v(-\triangle )^{%
\frac{\alpha }{2}}\psi +V_{2}(x)v\psi \right) \mathrm{d}x \\
&&-(q+1)\int_{\mathbb{R}^{N}}\beta (x)\left( \left\vert u^{+}\right\vert
^{q-1}u^{+}\left\vert v^{+}\right\vert ^{q+1}\phi +\left\vert
u^{+}\right\vert ^{q+1}\left\vert v^{+}\right\vert ^{q-1}v^{+}\psi \right)
\mathrm{d}x \\
&&-(2p+2)\int_{\mathbb{R}^{N}}\mu _{2}\left\vert v^{+}\right\vert
^{2p}v^{+}\psi \mathrm{d}x,~\forall ~\left( \phi ,\psi \right) \in H,
\end{eqnarray*}%
and
\begin{eqnarray*}
\Gamma (u,v) &=&4p^{2}\mu _{1}\mu _{2}\int_{\mathbb{R}^{N}}\left\vert
u^{+}\right\vert ^{2p+2}\mathrm{d}x\int_{\mathbb{R}^{N}}\left\vert
v^{+}\right\vert ^{2p+2}\mathrm{d}x-4q\left[ L_{\beta }\left( u,v\right) %
\right] ^{2} \\
&&+2p(q-1)L_{\beta }\left( u,v\right) \left[ \mu _{1}\int_{\mathbb{R}%
^{N}}\left\vert u^{+}\right\vert ^{2p+2}\mathrm{d}x+\mu _{2}\int_{\mathbb{R}%
^{N}}\left\vert v^{+}\right\vert ^{2p+2}\mathrm{d}x\right] .
\end{eqnarray*}
\end{lemma}

\begin{proof}
Let $L_{\beta }\left( u,v\right) =\int_{\mathbb{R}^{N}}\beta (x)\left\vert
u^{+}\right\vert ^{q+1}\left\vert v^{+}\right\vert ^{q+1}\mathrm{d}x.$
Define a function $F:H\times \mathbb{R}^{2}\rightarrow \mathbb{R}^{2}$ given
by
\begin{equation*}
F_{\left( u,v\right) }\left( w_{1},w_{2},s,t\right) =\left(
\begin{array}{c}
F_{1}\left( w_{1},w_{2},s,t\right) \\
F_{2}\left( w_{1},w_{2},s,t\right)%
\end{array}%
\right) \,,
\end{equation*}%
where%
\begin{eqnarray*}
F_{1}\left( w_{1},w_{2},s,t\right) &=&s\left\Vert u-w_{1}\right\Vert
_{V_{1}}^{2}-\mu _{1}s^{2p+1}\int_{\mathbb{R}^{N}}\left\vert \left(
u-w_{1}\right) ^{+}\right\vert ^{2p+2}\mathrm{d}x \\
&&-s^{q}t^{q+1}L_{\beta }\left( u-w_{1},v-w_{2}\right)
\end{eqnarray*}%
and%
\begin{eqnarray*}
F_{2}\left( w_{1},w_{2},s,t\right) &=&t\left\Vert v-w_{2}\right\Vert
_{V_{2}}^{2}-\mu _{2}t^{2p+1}\int_{\mathbb{R}^{N}}\left\vert \left(
v-w_{2}\right) ^{+}\right\vert ^{2p+2}\mathrm{d}x \\
&&-s^{q+1}t^{q}L_{\beta }\left( u-w_{1},v-w_{2}\right) ,
\end{eqnarray*}%
Because $\left( u,v\right) \in N_{i,j}\left( \varepsilon \right) $ and $%
J_{\varepsilon }\left( u,v\right) <\min \left\{ \varepsilon ^{N}\left(
\widehat{\alpha }_{1}+\widehat{\alpha }_{2}+\delta _{0}\right) ,\widetilde{%
\gamma }_{i,j}\left( \varepsilon \right) \right\} ,$ we have $F\left(
0,0,1,1\right) =\left( 0,0\right) $ and $\Vert (u,v)\Vert _{H}\leq C$ for
some $C>0$. By a direct computation, we obtain
\begin{equation*}
\frac{\partial }{\partial s}F_{1}\left( 0,0,1,1\right) =-2p\mu _{1}\int_{%
\mathbb{R}^{N}}\left\vert u^{+}\right\vert ^{2p+2}\mathrm{d}x-(q-1)L_{\beta
}\left( u,v\right) ,
\end{equation*}%
\begin{equation*}
\frac{\partial }{\partial t}F_{1}\left( 0,0,1,1\right) =\frac{\partial }{%
\partial s}F_{2}\left( 0,0,1,1\right) =-(q+1)L_{\beta }\left( u,v\right) ,
\end{equation*}%
and
\begin{equation*}
\frac{\partial }{\partial t}F_{2}\left( 0,0,1,1\right) =-2p\mu _{2}\int_{%
\mathbb{R}^{N}}\left\vert v^{+}\right\vert ^{2p+2}\mathrm{d}x-(q-1)L_{\beta
}\left( u,v\right) .
\end{equation*}%
Hence, by (\ref{2.10}), we obtain
\begin{eqnarray*}
\left\vert
\begin{array}{cc}
\frac{\partial }{\partial s}F_{1}\left( 0,0,1,1\right) & \frac{\partial }{%
\partial t}F_{1}\left( 0,0,1,1\right) \\
\frac{\partial }{\partial s}F_{2}\left( 0,0,1,1\right) & \frac{\partial }{%
\partial t}F_{2}\left( 0,0,1,1\right)%
\end{array}%
\right\vert &=&\left[ 2p\mu _{1}\int_{\mathbb{R}^{N}}\left\vert
u^{+}\right\vert ^{2p+2}\mathrm{d}x+(q-1)L_{\beta }\left( u,v\right) \right]
\\
&&\times \left[ 2p\mu _{2}\int_{\mathbb{R}^{N}}\left\vert v^{+}\right\vert
^{2p+2}\mathrm{d}x+(q-1)L_{\beta }\left( u,v\right) \right] \\
&&-(q+1)^{2}\left[ L_{\beta }\left( u,v\right) \right] ^{2} \\
&>&0,
\end{eqnarray*}%
since $\beta \left( x\right) \leq 0$ in $\mathbb{R}^{N}$, $0<q\leq p\leq 1$
if $N\leq 4\alpha $ and $0<q\leq p\leq \frac{2\alpha }{N-2\alpha }$ if $%
N>4\alpha $. Thus, in view of the Implicit Function Theorem, there exists a
differentiable function
\begin{equation*}
\left( s,t\right) :B\left( 0;b\right) \subset H\left( \mathbb{R}^{N}\right)
\rightarrow \mathbb{R}^{+}\times \mathbb{R}^{+}
\end{equation*}%
such that $\left( s\left( 0\right) ,t\left( 0\right) \right) =\left(
1,1\right) $ and
\begin{equation*}
F_{\left( u,v\right) }\left( w_{1},w_{2},s\left( \widetilde{u},\widetilde{v}%
\right) ,t\left( \widetilde{u},\widetilde{v}\right) \right) =\left(
0,0\right) \text{ for all }\left( \widetilde{u},\widetilde{v}\right) \in
B\left( 0;b\right) .
\end{equation*}%
This function is equivalent to $\left( s\left( \widetilde{u},\widetilde{v}%
\right) \left( u-\widetilde{u}\right) ,t\left( \widetilde{u},\widetilde{v}%
\right) \left( v-\widetilde{v}\right) \right) \in \mathbf{N}_{\varepsilon }$%
. Moreover, by the continuity of $J_{\varepsilon }$ and $\left( s,t\right) ,$
we have
\begin{equation*}
\left( s\left( \widetilde{u},\widetilde{v}\right) \left( u-\widetilde{u}%
\right) ,t\left( \widetilde{u},\widetilde{v}\right) \left( v-\widetilde{v}%
\right) \right) \in N_{i,j}\left( \varepsilon \right)
\end{equation*}%
and
\begin{equation*}
J_{\varepsilon }\left( s\left( \widetilde{u},\widetilde{v}\right) \left( u-%
\widetilde{u}\right) ,t\left( \widetilde{u},\widetilde{v}\right) \left( v-%
\widetilde{v}\right) \right) <\min \left\{ \varepsilon ^{N}\left( \widehat{%
\alpha }_{1}+\widehat{\alpha }_{2}+\delta _{0}\right) ,\widetilde{\gamma }%
_{i,j}\left( \varepsilon \right) \right\} \,,
\end{equation*}%
if $b$ is sufficiently small.
\end{proof}

Now we may use Lemma~\ref{L3.1} to find a $(PS)_{\gamma _{i,j}\left(
\varepsilon \right) }$ sequence as follows:

\begin{proposition}
\label{P3.1}For each $1\leq i\leq k,1\leq j\leq \ell $ and $\varepsilon \in
\left( 0,\varepsilon _{0}\right) $, there exists a sequence $\left\{ \left(
u_{n},v_{n}\right) \right\} \subset N_{i,j}\left( \varepsilon \right) $ such
that%
\begin{equation*}
J_{\varepsilon }\left( u_{n},v_{n}\right) =\gamma _{i,j}\left( \varepsilon
\right) +o\left( 1\right) \text{ and }J_{\varepsilon }^{\prime }\left(
u_{n},v_{n}\right) =o\left( 1\right) \;\text{in }H^{\ast }\text{ as }%
n\rightarrow \infty .
\end{equation*}
\end{proposition}

\begin{proof}
By the Ekeland variational principle \cite{E}, there exists a minimizing
sequence $\left\{ \left( u_{n},v_{n}\right) \right\} \subset N_{i,j}\left(
\varepsilon \right) $ such that%
\begin{equation}
J_{\varepsilon }\left( u_{n},v_{n}\right) <\gamma _{i,j}\left( \varepsilon
\right) +\frac{1}{n}<\min \left\{ \varepsilon ^{N}\left( \widehat{\alpha }%
_{1}+\widehat{\alpha }_{2}+\delta _{0}\right) ,\widetilde{\gamma }%
_{i,j}\left( \varepsilon \right) \right\}  \label{3.4}
\end{equation}%
and
\begin{equation}
J_{\varepsilon }\left( u_{n},v_{n}\right) <J_{\varepsilon }\left(
w_{1},w_{2}\right) +\frac{1}{n}\left\Vert \left( w_{1},w_{2}\right) -\left(
u_{n},v_{n}\right) \right\Vert _{H}  \label{3.5}
\end{equation}%
for $\left( w_{1},w_{2}\right) \in \mathbf{N}_{\varepsilon }$ and $n\geq
n_{0}$, where $n_{0}$ is a sufficiently large constant independent of $%
\left( w_{1},w_{2}\right) $. Applying Lemma~\ref{L3.1} for each $\left(
u_{n},v_{n}\right) ,$ we obtain the function
\begin{equation*}
\left( s_{n},t_{n}\right) :B\left( 0;b_{n}\right) \rightarrow \mathbb{R}%
^{+}\times \mathbb{R}^{+}
\end{equation*}%
such that $\left( s_{n}\left( \widetilde{u},\widetilde{v}\right) \left(
u_{n}-\widetilde{u}\right) ,t_{n}\left( \widetilde{u},\widetilde{v}\right)
\left( v_{n}-\widetilde{v}\right) \right) \in N_{i,j}\left( \varepsilon
\right) ,$ for $\left( \widetilde{u},\widetilde{v}\right) \in B\left(
0;b_{n}\right) ,$ where $b_{n}>0$. Fix $\left( u,v\right) \in H\setminus
\{(0,0)\}$ arbitrarily. For $0<\rho <b_{n},$ let $\left( \phi _{\rho },\psi
_{\rho }\right) =\left( \rho u,\rho v\right) \,\left\Vert \left( u,v\right)
\right\Vert _{H}^{-1}$, $y_{\rho }=s_{n,\rho }u_{n,\rho }$ and $z_{\rho
}=t_{n,\rho }v_{n,\rho }.$ Hereafter, $s_{n,\rho }=s_{n}\left( \phi _{\rho
},\psi _{\rho }\right) $, $t_{n,\rho }=t_{n}\left( \phi _{\rho },\psi _{\rho
}\right) $, $u_{n,\rho }=u_{n}-\phi _{\rho }$ and $v_{n,\rho }=v_{n}-\psi
_{\rho }$. Then $\left( y_{\rho },z_{\rho }\right) \in N_{i,j}\left(
\varepsilon \right) $. By $\left( \ref{3.5}\right) $, there holds
\begin{equation}
J_{\varepsilon }\left( y_{\rho },z_{\rho }\right) -J_{\varepsilon }\left(
u_{n},v_{n}\right) \geq -\frac{1}{n}\left\Vert \left( y_{\rho },z_{\rho
}\right) -\left( u_{n},v_{n}\right) \right\Vert _{H}.  \label{3.6}
\end{equation}%
Because $(y_{\rho },z_{\rho })\rightarrow (u_{n},v_{n})$ as $\rho
\rightarrow 0^{+},$ it follows from (\ref{3.6}) that
\begin{eqnarray}
&&-\left\langle J_{\varepsilon }^{\prime }\left( u_{n},v_{n}\right) ,\left(
\phi _{\rho },\psi _{\rho }\right) \right\rangle +\left\langle
J_{\varepsilon }^{\prime }\left( u_{n},v_{n}\right) ,\left( \left( s_{n,\rho
}-1\right) u_{n,\rho },\left( t_{n,\rho }-1\right) v_{n,\rho }\right)
\right\rangle  \notag \\
&\geq &-\frac{1}{n}\left\Vert \left( y_{\rho },z_{\rho }\right) -\left(
u_{n},v_{n}\right) \right\Vert _{H}+o\left( \left\Vert \left( y_{\rho
},z_{\rho }\right) -\left( u_{n},v_{n}\right) \right\Vert _{H}\right) .
\label{3.7}
\end{eqnarray}%
On the other hand, due to $\left( y_{\rho },z_{\rho }\right) =\left(
s_{n,\rho }u_{n,\rho }\,,t_{n,\rho }v_{n,\rho }\right) \in N_{i,j}\left(
\varepsilon \right) $, we have
\begin{eqnarray*}
\int_{\mathbb{R}^{N}}\left( \varepsilon ^{2\alpha }\left\vert (-\triangle )^{%
\frac{\alpha }{2}}u_{n,\rho }\right\vert ^{2}+V_{1}(x)u_{n,\rho }^{2}\right)
\mathrm{d}x &=&\mu _{1}s_{n,\rho }^{2p}\int_{\mathbb{R}^{N}}\left\vert
u_{n,\rho }^{+}\right\vert ^{2p+2}\mathrm{d}x \\
&&+s_{n,\rho }^{q-1}t_{n,\rho }^{q+1}\int_{\mathbb{R}^{N}}\beta
(x)\left\vert u_{n,\rho }^{+}\right\vert ^{q+1}\left\vert v_{n,\rho
}^{+}\right\vert ^{q+1}\mathrm{d}x
\end{eqnarray*}%
and
\begin{eqnarray*}
\int_{\mathbb{R}^{N}}\left( \varepsilon ^{2\alpha }\left\vert (-\triangle )^{%
\frac{\alpha }{2}}v_{n,\rho }\right\vert ^{2}+V_{2}(x)v_{n,\rho }^{2}\right)
\mathrm{d}x &=&\mu _{2}t_{n,\rho }^{2p}\int_{\mathbb{R}^{N}}\left\vert
u_{n,\rho }^{+}\right\vert ^{2p+2}\mathrm{d}x \\
&&+t_{n,\rho }^{q-1}s_{n,\rho }^{q+1}\int_{\mathbb{R}^{N}}\beta
(x)\left\vert u_{n,\rho }^{+}\right\vert ^{q+1}\left\vert v_{n,\rho
}^{+}\right\vert ^{q+1}\mathrm{d}x.
\end{eqnarray*}%
Then it is easy to verify that
\begin{eqnarray}
&&\left\langle J_{\varepsilon }^{\prime }\left( y_{\rho },z_{\rho }\right)
,\left( \left( s_{n,\rho }-1\right) u_{n,\rho },\left( t_{n,\rho }-1\right)
v_{n,\rho }\right) \right\rangle  \notag \\
&=&\left\langle J_{\varepsilon }^{\prime }\left( s_{n,\rho }u_{n,\rho
}\,,t_{n,\rho }v_{n,\rho }\right) ,\left( \left( s_{n,\rho }-1\right)
u_{n,\rho },\left( t_{n,\rho }-1\right) v_{n,\rho }\right) \right\rangle
\label{3.8} \\
&=&0.  \notag
\end{eqnarray}%
Combining $\left( \ref{3.7}\right) $ and (\ref{3.8}), we obtain%
\begin{eqnarray}
&&\left\langle J_{\varepsilon }^{\prime }\left( u_{n},v_{n}\right) ,\left(
u,v\right) \left\Vert \left( u,v\right) \right\Vert _{H}^{-1}\right\rangle
\notag \\
&\leq &\frac{\left\Vert \left( y_{\rho },z_{\rho }\right) -\left(
u_{n},v_{n}\right) \right\Vert _{H}}{n\rho }+\frac{o\left( \left\Vert \left(
y_{\rho },z_{\rho }\right) -\left( u_{n},v_{n}\right) \right\Vert
_{H}\right) }{\rho }  \label{3.9} \\
&&+\left\langle J_{\varepsilon }^{\prime }\left( u_{n},v_{n}\right)
-J_{\varepsilon }^{\prime }\left( y_{\rho },z_{\rho }\right) ,\left( \frac{%
s_{n,\rho }-1}{\rho }\,u_{n,\rho },\frac{t_{n,\rho }-1}{\rho }\,v_{n,\rho
}\right) \right\rangle .  \notag
\end{eqnarray}%
By $\left( \ref{3.3}\right) $ and (\ref{3.4}), there exists a positive
constant $C$ that is independent of $\rho $ and $n$ such that
\begin{equation}
\lim_{\rho \rightarrow 0^{+}}\frac{\left\vert s_{n,\rho }-1\right\vert }{%
\rho }\leq C\,,\quad \lim_{\rho \rightarrow 0^{+}}\frac{\left\vert t_{n,\rho
}-1\right\vert }{\rho }\leq C,  \label{3.10}
\end{equation}%
and
\begin{equation}
\left\Vert \left( y_{\rho },z_{\rho }\right) -\left( u_{n},v_{n}\right)
\right\Vert _{H}\leq C\left( \rho +\max \left\{ \left\vert s_{n,\rho
}-1\right\vert ,\left\vert t_{n,\rho }-1\right\vert \right\} \right) .
\label{3.11}
\end{equation}%
Setting $\rho \rightarrow 0^{+}$ on (\ref{3.9}), we obtain
\begin{equation*}
\left\langle J_{\varepsilon }^{\prime }\left( u_{n},v_{n}\right) ,\left(
u,v\right) \left\Vert \left( u,v\right) \right\Vert _{H}^{-1}\right\rangle
\leq \frac{C}{n}.
\end{equation*}%
Here we have used (\ref{3.10}), (\ref{3.11}) and the fact that $\left(
y_{\rho },z_{\rho }\right) \rightarrow \left( u_{n},v_{n}\right) $ as $\rho
\rightarrow 0^{+}$. Therefore,
\begin{equation}
\Vert J_{\varepsilon }^{\prime }\left( u_{n},v_{n}\right) \Vert _{H^{\ast
}}\leq \frac{C}{n}\text{ for }n\geq n_{0}.  \label{3.12}
\end{equation}%
Therefore, by (\ref{3.12}), we may complete the proof of Proposition~\ref%
{P3.1}.
\end{proof}

To prove Theorem~\ref{t1.1}, we need the (PS) condition of $J_{\varepsilon }$
as follows:

\begin{proposition}
\label{P3.2}Assume that $\left\{ \left( u_{n},v_{n}\right) \right\} $ is a
sequence in $N_{i,j}\left( \varepsilon \right) $ satisfying\newline
$\left( i\right) \ J_{\varepsilon }\left( u_{n},v_{n}\right) \rightarrow
\varepsilon ^{N}\theta _{\varepsilon }$ as $n\rightarrow \infty ,$ where $%
\theta _{\varepsilon }<\widehat{\alpha }_{1}+\widehat{\alpha }_{2}+\delta
_{0};$\newline
$\left( ii\right) \ J_{\varepsilon }^{\prime }\left( u_{n},v_{n}\right)
\rightarrow 0\;$strongly in $H^{\ast }$ as $n\rightarrow \infty .$\newline
Then there exists a convergent subsequence such that as $n\rightarrow \infty
$, $\left( u_{n},v_{n}\right) \rightarrow \left( u_{0},v_{0}\right) $
strongly in $H,$ where $\left( u_{0},v_{0}\right) \in N_{i,j}\left(
\varepsilon \right) .$
\end{proposition}

\begin{proof}
It follows from (\ref{1.14}) and $(i)$ that
\begin{equation*}
\frac{p}{2p+2}\left\Vert \left( u_{n},v_{n}\right) \right\Vert _{H}^{2}\leq
J_{\varepsilon }\left( u_{n},v_{n}\right) \rightarrow \varepsilon ^{N}\theta
_{\varepsilon }\text{ as }n\rightarrow \infty ,
\end{equation*}%
and $\left\{ \left( u_{n},v_{n}\right) \right\} $ is bounded in $H.$ Thus,
there exists a convergent subsequence of $\left\{ \left( u_{n},v_{n}\right)
\right\} $ (denoted as $\left\{ \left( u_{n},v_{n}\right) \right\} $ for
notation convenience) such that as $n\rightarrow \infty $,
\begin{equation}
\left( u_{n},v_{n}\right) \rightharpoonup \left( u_{0},v_{0}\right) \text{
weakly in }H\,,  \label{3.13}
\end{equation}%
where $\left( u_{0},v_{0}\right) \in H$. By $0<q\leq p\leq 1$ if $N\leq
4\alpha $, $0<q\leq p<\frac{2\alpha }{N-2\alpha }$ if $N>4\alpha $, then in
view of (\ref{3.13}) and the Sobolev compact embedding, we have
\begin{equation}
\left( u_{n},v_{n}\right) \rightarrow \left( u_{0},v_{0}\right) ~\text{%
strongly in }L_{loc}^{r}\left( \mathbb{R}^{N}\right) \times
L_{loc}^{r}\left( \mathbb{R}^{N}\right) \text{ for }2<r<2_{\alpha }^{\ast }
\label{3.14}
\end{equation}%
and
\begin{equation}
\left( u_{n},v_{n}\right) \rightarrow \left( u_{0},v_{0}\right) ~\text{a.e.
in }\mathbb{R}^{N}.  \label{3.15}
\end{equation}%
Now we claim that there exist a subsequence $\left\{ \left(
u_{n},v_{n}\right) \right\} $ and a sequence $\{z_{n}\}\subset \mathbb{R}%
^{N} $ such that
\begin{equation}
\int_{B^{N}\left( z_{n},R\right) }\left\vert u_{n}^{+}\right\vert
^{q+1}\left\vert v_{n}^{+}\right\vert ^{q+1}\mathrm{d}x\geq d_{0}>0,\quad
\forall ~n\in \mathbb{N},  \label{3.16}
\end{equation}%
where $d_{0}$ and $R$ are positive constants that are independent of $n.$
Suppose on the contrary. Thus, for all $R>0$,
\begin{equation*}
\sup_{x\in \mathbb{R}^{N}}\int_{B^{N}\left( x,R\right) }\left\vert
u_{n}^{+}\right\vert ^{q+1}\left\vert v_{n}^{+}\right\vert ^{q+1}\mathrm{d}%
x\rightarrow 0\text{ as }n\rightarrow \infty .
\end{equation*}%
Then, similarly to the argument of Lemma I.1 in \cite{Li1} (see also \cite%
{Wi}), we have
\begin{equation*}
\int_{\mathbb{R}^{N}}\left\vert u_{n}^{+}\right\vert ^{q+1}\left\vert
v_{n}^{+}\right\vert ^{q+1}\mathrm{d}x\rightarrow 0\text{ as }n\rightarrow
\infty ,
\end{equation*}%
this implies that%
\begin{equation}
\left\vert \int_{\mathbb{R}^{N}}\beta (x)\left\vert u_{n}^{+}\right\vert
^{q+1}\left\vert v_{n}^{+}\right\vert ^{q+1}\mathrm{d}x\right\vert
\rightarrow 0\text{ as }n\rightarrow \infty .  \label{3.17}
\end{equation}%
Therefore, by~$(ii)$, (\ref{3.13}), (\ref{3.14}) and (\ref{3.17}), the
functions $u_{0}$ and $v_{0}$ satisfy
\begin{equation}
\varepsilon ^{2\alpha }(-\triangle )^{\alpha }u_{0}+V_{1}(x)u_{0}=\mu
_{1}|u_{0}|^{2p}u_{0}^{+},\text{ in }\mathbb{R}^{N}  \label{3.18}
\end{equation}%
and
\begin{equation}
\varepsilon ^{2\alpha }(-\triangle )^{\alpha }v_{0}+V_{2}(x)u_{0}=\mu
_{2}|v_{0}|^{2p}v_{0}^{+},\text{ in }\mathbb{R}^{N}.  \label{3.19}
\end{equation}%
On the other hand, by hypethesis $\left( ii\right) $ and (\ref{3.15}), we
have
\begin{equation*}
u_{0},v_{0}\geq 0\text{ in }\mathbb{R}^{N}.
\end{equation*}%
Furthermore, it follows from (\ref{3.15}) and (\ref{3.17}) that
\begin{equation}
u_{0}v_{0}\equiv 0.  \label{3.20}
\end{equation}%
Thus, by (\ref{3.18})--(\ref{3.20}) and the strong maximum principle for the
fractional Laplacian \cite{Sil}, we have $u_{0}\equiv 0$ or $v_{0}\equiv 0$.
Without loss of generality, we may assume that $u_{0}\equiv 0$. By Lemma~\ref%
{L2.1} and the concentration--compactness principle (cf. \cite{Li1,Li2}),
there are positive constants $R,b$ and a sequence $\left\{ x_{n}\right\}
_{n=1}^{\infty }\subset \mathbb{R}^{N}$ such that
\begin{equation}
\int_{B^{N}\left( 0;R\right) }\left\vert \left( -x_{n}\ast u_{n}\right)
^{+}\right\vert ^{2p+2}\mathrm{d}x\geq b~\text{for}~n~\text{sufficiently
large.}  \label{3.21}
\end{equation}%
By $\left( \ref{3.14}\right) $, $\left( \ref{3.21}\right) $ and $u_{0}\equiv
0$, it is easy to verify that $\left\{ x_{n}\right\} $ is an unbounded
sequence in $\mathbb{R}^{N}.$ Let $\widehat{u}_{n}\left( x\right)
=(-x_{n}\ast u_{n})(x)$. Then sequence $\left\{ \widehat{u}_{n}\right\} $ is
bounded in $H^{\alpha }\left( \mathbb{R}^{N}\right) ,$ so we may assume that
there exists $\widehat{u}_{0}\in H^{\alpha }\left( \mathbb{R}^{N}\right) $
such that
\begin{equation}
\widehat{u}_{n}\rightharpoonup \widehat{u}_{0}\text{ weakly in }H^{\alpha
}\left( \mathbb{R}^{N}\right)  \label{3.22}
\end{equation}%
and
\begin{equation}
\widehat{u}_{n}\rightarrow \widehat{u}_{0}\text{ strongly\ in }%
L_{loc}^{2p+2}\left( \mathbb{R}^{N}\right) .  \label{3.23}
\end{equation}%
It follows from $\left( \ref{3.21}\right) $ and $\left( \ref{3.23}\right) $
that $\widehat{u}_{0}\not\equiv 0$ in $\mathbb{R}^{N}.$ Set $w_{n}=\widehat{u%
}_{n}-\widehat{u}_{0}.$ Then we may divide the proof into the following
cases:\newline
Case $I:$ $\int_{\mathbb{R}^{N}}\left( |(-\triangle )^{\frac{\alpha }{2}%
}w_{n}|^{2}+\lambda _{1}w_{n}^{2}\right) \mathrm{d}x\rightarrow 0$ as $%
n\rightarrow \infty ;$\newline
Case $II:$ $\int_{\mathbb{R}^{N}}\left( |(-\triangle )^{\frac{\alpha }{2}%
}w_{n}|^{2}+\lambda _{1}w_{n}^{2}\right) \mathrm{d}x\geq b$ for large $n$
and for some constant $b>0.$

Suppose Case~$I$ holds. Then
\begin{equation*}
x_{n}=\Phi \left( u_{n}\right) -\Phi \left( \widehat{u}_{0}\right) +o(1),
\end{equation*}%
which implies $\left\vert \Phi \left( u_{n}\right) \right\vert \rightarrow
\infty $ as $n\rightarrow \infty ,$ this contradicts $\Phi \left(
u_{n}\right) \in C_{s}\left( z_{1,i}\right) .$ Here, we have used the fact
that $\left\{ x_{n}\right\} $ is a unbounded sequence in $\mathbb{R}^{N}.$

Suppose Case $II$ holds. Since $\ J_{\varepsilon }^{\prime
}(u_{n},v_{n})\rightarrow 0\;$strongly in $H^{\ast },$ then by (\ref{3.17}),
we have
\begin{equation}
\int_{\mathbb{R}^{N}}\left( \varepsilon ^{2\alpha }|(-\triangle )^{\frac{%
\alpha }{2}}u_{n}|^{2}+V_{1}(x)u_{n}^{2}\right) \mathrm{d}x=\mu _{1}\int_{%
\mathbb{R}^{N}}\left\vert u_{n}^{+}\right\vert ^{2p+2}\mathrm{d}x+o\left(
1\right)  \label{3.24}
\end{equation}%
and
\begin{equation*}
\int_{\mathbb{R}^{N}}\left( \varepsilon ^{2\alpha }|(-\triangle )^{\frac{%
\alpha }{2}}v_{n}|^{2}+V_{2}(x)v_{n}^{2}\right) \mathrm{d}x=\mu _{2}\int_{%
\mathbb{R}^{N}}\left\vert v_{n}^{+}\right\vert ^{2p+2}\mathrm{d}x+o\left(
1\right) .
\end{equation*}%
By $\left( \ref{3.24}\right) $ and Br\'{e}zis--Lieb lemma \cite{BLi}, we
obtain%
\begin{equation*}
\int_{\mathbb{R}^{N}}\left( \varepsilon ^{2\alpha }|(-\triangle )^{\frac{%
\alpha }{2}}w_{n}|^{2}+V_{1}(x)w_{n}^{2}\right) \mathrm{d}x=\mu _{1}\int_{%
\mathbb{R}^{N}}\left\vert w_{n}^{+}\right\vert ^{2p+2}\mathrm{d}x+o\left(
1\right) .
\end{equation*}%
Because $\int_{\mathbb{R}^{N}}\left( |(-\triangle )^{\frac{\alpha }{2}%
}\omega _{n}|^{2}+\lambda _{1}w_{n}^{2}\right) dx\geq b$ for large $n,$ it
is easy to find a sequence $\left\{ s_{n}\right\} \subset \mathbb{R}^{+}$
with $s_{n}\rightarrow 1$ as $n\rightarrow \infty $ such that
\begin{equation*}
s_{n}^{2}\int_{\mathbb{R}^{N}}\left( \varepsilon ^{2\alpha }\left\vert
(-\triangle )^{\frac{\alpha }{2}}w_{n}\right\vert ^{2}+(-x_{n}\ast
V_{1})w_{n}^{2}\right) \mathrm{d}x=\mu _{1}s_{n}^{2p+2}\int_{\mathbb{R}%
^{N}}\left\vert w_{n}^{+}\right\vert ^{2p+2}\mathrm{d}x
\end{equation*}%
and so%
\begin{equation*}
\frac{1}{2}\int_{\mathbb{R}^{N}}\left( \varepsilon ^{2\alpha }\left\vert
(-\triangle )^{\frac{\alpha }{2}}w_{n}\right\vert ^{2}+(-x_{n}\ast
V_{1})w_{n}^{2}\right) \mathrm{d}x-\frac{\mu _{1}}{2p+2}\int_{\mathbb{R}%
^{N}}\left\vert w_{n}^{+}\right\vert ^{2p+2}\mathrm{d}x\geq \varepsilon ^{N}%
\widehat{\alpha }_{1}+o\left( 1\right) .
\end{equation*}%
Analogously,%
\begin{equation*}
\frac{1}{2}\int_{\mathbb{R}^{N}}\left( \varepsilon ^{2\alpha }\left\vert
(-\triangle )^{\frac{\alpha }{2}}v_{n}\right\vert ^{2}+(-x_{n}\ast
V_{2})v_{n}^{2}\right) \mathrm{d}x-\frac{\mu _{2}}{2p+2}\int_{\mathbb{R}%
^{N}}\left\vert v_{n}^{+}\right\vert ^{2p+2}\mathrm{d}x\geq \varepsilon ^{N}%
\widehat{\alpha }_{2}+o\left( 1\right) .
\end{equation*}%
Furthermore, $\left\langle J_{\varepsilon }^{\prime }(u_{n},v_{n}),\left(
x_{n}\ast \widehat{u}_{0},0\right) \right\rangle \rightarrow 0.$ By $\left( %
\ref{3.22}\right) $ and $\left( \ref{3.23}\right) ,$ we have%
\begin{equation*}
\int_{\mathbb{R}^{N}}\left( \varepsilon ^{2\alpha }\left\vert (-\triangle )^{%
\frac{\alpha }{2}}\widehat{u}_{0}\right\vert ^{2}+(-x_{n}\ast V_{1})\widehat{%
u}_{0}^{2}\right) \mathrm{d}x-\mu _{1}\int_{\mathbb{R}^{N}}\left\vert
\widehat{u}_{0}^{+}\right\vert ^{2p+2}\mathrm{d}x=o\left( 1\right) ,
\end{equation*}%
which implies that
\begin{equation*}
\frac{1}{2}\int_{\mathbb{R}^{N}}\left( \varepsilon ^{2\alpha }\left\vert
(-\triangle )^{\frac{\alpha }{2}}\widehat{u}_{0}\right\vert ^{2}+(-x_{n}\ast
V_{1})\widehat{u}_{0}^{2}\right) \mathrm{d}x-\frac{\mu _{1}}{2p+2}\int_{%
\mathbb{R}^{N}}\left\vert \widehat{u}_{0}^{+}\right\vert ^{2p+2}\mathrm{d}%
x\geq \varepsilon ^{N}\widehat{\alpha }_{1}+o\left( 1\right) .
\end{equation*}%
It follows from $\left( \ref{1.14}\right) ,$ $\left( \ref{3.17}\right) $ and
Br\'{e}zis--Lieb lemma \cite{BLi} that
\begin{eqnarray*}
J_{\varepsilon }\left( u_{n},v_{n}\right) &=&\frac{1}{2}\int_{\mathbb{R}%
^{N}}\left( \varepsilon ^{2\alpha }\left\vert (-\triangle )^{\frac{\alpha }{2%
}}\widehat{w}_{n}\right\vert ^{2}+(-x_{n}\ast V_{1})\widehat{w}%
_{n}^{2}\right) \mathrm{d}x-\frac{\mu _{1}}{2p+2}\int_{\mathbb{R}%
^{N}}\left\vert \widehat{w}_{n}^{+}\right\vert ^{2p+2}\mathrm{d}x \\
&&+\frac{1}{2}\int_{\mathbb{R}^{N}}\left( \varepsilon ^{2\alpha }\left\vert
(-\triangle )^{\frac{\alpha }{2}}\widehat{u}_{0}\right\vert ^{2}+(-x_{n}\ast
V_{1})\widehat{u}_{0}^{2}\right) \mathrm{d}x-\frac{\mu _{1}}{2p+2}\int_{%
\mathbb{R}^{N}}\left\vert \widehat{u}_{0}^{+}\right\vert ^{2p+2}\mathrm{d}x
\\
&&+\frac{1}{2}\int_{\mathbb{R}^{N}}\left( \varepsilon ^{2\alpha }\left\vert
(-\triangle )^{\frac{\alpha }{2}}\widehat{v}_{n}\right\vert ^{2}+V_{2}\left(
x\right) v_{n}^{2}\right) \mathrm{d}x-\frac{\mu _{2}}{2p+2}\int_{\mathbb{R}%
^{N}}\left\vert v_{n}^{+}\right\vert ^{2p+2}\mathrm{d}x+o\left( 1\right) \\
&\geq &2\varepsilon ^{N}\widehat{\alpha }_{1}+\varepsilon ^{N}\widehat{%
\alpha }_{2}+o\left( 1\right) ,
\end{eqnarray*}%
which implies that
\begin{equation*}
\varepsilon ^{N}\theta _{\varepsilon }=\lim_{n\rightarrow \infty
}J_{\varepsilon }\left( u_{n},v_{n}\right) \geq 2\varepsilon ^{N}\widehat{%
\alpha }_{1}+\varepsilon ^{N}\widehat{\alpha }_{2}.
\end{equation*}%
However, $\theta _{\varepsilon }<\widehat{\alpha }_{1}+\widehat{\alpha }%
_{2}+\delta _{0}$ and $0<\delta _{0}\leq \frac{1}{2}\min \{\widehat{\alpha }%
_{1},\widehat{\alpha }_{2}\}.$ Hence, we find a contradiction and complete
the proof of (\ref{3.16}). Let $\left( \widetilde{u}_{n},\widetilde{v}%
_{n}\right) =\left( -z_{n}\ast u_{n},-z_{n}\ast v_{n}\right) $ for $n\in
\mathbb{N}$, where the operation $\ast $ is defined in (\ref{1.11}). Then as
for (\ref{3.13})--(\ref{3.14}), we have
\begin{eqnarray}
\left( \widetilde{u}_{n},\widetilde{v}_{n}\right) &\rightharpoonup &\left(
\widetilde{u}_{0},\widetilde{v}_{0}\right) \text{ weakly in }H,  \label{3.27}
\\
\left( \widetilde{u}_{n},\widetilde{v}_{n}\right) &\rightarrow &\left(
\widetilde{u}_{0},\widetilde{v}_{0}\right) \text{ strongly in }%
L_{loc}^{r}\left( \mathbb{R}^{N}\right) \times L_{loc}^{r}\left( \mathbb{R}%
^{N}\right) \text{ and for all }2<r<2_{\alpha }^{\ast },  \label{3.28} \\
\left( \widetilde{u}_{n},\widetilde{v}_{n}\right) &\rightarrow &\left(
\widetilde{u}_{0},\widetilde{v}_{0}\right) \text{ a.e. in }\mathbb{R}^{N},
\label{3.29}
\end{eqnarray}%
where $\left( \widetilde{u}_{0},\widetilde{v}_{0}\right) \in H$. Besides, it
follows from $\left( \ref{3.16}\right) $ and $\left( \ref{3.28}\right) $
that
\begin{equation*}
\int_{B^{N}\left( 0,R\right) }\left\vert \widetilde{u}_{0}^{+}\right\vert
^{q+1}\left\vert \widetilde{v}_{0}^{+}\right\vert ^{q+1}\mathrm{d}x\geq
d_{0},
\end{equation*}%
which implies $\widetilde{u}_{0}^{+}\not\equiv 0$ and $\widetilde{v}%
_{0}^{+}\not\equiv 0$ in $\mathbb{R}^{N}.$ Because $\beta \left( x\right) $
and $V_{l}(x)$ are uniformly continuous for each $l\in \{1,2\}$, the
sequences $\left\{ -z_{n}\ast \beta \right\} $ and $\left\{ -z_{n}\ast
V_{l}\right\} $ are equicontinuous, the Arzel\`{a}--Ascoli theorem allows us
to assume that $-z_{n}\ast \beta \rightarrow \beta _{\infty }\in C\left(
\mathbb{R}^{N}\right) $ and $-z_{n}\ast V_{l}\rightarrow V_{l}^{0}\in
C\left( \mathbb{R}^{N}\right) $ uniformly on bounded subsets of $\mathbb{R}%
^{N}.$ Then%
\begin{eqnarray*}
\int_{\mathbb{R}^{N}}\left( -z_{n}\ast \beta \right) \left\vert \widetilde{u}%
_{n}^{+}\right\vert ^{q-1}\widetilde{u}_{n}^{+}\left\vert \widetilde{v}%
_{n}^{+}\right\vert ^{q+1}\varphi _{1}\mathrm{d}x &\rightarrow &\int_{%
\mathbb{R}^{N}}\beta _{\infty }(x)\left\vert \widetilde{u}%
_{0}^{+}\right\vert ^{q-1}\widetilde{u}_{0}^{+}\left\vert \widetilde{v}%
_{0}^{+}\right\vert ^{q+1}\varphi _{1}\mathrm{d}x\text{ as }n\rightarrow
\infty , \\
\int_{\mathbb{R}^{N}}\left( -z_{n}\ast \beta \right) \left\vert \widetilde{u}%
_{n}^{+}\right\vert ^{q+1}\left\vert \widetilde{v}_{n}^{+}\right\vert ^{q-1}%
\widetilde{v}_{n}^{+}\varphi _{2}\mathrm{d}x &\rightarrow &\int_{\mathbb{R}%
^{N}}\beta _{\infty }(x)\left\vert \widetilde{u}_{0}^{+}\right\vert
^{q+1}\left\vert \widetilde{v}_{0}^{+}\right\vert ^{q+1}\widetilde{v}%
_{0}^{+}\varphi _{2}\mathrm{d}x\text{ as }n\rightarrow \infty
\end{eqnarray*}%
and%
\begin{eqnarray*}
\int_{\mathbb{R}^{N}}\left( -z_{n}\ast V_{1}\right) \widetilde{u}_{n}\varphi
\mathrm{d}x &\rightarrow &\int_{\mathbb{R}^{N}}V_{1}^{0}(x)\widetilde{u}%
_{0}\varphi _{1}\mathrm{d}x\text{ as }n\rightarrow \infty , \\
\int_{\mathbb{R}^{N}}\left( -z_{n}\ast V_{2}\right) \widetilde{v}_{n}\varphi
_{2}\mathrm{d}x &\rightarrow &\int_{\mathbb{R}^{N}}V_{2}^{0}(x)\widetilde{v}%
_{0}\varphi _{2}\mathrm{d}x\text{ as }n\rightarrow \infty ,
\end{eqnarray*}%
for all $\varphi _{1},\varphi _{2}\in C_{0}^{\infty }\left( \mathbb{R}%
^{N}\right) $. Therefore, by (\ref{3.27}), (\ref{3.28}) and the hypothesis $%
(ii)$, we obtain
\begin{eqnarray*}
&&\int_{\mathbb{R}^{N}}\left( \varepsilon ^{2\alpha }(-\triangle )^{\frac{%
\alpha }{2}}\widetilde{u}_{0}(-\triangle )^{\frac{\alpha }{2}}\varphi
_{1}+V_{1}^{0}\widetilde{u}_{0}\varphi _{1}\right) \mathrm{d}x-\mu _{1}\int_{%
\mathbb{R}^{N}}\left\vert \widetilde{u}_{0}^{+}\right\vert ^{2p}\widetilde{u}%
_{0}^{+}\varphi _{1}\mathrm{d}x \\
&&+\int_{\mathbb{R}^{N}}\left( \varepsilon ^{2\alpha }(-\triangle )^{\frac{%
\alpha }{2}}\widetilde{v}_{0}(-\triangle )^{\frac{\alpha }{2}}\varphi
_{2}+V_{2}^{0}\widetilde{v}_{0}\varphi _{2}\right) \mathrm{d}x-\mu _{2}\int_{%
\mathbb{R}^{N}}\left\vert \widetilde{v}_{0}^{+}\right\vert ^{2p}\widetilde{v}%
_{0}^{+}\varphi _{2}\mathrm{d}x \\
&&-\int_{\mathbb{R}^{N}}\beta _{\infty }(x)\left\vert \widetilde{u}%
_{0}\right\vert ^{q-1}\widetilde{u}_{0}\left\vert \widetilde{v}%
_{0}\right\vert ^{q+1}\varphi _{1}\mathrm{d}x-\int_{\mathbb{R}^{N}}\beta
_{\infty }(x)\left\vert \widetilde{u}_{0}^{+}\right\vert ^{q+1}\left\vert
\widetilde{v}_{0}^{+}\right\vert ^{q-1}\widetilde{v}_{0}^{+}\varphi _{2}%
\mathrm{d}x \\
&=&\lim_{n\rightarrow \infty }\left\langle J_{\varepsilon }^{\prime }\left(
\widetilde{u}_{n},\widetilde{v}_{n}\right) ,\left( \varphi _{1},\varphi
_{2}\right) \right\rangle \\
&=&\lim_{n\rightarrow \infty }\left\langle J_{\varepsilon }^{\prime }\left(
u_{n},v_{n}\right) ,\left( z_{n}\ast \varphi _{1}\,,z_{n}\ast \varphi
_{2}\right) \right\rangle =0,\quad \forall ~\varphi _{1},\varphi _{2}\in
C_{0}^{\infty }\left( \mathbb{R}^{N}\right) ,
\end{eqnarray*}%
i.e., $\left( \widetilde{u}_{0},\widetilde{v}_{0}\right) $ is a nonnegative
solution of the following problem
\begin{equation*}
\left\{
\begin{array}{ll}
\varepsilon ^{2\alpha }(-\triangle )^{\alpha }u+V_{1}^{0}\left( x\right)
u=\left( \mu _{1}|u|^{2p}+\beta _{\infty }\left( x\right)
|u|^{q-1}|v|^{q+1}\right) u, & \text{in }\mathbb{R}^{N}, \\
\varepsilon ^{2\alpha }(-\triangle )^{\alpha }v+V_{2}^{0}\left( x\right)
v=\left( \mu _{2}|v|^{2p}+\beta _{\infty }\left( x\right)
|v|^{q-1}|u^{+}|^{q+1}\right) v, & \text{in }\mathbb{R}^{N}.%
\end{array}%
\right.
\end{equation*}%
Hence%
\begin{equation}
\int_{\mathbb{R}^{N}}\left( \varepsilon ^{2\alpha }|(-\triangle )^{\frac{%
\alpha }{2}}\widetilde{u}_{0}|^{2}+V_{1}^{0}(x)\widetilde{u}_{0}^{2}\right)
\mathrm{d}x=\int_{\mathbb{R}^{N}}\left( \mu _{1}\left\vert \widetilde{u}%
_{0}^{+}\right\vert ^{2p+2}+\beta _{\infty }(x)\left\vert \widetilde{u}%
_{0}^{+}\right\vert ^{q+1}\left\vert \widetilde{v}_{0}^{+}\right\vert
^{q+1}\right) \mathrm{d}x  \label{3.30}
\end{equation}%
and
\begin{equation}
\int_{\mathbb{R}^{N}}\left( \varepsilon ^{2\alpha }|(-\triangle )^{\frac{%
\alpha }{2}}\widetilde{v}_{0}|^{2}+V_{2}^{0}(x)\widetilde{v}_{0}^{2}\right)
\mathrm{d}x=\int_{\mathbb{R}^{N}}\left( \mu _{2}\left\vert \widetilde{v}%
_{0}^{+}\right\vert ^{2p+2}+\beta _{\infty }(x)\left\vert \widetilde{u}%
_{0}^{+}\right\vert ^{q+1}\left\vert \widetilde{v}_{0}^{+}\right\vert
^{q+1}\right) \mathrm{d}x.  \label{3.31}
\end{equation}%
Then in view of Lemma \ref{L2.3}, we obtain
\begin{equation}
I_{\varepsilon }\left( \widetilde{u}_{0},\widetilde{v}_{0}\right) \geq
\varepsilon ^{N}\left( \widehat{\alpha }_{1}+\widehat{\alpha }_{2}\right) ,
\label{3.32}
\end{equation}%
where
\begin{eqnarray*}
I_{\varepsilon }\left( u,v\right) &=&\frac{1}{2}\int_{\mathbb{R}^{N}}\left(
\varepsilon ^{2\alpha }|(-\triangle )^{\frac{\alpha }{2}%
}u|^{2}+V_{1}^{0}(x)u^{2}\right) \mathrm{d}x+\frac{1}{2}\int_{\mathbb{R}%
^{N}}\left( \varepsilon ^{2\alpha }|(-\triangle )^{\frac{\alpha }{2}%
}v|^{2}+V_{2}^{0}(x)v^{2}\right) \mathrm{d}x \\
&&-\frac{1}{2p+2}\left( \int_{\mathbb{R}^{N}}\mu _{1}\left\vert
u^{+}\right\vert ^{2p+2}\mathrm{d}x+\int_{\mathbb{R}^{N}}\mu _{2}\left\vert
v^{+}\right\vert ^{2p+2}\right) \mathrm{d}x \\
&&-\frac{1}{q+1}\int_{\mathbb{R}^{N}}\beta _{\infty }(x)\left\vert
u^{+}\right\vert ^{q+1}\left\vert v^{+}\right\vert ^{q+1}\mathrm{d}x,
\end{eqnarray*}%
for $(u,v)\in H.$ Let $\left( \widehat{u}_{n},\widehat{v}_{n}\right) =\left(
u_{n}-z_{n}\ast \widetilde{u}_{0},v_{n}-z_{n}\ast \widetilde{v}_{0}\right) ,$
i.e.
\begin{equation*}
\left( -z_{n}\ast \widehat{u}_{n},-z_{n}\ast \widehat{v}_{n}\right) =\left(
\widetilde{u}_{n}-\widetilde{u}_{0},\widetilde{v}_{n}-\widetilde{v}%
_{0}\right) .
\end{equation*}%
By (\ref{3.27}), it is evident that
\begin{equation}
\int_{\mathbb{R}^{N}}\left( \varepsilon ^{2\alpha }\left\vert (-\triangle )^{%
\frac{\alpha }{2}}\left( \widetilde{u}_{n}-\widetilde{u}_{0}\right)
\right\vert ^{2}-\varepsilon ^{2\alpha }\left\vert (-\triangle )^{\frac{%
\alpha }{2}}\widetilde{u}_{n}\right\vert ^{2}+\varepsilon ^{2\alpha
}\left\vert (-\triangle )^{\frac{\alpha }{2}}\widetilde{u}_{0}\right\vert
^{2}\right) \mathrm{d}x=o\left( 1\right)  \label{3.33}
\end{equation}%
and
\begin{equation}
\int_{\mathbb{R}^{N}}\left( \varepsilon ^{2\alpha }\left\vert (-\triangle )^{%
\frac{\alpha }{2}}\left( \widetilde{v}_{n}-\widetilde{v}_{0}\right)
\right\vert ^{2}-\varepsilon ^{2\alpha }\left\vert (-\triangle )^{\frac{%
\alpha }{2}}\widetilde{v}_{n}\right\vert ^{2}+\varepsilon ^{2\alpha
}\left\vert (-\triangle )^{\frac{\alpha }{2}}\widetilde{v}_{0}\right\vert
^{2}\right) \mathrm{d}x=o\left( 1\right) .  \label{3.34}
\end{equation}%
Furthermore, since $-z_{n}\ast \beta \rightarrow \beta _{\infty }(x)$ and $%
-z_{n}\ast V_{l}\rightarrow V_{l}^{0}(x)$ uniformly on bounded subsets for $%
l=1,2$, we may conclude that
\begin{equation}
\int_{\mathbb{R}^{N}}\left[ \left( -z_{n}\ast V_{1}\right) \left( \widetilde{%
u}_{n}-\widetilde{u}_{0}\right) ^{2}-\left( -z_{n}\ast V_{1}\right)
\widetilde{u}_{n}^{2}+V_{1}^{0}(x)\widetilde{u}_{0}^{2}\right] \mathrm{d}%
x=o\left( 1\right)  \label{3.35}
\end{equation}%
and
\begin{equation}
\int_{\mathbb{R}^{N}}\left[ \left( -z_{n}\ast V_{2}\right) \left( \widetilde{%
v}_{n}-\widetilde{v}_{0}\right) ^{2}-\left( -z_{n}\ast V_{2}\right)
\widetilde{v}_{n}^{2}+V_{2}^{0}(x)\widetilde{v}_{0}^{2}\right] \mathrm{d}%
x=o\left( 1\right) .  \label{3.36}
\end{equation}%
On the other hand, it follows from Br\'{e}zis--Lieb Lemma (cf.\cite{BLi})
that
\begin{eqnarray}
\int_{\mathbb{R}^{N}}\left[ \left\vert \left( \widetilde{u}_{n}-\widetilde{u}%
_{0}\right) ^{+}\right\vert ^{2p+2}-\left\vert \widetilde{u}%
_{n}^{+}\right\vert ^{2p+2}+\left\vert \widetilde{u}_{0}^{+}\right\vert
^{2p+2}\right] \mathrm{d}x &=&o\left( 1\right) ,  \label{3.37} \\
\int_{\mathbb{R}^{N}}\left[ \left\vert \left( \widetilde{v}_{n}-\widetilde{v}%
_{0}\right) ^{+}\right\vert ^{2p+2}-\left\vert \widetilde{v}%
_{n}^{+}\right\vert ^{2p+2}+\left\vert \widetilde{v}_{0}^{+}\right\vert
^{2p+2}\right] \mathrm{d}x &=&o\left( 1\right) ,  \label{3.38}
\end{eqnarray}%
and
\begin{eqnarray}
&&\int_{\mathbb{R}^{N}}\left( -z_{n}\ast \beta \right) \left\vert \left(
\widetilde{u}_{n}-\widetilde{u}_{0}\right) ^{+}\right\vert ^{q+1}\left\vert
\left( \widetilde{v}_{n}-\widetilde{v}_{0}\right) ^{+}\right\vert ^{q+1}%
\mathrm{d}x  \notag \\
&=&\int_{\mathbb{R}^{N}}\left( -z_{n}\ast \beta \right) \left\vert
\widetilde{u}_{n}^{+}\right\vert ^{q+1}\left\vert \widetilde{v}%
_{n}^{+}\right\vert ^{q+1}\mathrm{d}x+\int_{\mathbb{R}^{N}}\beta _{\infty
}(x)\left\vert \widetilde{u}_{0}^{+}\right\vert ^{q+1}\left\vert \widetilde{v%
}_{0}^{+}\right\vert ^{q+1}\mathrm{d}x+o\left( 1\right) .  \label{3.39}
\end{eqnarray}%
By (\ref{3.30}), (\ref{3.31}) and $\left( \ref{3.33}\right) -\left( \ref%
{3.39}\right) $, we have%
\begin{eqnarray}
&&\int_{\mathbb{R}^{N}}\left( \varepsilon ^{2\alpha }\left\vert (-\triangle
)^{\frac{\alpha }{2}}\widehat{u}_{n}\right\vert ^{2}+V_{1}(x)\widehat{u}%
_{n}^{2}\right) \mathrm{d}x=\mu _{1}\int_{\mathbb{R}^{N}}\left\vert \widehat{%
u}_{n}^{+}\right\vert ^{2p+2}\mathrm{d}x  \notag \\
&&+\int_{\mathbb{R}^{N}}\beta (x)\left\vert \widehat{u}_{n}^{+}\right\vert
^{q+1}\left\vert \widehat{v}_{n}^{+}\right\vert ^{q+1}\mathrm{d}x+o\left(
1\right) ,  \label{3.40}
\end{eqnarray}%
\begin{eqnarray}
&&\int_{\mathbb{R}^{N}}\left( \varepsilon ^{2\alpha }\left\vert (-\triangle
)^{\frac{\alpha }{2}}\widehat{v}_{n}\right\vert ^{2}+V_{2}\widehat{v}%
_{n}^{2}\right) \mathrm{d}x=\mu _{2}\int_{\mathbb{R}^{N}}\left\vert \widehat{%
v}_{n}^{+}\right\vert ^{2p+2}\mathrm{d}x  \notag \\
&&+\int_{\mathbb{R}^{N}}\beta (x)\left\vert \widehat{u}_{n}^{+}\right\vert
^{q+1}\left\vert \widehat{v}_{n}^{+}\right\vert ^{q+1}\mathrm{d}x+o\left(
1\right)  \label{3.41}
\end{eqnarray}%
and
\begin{equation}
J_{\varepsilon }\left( \widehat{u}_{n},\widehat{v}_{n}\right)
=J_{\varepsilon }\left( u_{n},v_{n}\right) -I_{\varepsilon }\left(
u_{0},v_{0}\right) +o(1)=\varepsilon ^{N}\theta _{\varepsilon
}-I_{\varepsilon }\left( u_{0},v_{0}\right) +o(1).  \label{3.42}
\end{equation}%
Now we want to show the strong convergence of $\left\Vert \left( \widehat{u}%
_{n},\widehat{v}_{n}\right) \right\Vert _{H}\rightarrow 0$ as $n\rightarrow
\infty .$ Suppose not, then $\left\Vert \left( \widehat{u}_{n},\widehat{v}%
_{n}\right) \right\Vert _{H}\geq c_{0}$ for large $n$, where $c_{0}$ is a
positive constant independent of $n$. Hence, it follows from (\ref{3.40}), (%
\ref{3.41}), Lemma \ref{L2.1} and Lemma \ref{L2.3} that
\begin{equation}
J_{\varepsilon }\left( \widehat{u}_{n},\widehat{v}_{n}\right) \geq \frac{%
\varepsilon ^{N}}{2}\min \left\{ \overline{\delta },\widehat{\alpha }_{1},%
\widehat{\alpha }_{2}\right\} ,  \label{3.43}
\end{equation}%
for $n$ sufficiently large. Thus, in view of (\ref{3.32}), (\ref{3.42}) and (%
\ref{3.43}), we obtain
\begin{equation*}
\varepsilon ^{N}\theta _{\varepsilon }\geq \varepsilon ^{N}\left( \widehat{%
\alpha }_{1}+\widehat{\alpha }_{2}\right) +\frac{\varepsilon ^{N}}{2}\min
\left\{ \overline{\delta },\widehat{\alpha }_{1},\widehat{\alpha }%
_{2}\right\} .
\end{equation*}%
which is a contradiction with $\theta _{\varepsilon }<\widehat{\alpha }_{1}+%
\widehat{\alpha }_{2}+\delta _{0}$. Hence, the strong convergence $%
\left\Vert \left( \widehat{u}_{n},\widehat{v}_{n}\right) \right\Vert
_{H}\rightarrow 0$ as $n\rightarrow \infty $ holds, i.e.
\begin{equation}
\left\Vert \left( u_{n}-z_{n}\ast \widetilde{u}_{0},v_{n}-z_{n}\ast
\widetilde{v}_{0}\right) \right\Vert _{H}\rightarrow 0\text{ as }%
n\rightarrow \infty .  \label{3.44}
\end{equation}%
It follows from (\ref{1.11}) and (\ref{3.44}) that
\begin{equation*}
o\left( 1\right) =\Phi \left( u_{n}\right) -\Phi \left( z_{n}\ast \widetilde{%
u}_{0}\right) =\Phi \left( u_{n}\right) -z_{n}-\Phi \left( \widetilde{u}%
_{0}\right) ,
\end{equation*}%
i.e.
\begin{equation}
z_{n}=\Phi \left( u_{n}\right) -\Phi \left( \widetilde{u}_{0}\right)
+o\left( 1\right) .  \label{3.45}
\end{equation}%
Since $\Phi \left( u_{n}\right) \in C_{s}\left( z_{1,i}\right) $ for $n\in
\mathbb{N}$, then (\ref{3.45}) implies that $\{z_{n}\}$ is a bounded
sequence in $\mathbb{R}^{N}$. As a result, we may assume $z_{n}\rightarrow
z_{0}\in \mathbb{R}^{N}$ as $n\rightarrow \infty $. Thus, (\ref{3.44})
becomes
\begin{equation}
\left( u_{n},v_{n}\right) \rightarrow \left( z_{0}\ast \widetilde{u}%
_{0},z_{0}\ast \widetilde{v}_{0}\right) \text{ strongly in }H.  \label{3.46}
\end{equation}%
Since $\left( u_{n},v_{n}\right) \in N_{i,j}\left( \varepsilon \right) $ for
$n\in \mathbb{N}$, then by (\ref{3.46}), we have $\left( z_{0}\ast
\widetilde{u}_{0},z_{0}\ast \widetilde{v}_{0}\right) \in N_{i,j}\left(
\varepsilon \right) \cup O_{i,j}\left( \varepsilon \right) .$ In view of the
assumption $\left( i\right) $, (\ref{3.2}) and (\ref{3.46}), we obtain
\begin{equation*}
J_{\varepsilon }\left( z_{0}\ast \widetilde{u}_{0},z_{0}\ast \widetilde{v}%
_{0}\right) =\varepsilon ^{N}\theta _{\varepsilon }<\varepsilon ^{N}\left(
\widehat{\alpha }_{1}+\widehat{\alpha }_{2}+\delta _{0}\right) <\widetilde{%
\gamma }\left( \varepsilon \right) ,
\end{equation*}%
so $\left( z_{0}\ast \widetilde{u}_{0},z_{0}\ast \widetilde{v}_{0}\right)
\notin O_{i,j}\left( \varepsilon \right) ,$ i.e., $\left( z_{0}\ast
\widetilde{u}_{0},z_{0}\ast \widetilde{v}_{0}\right) \in N_{i,j}\left(
\varepsilon \right) $. Therefore, we may complete the proof of Proposition %
\ref{P3.2}.
\end{proof}

\begin{theorem}
\label{t3.1}For each $0<\varepsilon <\varepsilon _{0},1\leq i\leq k$ and $%
1\leq j\leq \ell ,$ problem $\left( P_{\varepsilon }\right) $ has a positive
solution $\left( \widehat{u}_{\varepsilon ,i},\widehat{v}_{\varepsilon
,j}\right) \in N_{i,j}\left( \varepsilon \right) .$
\end{theorem}

\begin{proof}
From Propositions \ref{P3.1} and \ref{P3.2}, we can prove Theorem \ref{t3.1}.
\end{proof}

\section{Proof of Theorem \protect\ref{t1.1}}

Throughout this section, we define
\begin{equation*}
\left[ \varepsilon ^{-N}\,J_{\varepsilon }\leq d\right] _{i,j}=\left\{
\left( u,v\right) \in N_{i,j}\left( \varepsilon \right) :\varepsilon
^{-N}J_{\varepsilon }\left( u,v\right) \leq d\right\} \,~\text{for }d\in
\mathbb{R}.
\end{equation*}%
The main goal of this section is to show
\begin{equation*}
\mathrm{cat}\left( \left[ \varepsilon ^{-N}\,J_{\varepsilon }\leq \widehat{%
\alpha }_{1}+\widehat{\alpha }_{2}+\delta _{\varepsilon }\right]
_{i,j}\right) \geq 2~\text{ if }z_{1,i}=z_{2,j},
\end{equation*}%
where $0<\delta _{\varepsilon }\rightarrow 0$ as $\varepsilon \rightarrow
0^{+}$. Here $cat(\cdot )$ is the standard Lusternik--Schnirelman category
(cf.~\cite{B1}).

From Lemma 4.1 in \cite{Lin1} and Lemma 2.2 in \cite{LW}, we obtain

\begin{lemma}
\label{L4.1}Let $\left( u_{\varepsilon },v_{\varepsilon }\right) $ be a
constrained critical point of $J_{\varepsilon }$ on $\mathbf{N}_{\varepsilon
}$ with
\begin{equation*}
\varepsilon ^{-N}\,J_{\varepsilon }\left( u_{\varepsilon },v_{\varepsilon
}\right) \leq \widehat{\alpha }_{1}+\widehat{\alpha }_{2}+\delta _{0}.
\end{equation*}%
Then $\nabla J_{\varepsilon }\left( u_{\varepsilon },v_{\varepsilon }\right)
=0$ on $H^{\ast }.$
\end{lemma}

Therefore, from Theorem~2.3 in \cite{Am}, Proposition \ref{P3.2} and Lemma %
\ref{L4.1}, we have

\begin{proposition}
\label{P4.1} Suppose $cat\left( M_{i,j}(\varepsilon ,\delta _{0})\right)
\geq k,$ where $k\in \mathbb{N}$ and
\begin{equation*}
M_{i,j}(\varepsilon ,\delta _{0})=\left[ \varepsilon ^{-N}J_{\varepsilon
}\leq \widehat{\alpha }_{1}+\widehat{\alpha }_{2}+\delta _{0}\right] _{i,j}.
\end{equation*}%
Then the functional $J_{\varepsilon }$ has at least $k$ critical points in $%
M_{i,j}(\varepsilon ,\delta _{0})$.
\end{proposition}

Now we define a function by
\begin{equation}
h(u,v)=\Phi \left( u\right) -\Phi \left( v\right) \text{ for }(u,v)\in H.
\label{4.1}
\end{equation}%
This leads to the following results.

\begin{lemma}
\label{L4.2}Suppose that $z_{1,i}=z_{2,j}.$ Then there exist $\widetilde{%
\varepsilon }_{0}>0$ and $0<\delta _{1}\leq \delta _{0}$ such that for every
$0<\varepsilon <\widetilde{\varepsilon }_{0},$ there holds
\begin{equation*}
\left\vert h\left( u,v\right) \right\vert >0\text{ for all }\left(
u,v\right) \in \left[ \varepsilon ^{-N}J_{\varepsilon }\leq \widehat{\alpha }%
_{1}+\widehat{\alpha }_{2}+\delta _{1}\right] _{i,j}.
\end{equation*}
\end{lemma}

\begin{proof}
We may prove this by contradiction. Suppose that there exists $\varepsilon
_{n}>0$, for $n\in \mathbb{N},\varepsilon _{n}\rightarrow 0$ as $%
n\rightarrow \infty $ and $\left( u_{n},v_{n}\right) \in N_{i,j}\left(
\varepsilon _{n}\right) $ such that
\begin{equation}
h(u_{n},v_{n})=0,\text{ for all }n\in \mathbb{N},  \label{4.2}
\end{equation}%
and
\begin{equation}
\varepsilon _{n}^{-N}J_{\varepsilon _{n}}\left( u_{n},v_{n}\right) =\widehat{%
\alpha }_{1}+\widehat{\alpha }_{2}+o\left( 1\right) .  \label{4.3}
\end{equation}%
To obtain a contradiction with (\ref{4.2}), it is sufficient to show that
\begin{equation}
\frac{1}{\varepsilon _{n}}\left\vert h(u_{n},v_{n})\right\vert \rightarrow
\infty \text{ as }n\rightarrow \infty .  \label{4.4}
\end{equation}%
Then $\Phi \left( u_{n}\right) =\Phi \left( v_{n}\right) $ and
\begin{equation*}
\Phi \left( u_{n}\right) ,\Phi \left( v_{n}\right) \in C_{s}\left(
z_{1,i}\right) =C_{s}\left( z_{2,j}\right) .
\end{equation*}%
Let $\widetilde{u}_{n}\left( x\right) =u_{n}\left( \varepsilon _{n}x\right) $
and $\widetilde{v}_{n}\left( x\right) =v_{n}\left( \varepsilon _{n}x\right)
. $ Clearly, $\Phi \left( \widetilde{u}_{n}\right) =\Phi \left( \widetilde{v}%
_{n}\right) $ and $\Phi \left( \widetilde{u}_{n}\right) ,\Phi \left(
\widetilde{v}_{n}\right) \in C_{s/\varepsilon _{n}}\left( z_{1,i}\right) .$
By a similar argument to the proof of Lemma \ref{L2.5} and the proof of
Lemma 3.3 in \cite{LWW}, we can obtain%
\begin{equation*}
\int_{\mathbb{R}^{N}}(V_{1}\left( \varepsilon _{n}x\right) -\lambda _{1})%
\widetilde{u}_{n}^{2}\mathrm{d}x=o\left( 1\right) \text{ and }\int_{\mathbb{R%
}^{N}}(V_{2}\left( \varepsilon _{n}x\right) -\lambda _{2})\widetilde{v}%
_{n}^{2}\mathrm{d}x=o\left( 1\right) .
\end{equation*}%
Moreover, we can conclude that $\widehat{I}_{1}\left( \widetilde{u}%
_{n}\right) =\widehat{\alpha }_{1}+o\left( 1\right) ,\widehat{I}_{2}\left(
\widetilde{v}_{n}\right) =\widehat{\alpha }_{2}+o\left( 1\right) $ and%
\begin{eqnarray*}
\int_{\mathbb{R}^{N}}(\left\vert (-\triangle )^{\frac{\alpha }{2}}\widetilde{%
u}_{n}\right\vert ^{2}+\lambda _{1}\widetilde{u}_{n}^{2})\mathrm{d}x &=&\mu
_{1}\int_{\mathbb{R}^{N}}\left\vert \widetilde{u}_{n}^{+}\right\vert ^{2p+2}%
\mathrm{d}x+o\left( 1\right) , \\
\int_{\mathbb{R}^{N}}(\left\vert (-\triangle )^{\frac{\alpha }{2}}\widetilde{%
v}_{n}\right\vert ^{2}+\lambda _{2}\widetilde{v}_{n}^{2})\mathrm{d}x &=&\mu
_{2}\int_{\mathbb{R}^{N}}\left\vert \widetilde{v}_{n}^{+}\right\vert ^{2p+2}%
\mathrm{d}x+o\left( 1\right) .
\end{eqnarray*}%
Thus, there exist $t_{n},s_{n}>0$ with $t_{n},s_{n}\rightarrow 1$ as $%
n\rightarrow \infty $ such that $\overline{u}_{n}=t_{n}\widetilde{u}_{n}\in
\widehat{\mathbf{M}}_{1}$ and $\overline{v}_{n}=s_{n}\widetilde{v}_{n}\in
\widehat{\mathbf{M}}_{2},$ which implies that $\widehat{I}_{1}\left(
\overline{u}_{n}\right) =\widehat{\alpha }_{1}+o\left( 1\right) $ and $%
\widehat{I}_{2}\left( \overline{v}_{n}\right) =\widehat{\alpha }_{2}+o\left(
1\right) $. Furthermore, by the Ekeland varitional principle \cite{E}, we
may assume that $\left\{ \overline{u}_{n}\right\} $ is a (PS)$_{\widehat{%
\alpha }_{1}}$--sequence of $\widehat{I}_{1}$ in $H^{\alpha }\left( \mathbb{R%
}^{N}\right) $ and that $\left\{ \overline{v}_{n}\right\} $ is a (PS)$_{%
\widehat{\alpha }_{2}}$--sequence of $\widehat{I}_{2}$ in $H^{\alpha }\left(
\mathbb{R}^{N}\right) $, respectively. Applying the
concentration-compactness principle of Lions \cite{Li1,Li2}, there are
positive constants $R,b_{0}$ and two sequences $\left\{ x_{n}\right\}
,\left\{ y_{n}\right\} \subset \mathbb{R}^{N}$ such that
\begin{equation}
\int_{B^{N}\left( x_{n};R\right) }\left\vert \overline{u}_{n}^{+}\right\vert
^{2p+2}\mathrm{d}x\geq b_{0}~\text{ and }\int_{B^{N}\left( y_{n};R\right)
}\left\vert \overline{v}_{n}^{+}\right\vert ^{2p+2}\mathrm{d}x\geq b_{0}~%
\text{ for all }n\in \mathbb{N}.  \label{4.5}
\end{equation}%
Let $\widehat{u}_{n}(x)=\overline{u}_{n}\left( x+x_{n}\right) $ for $x\in
\mathbb{R}^{N}$ and $n\in \mathbb{N}$. Then, due to translation invariance,
it is obvious that $\left\{ \widehat{u}_{n}\right\} $ is also a (PS) $_{%
\widehat{\alpha }_{1}}$--sequence of $\widehat{I}_{1}$ in $H^{\alpha }\left(
\mathbb{R}^{N}\right) .$ Hence by $\left( \ref{4.5}\right) $, we may assume
that there exists a subsequence of $\left\{ \widehat{u}_{n}\right\} $ such
that
\begin{eqnarray}
\widehat{u}_{n} &\rightharpoonup &u_{0}~~\text{ in }H^{\alpha }\left(
\mathbb{R}^{N}\right) ,  \notag \\
\widehat{u}_{n} &\rightarrow &u_{0}~~\text{ a.e. in }\mathbb{R}^{N},
\label{4.6} \\
\int_{B^{N}\left( 0;R\right) }\left\vert \widehat{u}_{n}^{+}\right\vert
^{2p+2}\mathrm{d}x &\rightarrow &\int_{B^{N}\left( 0;R\right) }\left\vert
u_{0}^{+}\right\vert ^{2p+2}\mathrm{d}x\geq b_{0}.  \notag
\end{eqnarray}%
Now we set $w_{n}=\widehat{u}_{n}-u_{0}.$ Then it follows from $(\ref{4.6})$
and Br\'{e}zis--Lieb Lemma \cite{BLi} that
\begin{equation}
\int_{\mathbb{R}^{N}}\left\vert \widehat{u}_{n}^{+}\right\vert ^{2p+2}%
\mathrm{d}x=\int_{\mathbb{R}^{N}}\left\vert u_{0}^{+}\right\vert ^{2p+2}%
\mathrm{d}x+\int_{\mathbb{R}^{N}}\left\vert w_{n}^{+}\right\vert ^{2p+2}%
\mathrm{d}x+o\left( 1\right)  \label{4.7}
\end{equation}%
and
\begin{eqnarray}
&&\int_{\mathbb{R}^{N}}(\left\vert (-\triangle )^{\frac{\alpha }{2}}\widehat{%
u}_{n}\right\vert ^{2}+\lambda _{1}\widehat{u}_{n}^{2})\mathrm{d}x=\int_{%
\mathbb{R}^{N}}(\left\vert (-\triangle )^{\frac{\alpha }{2}}u_{0}\right\vert
^{2}+\lambda _{1}u_{0}^{2})\mathrm{d}x  \label{4.8} \\
&&+\int_{\mathbb{R}^{N}}(\left\vert (-\triangle )^{\frac{\alpha }{2}%
}w_{n}\right\vert ^{2}+\lambda _{1}w_{n}^{2})\mathrm{d}x+o\left( 1\right) .
\notag
\end{eqnarray}%
Combining $\left( \ref{4.7}\right) ,\left( \ref{4.8}\right) $ and $\left\{
\widehat{u}_{n}\right\} $ is a (PS)$_{\widehat{\alpha }_{1}}$--sequence of $%
\widehat{I}_{1}$ in $H^{\alpha }\left( \mathbb{R}^{N}\right) ,$ we obtain
\begin{equation*}
\int_{\mathbb{R}^{N}}(\left\vert (-\triangle )^{\frac{\alpha }{2}%
}w_{n}\right\vert ^{2}+\lambda _{1}w_{n}^{2})\mathrm{d}x=\mu _{1}\int_{%
\mathbb{R}^{N}}\left\vert w_{n}^{+}\right\vert ^{2p+2}\mathrm{d}x+o\left(
1\right)
\end{equation*}%
and
\begin{equation*}
\frac{p}{2p+2}\int_{\mathbb{R}^{N}}(\left\vert (-\triangle )^{\frac{\alpha }{%
2}}w_{n}\right\vert ^{2}+\lambda _{1}w_{n}^{2})\mathrm{d}x=\widehat{I}%
_{1}\left( w_{n}\right) +o(1)=\widehat{I}_{1}\left( \widehat{u}_{n}\right) -%
\widehat{I}_{1}\left( u_{0}\right) +o\left( 1\right) =o\left( 1\right) .
\end{equation*}%
Consequently, $\widehat{u}_{n}\rightarrow u_{0}$ strongly in $H^{\alpha
}\left( \mathbb{R}^{N}\right) $ and $\widehat{I}_{1}\left( u_{0}\right) =%
\widehat{\alpha }_{1}.$ Similarly, $\widehat{v}_{n}\rightarrow v_{0}$
strongly in $H^{\alpha }\left( \mathbb{R}^{N}\right) $ and $\widehat{I}%
_{2}\left( v_{0}\right) =\widehat{\alpha }_{2},$ where $\widehat{v}_{n}(x)=%
\overline{v}_{n}\left( x+y_{n}\right) $ for $x\in \mathbb{R}^{N}$ and $n\in
\mathbb{N}$. Therefore, we obtain
\begin{equation}
\widehat{u}_{n}(x)\rightarrow u_{0}\left( x\right) ;\ \widehat{v}%
_{n}(x)\rightarrow v_{0}\left( x\right) ~\text{strongly in }H^{\alpha
}\left( \mathbb{R}^{N}\right) ,  \label{4.9}
\end{equation}%
as $n\rightarrow \infty $, where $u_{0}$ and $v_{0}$ are positive solutions
of Eq. $\left( \widehat{E}_{1}\right) $ and Eq. $\left( \widehat{E}%
_{2}\right) ,$ respectively (cf. \cite{BeL} and \cite{FLS}). From $\Phi
\left( \widetilde{u}_{n}\right) ,\Phi \left( \widetilde{v}_{n}\right) \in
C_{s/\varepsilon _{n}}\left( 0\right) $ $\widehat{u}_{n}(x)=\overline{u}%
_{n}\left( x+x_{n}\right) ,\widehat{v}_{n}(x)=\overline{v}_{n}\left(
x+y_{n}\right) $ and $\left( \ref{4.9}\right) ,$ we have%
\begin{equation*}
\varepsilon _{n}x_{n}=\varepsilon _{n}\Phi \left( \overline{u}_{n}\right)
-\varepsilon _{n}\Phi \left( \widehat{u}_{n}\right) =\varepsilon _{n}\Phi
\left( \widetilde{u}_{n}\right) -\varepsilon _{n}\Phi \left( u_{0}\right)
\end{equation*}%
and%
\begin{equation*}
\varepsilon _{n}y_{n}=\varepsilon _{n}\Phi \left( \overline{v}_{n}\right)
-\varepsilon _{n}\Phi \left( \widehat{v}_{n}\right) =\varepsilon _{n}\Phi
\left( \widetilde{v}_{n}\right) -\varepsilon _{n}\Phi \left( v_{0}\right) ,
\end{equation*}%
and so
\begin{equation*}
dist\left( \varepsilon _{n}x_{n},C_{s}\left( z_{1,i}\right) \right)
\rightarrow 0\text{ and }dist\left( \varepsilon _{n}y_{n},C_{s}\left(
z_{2,j}\right) \right) \rightarrow 0\text{ as }n\rightarrow \infty .
\end{equation*}%
Without loss of generality, we may assume $\varepsilon _{n}x_{n}\rightarrow
x_{0}\in \overline{C_{s}\left( z_{1,i}\right) },\varepsilon
_{n}y_{n}\rightarrow y_{0}\in \overline{C_{s}\left( z_{2,j}\right) }.$ By
condition $\left( D_{2}\right) ,$ we must have that
\begin{equation*}
V_{1}\left( x_{0}\right) =\lambda _{1}=V_{1}\left( z_{1,i}\right)
\end{equation*}%
and%
\begin{equation*}
V_{2}\left( y_{0}\right) =\lambda _{2}=V_{2}\left( z_{2,j}\right) ,
\end{equation*}%
this implies that $x_{0}=y_{0}=z_{1,i}=z_{2,j}.$ Moreover,
\begin{equation*}
\Phi \left( u_{0}\right) =\Phi \left( -x_{n}\ast \overline{u}_{n}\right)
+o\left( 1\right) =-x_{n}+\Phi \left( \widetilde{u}_{n}\right) +o\left(
1\right)
\end{equation*}%
and%
\begin{equation*}
\Phi \left( v_{0}\right) =\Phi \left( -y_{n}\ast \overline{v}_{n}\right)
+o\left( 1\right) =-y_{n}+\Phi \left( \widetilde{v}_{n}\right) +o\left(
1\right) ,
\end{equation*}%
which implies that%
\begin{equation}
x_{n}-y_{n}\rightarrow z_{0}:=\Phi \left( u_{0}\right) -\Phi \left(
v_{0}\right) \text{ as }n\rightarrow \infty .  \label{4.10}
\end{equation}%
It follows from condition $\left( D_{4}\right) ,$ $\varepsilon
_{n}x_{n}\rightarrow 0$ as $n\rightarrow \infty ,\left( \ref{2.34}\right) $
and $(\ref{4.9})$ that
\begin{eqnarray}
\int_{\mathbb{R}^{N}}\beta \left( \varepsilon _{n}x\right) \left\vert
\overline{u}_{n}^{+}\right\vert ^{q+1}\left\vert \overline{v}%
_{n}^{+}\right\vert ^{q+1}\mathrm{d}x &=&\int_{\mathbb{R}^{N}}\beta \left(
\varepsilon _{n}x\right) \left\vert \widehat{u}_{n}^{+}\left( x-x_{n}\right)
\right\vert ^{q+1}\left\vert \widehat{v}_{n}^{+}\left( x-y_{n}\right)
\right\vert ^{q+1}\mathrm{d}x  \notag \\
&=&\int_{\mathbb{R}^{N}}\beta \left( \varepsilon _{n}x+\varepsilon
_{n}x_{n}\right) \left\vert \widehat{u}_{n}^{+}\left( x\right) \right\vert
^{q+1}\left\vert \widehat{v}_{n}^{+}\left( x+x_{n}-y_{n}\right) \right\vert
^{q+1}\mathrm{d}x  \notag \\
&\leq &\frac{1}{2}\int_{C_{s/\varepsilon _{n}}\left( z_{1,i}\right) }\beta
\left( \varepsilon _{n}x+z_{1,i}\right) \left\vert u_{0}^{+}\left( x\right)
\right\vert ^{q+1}\left\vert v_{0}^{+}\left( x+z_{0}\right) \right\vert
^{q+1}  \notag \\
&\leq &-\frac{c_{0}}{2}\int_{\mathbb{R}^{N}}\left\vert u_{0}^{+}\left(
x\right) \right\vert ^{q+1}\left\vert v_{0}^{+}\left( x+z_{0}\right)
\right\vert ^{q+1}\mathrm{d}x  \label{4.11}
\end{eqnarray}%
for $n$ sufficiently large. However, by (\ref{4.3}) and Lemma \ref{L2.2}, we
obtain
\begin{equation*}
\int_{\mathbb{R}^{N}}\beta \left( \varepsilon _{n}x\right) \left\vert
\overline{u}_{n}^{+}\right\vert ^{q+1}\left\vert \overline{v}%
_{n}^{+}\right\vert ^{q+1}\mathrm{d}x\rightarrow 0\text{ as }n\rightarrow
\infty ,
\end{equation*}%
which is a contradiction with (\ref{4.11}). This completes the proof.
\end{proof}

For $0<\varepsilon <\widetilde{\varepsilon }_{0}$, a map $F_{\varepsilon
}^{\left( i,j\right) }:S^{N-1}\rightarrow H$ can be written as
\begin{equation*}
F_{\varepsilon }^{\left( i,j\right) }\left( e\right) =\left(
t_{\varepsilon,i }u_{\varepsilon ,i},s_{\varepsilon,j }v_{\varepsilon
,j}\right) \text{ for }e\in S^{N-1},
\end{equation*}%
where $\widetilde{\varepsilon }_{0}$ is given by Lemma~\ref{4.3}. Here $%
\left( t_{\varepsilon,i }u_{\varepsilon ,i},s_{\varepsilon,j }v_{\varepsilon
,j}\right) $ is as in the proof of Lemma \ref{L2.4}. Note that $\left(
t_{\varepsilon,i },s_{\varepsilon,j }\right) \rightarrow (1,1)$ as $%
\varepsilon \rightarrow 0^{+}$ and by $\left( \ref{2.28}\right) $%
\begin{equation}
\varepsilon ^{-N}J_{\varepsilon }\left( t_{\varepsilon,i }u_{\varepsilon
,i},s_{\varepsilon,j }v_{\varepsilon ,j}\right) \rightarrow \widehat{\alpha }%
_{1}+\widehat{\alpha }_{2}\text{ as }\varepsilon \rightarrow 0^{+}\text{
uniformly in }e\in S^{N-1}.  \label{4.14}
\end{equation}%
Then, from (\ref{4.14}) and Lemma~\ref{L2.3}, we obtain
\begin{equation}
\varepsilon ^{-N}J_{\varepsilon }\left( F_{\varepsilon }^{\left( i,j\right)
}(e)\right) <\widehat{\alpha }_{1}+\widehat{\alpha }_{2}+\delta
_{\varepsilon }\text{ uniformly\ in }e\in S^{N-1},  \label{4.15}
\end{equation}%
where
\begin{equation}
0<\delta _{\varepsilon }=2\max_{e\in S^{N-1}}\left[ \varepsilon
^{-N}\,J_{\varepsilon }\left( F_{\varepsilon }^{\left( i,j\right)
}(e)\right) -\widehat{\alpha }_{1}-\widehat{\alpha }_{2}\right] \rightarrow 0%
\text{ as }\varepsilon \rightarrow 0.  \label{4.16}
\end{equation}%
Here we have used the facts that the map $F_{\varepsilon }$ is continuous
and that the set $F_{\varepsilon }\left( S^{N-1}\right) $ is compact. As a
result, (\ref{4.15}) and (\ref{4.16}) imply:

\begin{lemma}
\label{L4.3}Suppose that $z_{1,i}=z_{2,j}.$ Then there exists $0<\sigma
_{\varepsilon }<\delta _{\varepsilon }$ such that
\begin{equation*}
F_{\varepsilon }^{\left( i,j\right) }\left( S^{N-1}\right) \subset \left[
\varepsilon ^{-N}J_{\varepsilon }\leq \widehat{\alpha }_{1}+\widehat{\alpha }%
_{2}+\delta _{\varepsilon }-\sigma _{\varepsilon }\right] _{i,j}.
\end{equation*}
\end{lemma}

By (\ref{4.16}) and Lemma~\ref{L4.2}, we may define
\begin{equation*}
G_{\varepsilon }^{\left( i,j\right) }:\left[ \varepsilon ^{-N}J_{\varepsilon
}\leq \widehat{\alpha }_{1}+\widehat{\alpha }_{2}+\delta _{\varepsilon }%
\right] _{i,j}\rightarrow S^{N-1}
\end{equation*}%
by
\begin{equation*}
G_{\varepsilon }^{\left( i,j\right) }\left( u,v\right) =-\frac{h(u,v)}{%
\left\vert h(u,v)\right\vert }\text{ for }(u,v)\in \left[ \varepsilon
^{-N}\,J_{\varepsilon }\leq \widehat{\alpha }_{1}+\widehat{\alpha }%
_{2}+\delta _{\varepsilon }\right] _{i,j}.
\end{equation*}%
Then we can prove the following results.

\begin{lemma}
\label{L4.4}Suppose that $z_{1,i}=z_{2,j}.$ Then there exists a positive
number $\varepsilon _{\ast }\leq \widetilde{\varepsilon }_{0}$ such that for
$0<\varepsilon <\varepsilon _{\ast },$ the map $G_{\varepsilon }^{\left(
i,j\right) }\circ F_{\varepsilon }^{\left( i,j\right) }:S^{N-1}\rightarrow
S^{N-1}$ is homotopic to the identity.
\end{lemma}

\begin{proof}
Let $\overline{G}_{\varepsilon }^{\left( i,j\right) }:\Theta \rightarrow
S^{N-1}$ satisfy
\begin{equation*}
\overline{G}_{\varepsilon }^{\left( i,j\right) }\left( u,v\right) =-\frac{%
h(u,v)}{\left\vert h(u,v)\right\vert }\text{ for }(u,v)\in \Theta ,
\end{equation*}%
where $\Theta =\left\{ \left( u,v\right) \in H\backslash \left\{ 0,0\right\}
:\left\vert h(u,v)\right\vert >0\right\} .$ By Lemma~\ref{L4.2}, one may
regard the map $\overline{G}_{\varepsilon }^{\left( i,j\right) }$ as an
extension of $G_{\varepsilon }^{\left( i,j\right) }$. By $\left(
t_{\varepsilon,i },s_{\varepsilon,j }\right) \rightarrow (1,1)$ as $%
\varepsilon \rightarrow 0^{+}$, for $\theta \in \left[ 0,1/2\right) $,
\begin{equation*}
\left( 1-2\theta \right) F_{\varepsilon }^{\left( i,j\right) }\left(
e\right) +2\theta \left( u_{\varepsilon ,i},v_{\varepsilon ,j}\right)
=\left( u_{\varepsilon ,i},v_{\varepsilon ,j}\right) +o_{\varepsilon }\left(
1\right) \text{ in }H^{\alpha }\left( \mathbb{R}^{N}\right) \times H^{\alpha
}\left( \mathbb{R}^{N}\right) \text{ as }\varepsilon \rightarrow 0^{+}.
\end{equation*}%
Hence as for the proof of Lemma~\ref{L4.2}, there exists a positive number $%
\varepsilon _{\ast }\leq \widetilde{\varepsilon }_{0}$ such that for $%
0<\varepsilon <\varepsilon _{\ast },$ there holds,
\begin{equation}
\left( 1-2\theta \right) F_{\varepsilon }^{\left( i,j\right) }\left(
e\right) +2\theta \left( u_{\varepsilon ,i},v_{\varepsilon ,j}\right) \in
\Theta \text{ for all }e\in S^{N-1}\text{ and }\theta \in \left[
0,1/2\right) .  \label{4.17}
\end{equation}%
On the other hand, for $e\in S^{N-1}$ and $\theta \in \left[ 1/2,1\right) $,
we may set
\begin{equation*}
f_{\varepsilon ,i}\left( x,\theta \right) =\omega _{1}\left( \frac{x-z_{1,i}%
}{\varepsilon }+\frac{e}{4\sqrt{\varepsilon }\left( 1-\theta \right) }%
\right) \psi _{\varepsilon }\left( \frac{x-z_{1,i}}{\varepsilon }+\frac{e}{4%
\sqrt{\varepsilon }\left( 1-\theta \right) }\right)
\end{equation*}%
and
\begin{equation*}
g_{\varepsilon ,j}\left( x,\theta \right) =\omega _{2}\left( \frac{x-z_{2,j}%
}{\varepsilon }-\frac{e}{4\sqrt{\varepsilon }\left( 1-\theta \right) }%
\right) \psi _{\varepsilon }\left( \frac{x-z_{2,j}}{\varepsilon }-\frac{e}{4%
\sqrt{\varepsilon }\left( 1-\theta \right) }\right) ,
\end{equation*}%
where $z_{1,i}$ and $z_{2,j}$ are minimum points of $V_{1}(x)$ and $V_{2}(x)$%
, respectively. Then it is easy to verify that
\begin{equation*}
\Phi (f_{\varepsilon ,i})=z_{1,i}-\frac{\sqrt{\varepsilon }}{4(1-\theta )}%
e+O(\varepsilon )\text{ and }\Phi (g_{\varepsilon ,j})=z_{2,j}+\frac{\sqrt{%
\varepsilon }}{4(1-\theta )}e+O(\varepsilon ),
\end{equation*}%
which implies that
\begin{equation}
h\left( f_{\varepsilon ,i},g_{\varepsilon ,j}\right) =\Phi (f_{\varepsilon
,i})-\Phi (g_{\varepsilon ,j})=z_{1,i}-z_{2,j}-\frac{\sqrt{\varepsilon }}{%
2\left( 1-\theta \right) }e+O(\varepsilon ),  \label{4.18}
\end{equation}%
for $e\in S^{N-1}$ and $\theta \in \left[ 1/2,1\right) $. Consequently, due
to $z_{1,i}=z_{2,j}$,
\begin{equation}
\left( f_{\varepsilon ,i}\left( \cdot ,\theta \right) ,g_{\varepsilon
,j}\left( \cdot ,\theta \right) \right) \in \Theta \text{ for all }e\in
S^{N-1}\text{ and }\theta \in \left[ 1/2,1\right) .  \label{4.19}
\end{equation}%
One may remark that to assure $|h(f_{\varepsilon ,i},g_{\varepsilon ,j})|>0$
for $e\in S^{N-1}$ and $\theta \in \left[ 1/2,1\right) $, it is necessary
for $z_{1,i}$ to be a common minimum point of the $V_{1}(x)$ and $V_{2}(x)$.
Otherwise, if $V_{1}(x)$ and $V_{2}(x)$ have two different minimum points at
$z_{1,i}$ and $z_{2,j}$, one may use (\ref{4.18}) to find some $e\sim \frac{%
z_{1,i}-z_{2,j}}{|z_{1,i}-z_{2,j}|}\in S^{N-1}$ and $\theta \sim 1-\frac{%
\sqrt{\varepsilon }}{2|z_{1,i}-z_{2,j}|}\in \left[ 1/2,1\right) $ such that $%
h(f_{\varepsilon ,i},g_{\varepsilon ,j})=0$ and $(f_{\varepsilon
,i},g_{\varepsilon ,j})\not\in \Theta $.

By (\ref{4.17}) and (\ref{4.19}), we may define
\begin{equation*}
\zeta _{\varepsilon }\left( \theta ,e\right) :\left[ 0,1\right] \times
S^{N-1}\rightarrow S^{N-1}
\end{equation*}%
by
\begin{equation*}
\zeta _{\varepsilon }^{\left( i,j\right) }\left( \theta ,e\right) =\left\{
\begin{array}{ll}
\overline{G}_{\varepsilon }^{\left( i,j\right) }\left( \left( 1-2\theta
\right) F_{\varepsilon }^{\left( i,j\right) }\left( e\right) +2\theta \left(
u_{\varepsilon },v_{\varepsilon }\right) \right) , & \text{for }\theta \in %
\left[ 0,1/2\right) ; \\
\overline{G}_{\varepsilon }^{\left( i,j\right) }\left( f_{\varepsilon
,i}\left( \cdot ,\theta \right) ,g_{\varepsilon ,j}\left( \cdot ,\theta
\right) \right) , & \text{for }\theta \in \left[ 1/2,1\right) ; \\
e, & \text{for }\theta =1.%
\end{array}%
\right.
\end{equation*}%
Then $\zeta _{\varepsilon }^{\left( i,j\right) }\left( 0,e\right) =\overline{%
G}_{\varepsilon }^{\left( i,j\right) }\left( F_{\varepsilon }^{\left(
i,j\right) }\left( e\right) \right) =G_{\varepsilon }^{\left( i,j\right)
}\left( F_{\varepsilon }^{\left( i,j\right) }\left( e\right) \right) $ and $%
\zeta _{\varepsilon }^{\left( i,j\right) }\left( 1,e\right) =e.$ It is easy
to verify that
\begin{equation}
\underset{\theta \rightarrow \frac{1}{2}^{-}}{\lim }\zeta _{\varepsilon
}^{\left( i,j\right) }\left( \theta ,e\right) =\overline{G}_{\varepsilon
}^{\left( i,j\right) }\left( f_{\varepsilon ,i}\left( \cdot ,\frac{1}{2}%
\right) ,g_{\varepsilon ,j}\left( \cdot ,\frac{1}{2}\right) \right) .
\label{4.20}
\end{equation}%
On the other hand, by (\ref{4.18}) and $z_{1,i}=z_{2,j}$, we have
\begin{equation}
\underset{\theta \rightarrow 1^{-}}{\lim }\zeta _{\varepsilon }^{\left(
i,j\right) }\left( \theta ,e\right) =e\text{ for }e\in S^{N-1}.  \label{4.21}
\end{equation}%
Consequently, by (\ref{4.20}) and (\ref{4.21}), we know that $\zeta
_{\varepsilon }^{\left( i,j\right) }\in C\left( \left[ 0,1\right] \times
S^{N-1},S^{N-1}\right) $ and
\begin{equation*}
\zeta _{\varepsilon }^{\left( i,j\right) }\left( 0,e\right) =G_{\varepsilon
}^{\left( i,j\right) }\left( F_{\varepsilon }^{\left( i,j\right) }\left(
e\right) \right) \quad \text{ for\ all }e\in S^{N-1},
\end{equation*}%
\begin{equation*}
\zeta _{\varepsilon }^{\left( i,j\right) }\left( 1,e\right) ,=e~\quad \quad
\quad \quad \quad \text{ for\ all }e\in S^{N-1},
\end{equation*}%
provided that $0<\varepsilon <\varepsilon _{\ast }.$ Therefore, we may
complete the proof.
\end{proof}

By Lemma~\ref{L4.4} and Adachi--Tanaka's~Lemma 2.5 (cf. \cite{AT1}), we
obtain
\begin{equation*}
\mathrm{cat}\left( \left[ \varepsilon ^{-N}J_{\varepsilon }\leq \widehat{%
\alpha }_{1}+\widehat{\alpha }_{2}+\delta _{\varepsilon }\right]
_{i,j}\right) \geq 2,
\end{equation*}%
for $z_{1,i}=z_{2,j}$ and $0<\varepsilon <\varepsilon _{\ast }$. Thus, by
Proposition~\ref{P4.1}, $J_{\varepsilon }$ has at least two critical points
in $M_{i,j}(\varepsilon ,\delta _{\varepsilon })$, where $%
M_{i,j}(\varepsilon ,\delta _{\varepsilon })=\left[ \varepsilon
^{-N}\,J_{\varepsilon }\leq \widehat{\alpha }_{1}+\widehat{\alpha }%
_{2}+\delta _{\varepsilon }\right] _{i,j}\,.$ This implies that problem $%
\left( P_{\varepsilon }\right) $ has two positive solutions $\left( \widehat{%
u}_{\varepsilon ,i},\widehat{v}_{\varepsilon ,j}\right) ,\left( \widetilde{u}%
_{\varepsilon ,i},\widetilde{v}_{\varepsilon ,j}\right) \in N_{i,j}\left(
\varepsilon \right) \,$ such that $\left\vert h\left( \widehat{u}%
_{\varepsilon ,i},\widehat{v}_{\varepsilon ,j}\right) \right\vert >0$ and $%
\left\vert h\left( \widetilde{u}_{\varepsilon ,i},\widetilde{v}_{\varepsilon
,j}\right) \right\vert >0.$ So we may conclude that:

\begin{theorem}
\label{t4.1} Suppose that $z_{1,i}=z_{2,j}.$ Then for each $0<\varepsilon
<\varepsilon _{\ast },$ problem $\left( P_{\varepsilon }\right) $ has at
least two positive solutions $\left( \widehat{u}_{\varepsilon ,i},\widehat{v}%
_{\varepsilon ,j}\right) ,\left( \widetilde{u}_{\varepsilon ,i},\widetilde{v}%
_{\varepsilon ,j}\right) \in M_{i,j}(\varepsilon ,\delta _{\varepsilon })$
such that
\begin{equation*}
\left\vert h\left( \widehat{u}_{\varepsilon ,i},\widehat{v}_{\varepsilon
,j}\right) \right\vert >0\text{ and }\left\vert h\left( \widetilde{u}%
_{\varepsilon ,i},\widetilde{v}_{\varepsilon ,j}\right) \right\vert >0.
\end{equation*}
\end{theorem}

\noindent \textbf{We are now ready to prove Theorem \ref{t1.1}}. By Theorem %
\ref{t3.1} and Theorem \ref{t4.1}, we can conclude that problem $\left(
P_{\varepsilon }\right) $ has at least $k\times \ell +m$ positive solutions.
As for the proofs of $(\ref{4.9})$ and $(\ref{4.10}),$ we may use $(\ref%
{4.16})$ and a similar argument to that in the proof of \cite[Theorem 2.4]%
{lw1}, we can conclude that part $\left( ii\right) $ is holds.

\section{Proof of Theorem \protect\ref{t1.2}}

By Lemmas \ref{L2.2} and \ref{L2.4}, there exists a positive function $%
\delta _{\varepsilon }$ with $\delta _{\varepsilon }\rightarrow 0$ as $%
\varepsilon \rightarrow 0$ such that the sublevel set
\begin{equation*}
\widehat{M}(\varepsilon ,\delta _{\varepsilon })=\left\{ \left( u,v\right)
\in \mathbf{N}_{\varepsilon }:\varepsilon ^{-N}J_{\varepsilon }\left(
u,v\right) \leq \widehat{\alpha }_{1}+\widehat{\alpha }_{2}+\delta
_{\varepsilon }\right\}
\end{equation*}%
is nonempty. Then we have the following results.

\begin{lemma}
\label{L5.1}Assume that $\liminf\limits_{|x|\rightarrow \infty
}\,V_{l}(x)\equiv V_{l,\infty }>\lambda _{l}$ for all $l=1,2.$ Then%
\begin{equation}
\lim_{\varepsilon \rightarrow 0}\sup_{\left( u,v\right) \in \widehat{M}%
(\varepsilon ,\delta _{\varepsilon })}\inf_{\left( x,y\right) \in \mathbf{B}%
}\left\vert \left( \Phi \left( u\right) ,\Phi \left( v\right) \right)
-\left( x,y\right) \right\vert =0.  \label{5.1}
\end{equation}%
where $\mathbf{B}=\underset{1\leq i\leq k,1\leq j\leq \ell }{\cup }\left[
C_{s/2}\left( z_{1,i}\right) \times C_{s/2}\left( z_{2,j}\right) \right] .$
\end{lemma}

\begin{proof}
Let $\varepsilon _{n}\rightarrow 0$ as $n\rightarrow \infty ;$ for any $n\in
\mathbb{N}$, there exists $\left( u_{n},v_{n}\right) \in \widehat{M}%
(\varepsilon ,\delta _{\varepsilon _{n}})$ such that
\begin{equation*}
\inf_{\left( x,y\right) \in \mathbf{B}}\left\vert \left( \Phi \left(
u_{n}\right) ,\Phi \left( v_{n}\right) \right) -\left( x,y\right)
\right\vert =\sup_{\left( u,v\right) \in \widehat{M}(\varepsilon ,\delta
_{\varepsilon _{n}})}\inf_{\left( x,y\right) \in \mathbf{B}}\left\vert
\left( \Phi \left( u\right) ,\Phi \left( v\right) \right) -\left( x,y\right)
\right\vert +o\left( 1\right) .
\end{equation*}%
To prove $\left( \ref{5.1}\right) ,$ it is sufficient to find points $\left(
x_{n},y_{n}\right) \in \mathbf{B}$ such that%
\begin{equation}
\lim_{n\rightarrow \infty }\left\vert \left( \Phi \left( u_{n}\right) ,\Phi
\left( v_{n}\right) \right) -\left( x_{n},y_{n}\right) \right\vert =0,
\label{5.2}
\end{equation}%
possibly up to subsequence. For any $n\in \mathbb{N},$ let $\widetilde{u}%
_{n}\left( x\right) =u_{n}\left( \varepsilon _{n}x\right) $ and $\widetilde{v%
}_{n}\left( x\right) =v_{n}\left( \varepsilon _{n}x\right) .$ Then similar
to the argument in the proof of Lemma \ref{L2.5}, we have
\begin{equation*}
\lim_{n\rightarrow \infty }\int_{\mathbb{R}^{N}}\left( \left\vert
(-\triangle )^{\frac{\alpha }{2}}\widetilde{u}_{n}\right\vert
^{2}+V_{1}\left( \varepsilon _{n}x\right) \widetilde{u}_{n}^{2}\right)
\mathrm{d}x=\frac{2p+2}{p}\widehat{\alpha }_{1}
\end{equation*}%
and%
\begin{equation*}
\lim_{n\rightarrow \infty }\int_{\mathbb{R}^{N}}\left( \left\vert
(-\triangle )^{\frac{\alpha }{2}}\widetilde{v}_{n}\right\vert
^{2}+V_{2}\left( \varepsilon _{n}x\right) \widetilde{v}_{n}^{2}\right)
\mathrm{d}x=\frac{2p+2}{p}\widehat{\alpha }_{2}.
\end{equation*}%
Furthermore, there exists $\left\{ \left( \widetilde{x}_{n},\widetilde{y}%
_{n}\right) \right\} \subset \mathbb{R}^{N}\times \mathbb{R}^{N}$ such that
\newline
$(i)\ \widetilde{u}_{n}\left( \cdot +\widetilde{x}_{n}\right) $ converges
strongly in $H^{\alpha }\left( \mathbb{R}^{N}\right) $ to $\widehat{u}_{0},$
a positive ground state solution of $\left( \widehat{E}_{1}\right) ;$\newline
$(ii)\ \widetilde{v}_{n}\left( \cdot +\widetilde{y}_{n}\right) $ converges
strongly in $H^{\alpha }\left( \mathbb{R}^{N}\right) $ to $\widehat{v}_{0},$
a positive ground state solution of $\left( \widehat{E}_{2}\right) .$\newline
Let us prove that $\left\{ \left( \varepsilon _{n}\widetilde{x}%
_{n},\varepsilon _{n}\widetilde{y}_{n}\right) \right\} $ is a bounded
sequence in $\mathbb{R}^{N}\times \mathbb{R}^{N}.$ Arguing by contradiction,
we assume that $\left\vert \varepsilon _{n}\widetilde{x}_{n}\right\vert
\rightarrow \infty $ as $n\rightarrow \infty .$ Then%
\begin{eqnarray*}
&&\lim_{n\rightarrow \infty }\int_{\mathbb{R}^{N}}\left( \left\vert
(-\triangle )^{\frac{\alpha }{2}}\widetilde{u}_{n}\right\vert
^{2}+V_{1}\left( \varepsilon _{n}x\right) \widetilde{u}_{n}^{2}\right)
\mathrm{d}x \\
&=&\lim_{n\rightarrow \infty }\int_{\mathbb{R}^{N}}\left( \left\vert
(-\triangle )^{\frac{\alpha }{2}}\widetilde{u}_{n}\left( x+\widetilde{x}%
_{n}\right) \right\vert ^{2}+V_{1}\left( \varepsilon _{n}x+\varepsilon _{n}%
\widetilde{x}_{n}\right) \widetilde{u}_{n}^{2}\left( x+\widetilde{x}%
_{n}\right) \right) \mathrm{d}x \\
&\geq &\int_{\mathbb{R}^{N}}\left( \left\vert (-\triangle )^{\frac{\alpha }{2%
}}u_{0}\right\vert ^{2}+V_{1,\infty }u_{0}^{2}\right) \mathrm{d}x>\int_{%
\mathbb{R}^{N}}\left( \left\vert (-\triangle )^{\frac{\alpha }{2}%
}u_{0}\right\vert ^{2}+\lambda _{1}u_{0}^{2}\right) \mathrm{d}x \\
&=&\frac{2p+2}{p}\widehat{\alpha }_{1},
\end{eqnarray*}%
which is a contradiction. Thus, $\left\{ \left( \varepsilon _{n}\widetilde{x}%
_{n},\varepsilon _{n}\widetilde{y}_{n}\right) \right\} $ is bounded and
converges to some $\left( x_{0},y_{0}\right) $ (up to a subsequence). We are
left to prove $\left( x_{0},y_{0}\right) =\left( z_{1,i},z_{2,j}\right) $
for some $1\leq i\leq k,1\leq j\leq \ell .$ Because
\begin{eqnarray*}
\frac{2p+2}{p}\widehat{\alpha }_{1} &=&\int_{\mathbb{R}^{N}}\left(
\left\vert (-\triangle )^{\frac{\alpha }{2}}u_{0}\right\vert ^{2}+\lambda
_{1}u_{0}^{2}\right) \mathrm{d}x \\
&=&\lim_{n\rightarrow \infty }\int_{\mathbb{R}^{N}}\left[ \left\vert
(-\triangle )^{\frac{\alpha }{2}}\widetilde{u}_{n}\left( x+\widetilde{x}%
_{n}\right) \right\vert ^{2}+V_{1}\left( \varepsilon _{n}x+\varepsilon _{n}%
\widetilde{x}_{n}\right) \widetilde{u}_{n}^{2}\left( x+\widetilde{x}%
_{n}\right) \right] \mathrm{d}x \\
&=&\int_{\mathbb{R}^{N}}\left( \left\vert (-\triangle )^{\frac{\alpha }{2}%
}u_{0}\right\vert ^{2}+V_{1,\infty }(x_{0})u_{0}^{2}\right) \mathrm{d}x
\end{eqnarray*}%
and%
\begin{eqnarray*}
\frac{2p+2}{p}\widehat{\alpha }_{2} &=&\int_{\mathbb{R}^{N}}\left(
\left\vert (-\triangle )^{\frac{\alpha }{2}}v_{0}\right\vert ^{2}+\lambda
_{2}v_{0}^{2}\right) \mathrm{d}x \\
&=&\lim_{n\rightarrow \infty }\int_{\mathbb{R}^{N}}\left[ \left\vert
(-\triangle )^{\frac{\alpha }{2}}\widetilde{v}_{n}\left( x+\widetilde{y}%
_{n}\right) \right\vert ^{2}+V_{2}\left( \varepsilon _{n}x+\varepsilon _{n}%
\widetilde{y}_{n}\right) \widetilde{v}_{n}^{2}\left( x+\widetilde{y}%
_{n}\right) \right] \mathrm{d}x \\
&=&\int_{\mathbb{R}^{N}}\left( \left\vert (-\triangle )^{\frac{\alpha }{2}%
}v_{0}\right\vert ^{2}+V_{2,\infty }(y_{0})v_{0}^{2}\right) \mathrm{d}x,
\end{eqnarray*}%
which implies that $V_{1,\infty }\left( x_{0}\right) =\lambda _{1}$ and $%
V_{2,\infty }\left( y_{0}\right) =\lambda _{2},$ that is, there exists $%
1\leq i\leq k,1\leq j\leq \ell $ such that $\left( x_{0},y_{0}\right)
=\left( z_{1,i},z_{2,j}\right) .$ Take $x_{n}=\varepsilon _{n}\widetilde{x}%
_{n}$ and $y_{n}=\varepsilon _{n}\widetilde{y}_{n}.$ Then $\left( \ref{5.2}%
\right) $ holds.
\end{proof}

\begin{lemma}
\label{L5.2} Assume that $\liminf\limits_{|x|\rightarrow \infty
}\,V_{l}(x)\equiv V_{l,\infty }>\lambda _{l}$ for all $l=1,2.$ Then there
exists a positive number $\varepsilon _{\ast \ast }$ such that for every $%
\varepsilon <\varepsilon _{\ast \ast },$ we have $\widehat{M}(\varepsilon
,\delta _{\varepsilon })\subset \underset{1\leq i\leq k,1\leq j\leq \ell }{%
\cup }N_{i,j}\left( \varepsilon \right) .$
\end{lemma}

\begin{proof}
By Lemma \ref{L5.1}, we can find $\varepsilon _{\ast \ast }>0$ such that for
every $\varepsilon <\varepsilon _{\ast \ast },$ we have%
\begin{equation*}
\sup_{\left( u,v\right) \in \widehat{M}(\varepsilon ,\delta _{\varepsilon
})}\inf_{\left( x,y\right) \in \mathbf{B}}\left\vert \left( \Phi \left(
u\right) ,\Phi \left( v\right) \right) -\left( x,y\right) \right\vert <\frac{%
s}{2}
\end{equation*}%
or%
\begin{equation*}
\mathrm{dist}\left( \left( \Phi \left( u\right) ,\Phi \left( v\right)
\right) ,\mathbf{B}\right) <\frac{s}{2}\text{ for all }\left( u,v\right) \in
\widehat{M}(\varepsilon ,\delta _{\varepsilon }),
\end{equation*}%
where $\mathbf{B}=\underset{1\leq i\leq k,1\leq j\leq \ell }{\cup }\left[
C_{s/2}\left( z_{1,i}\right) \times C_{s/2}\left( z_{2,j}\right) \right] .$
This implies that%
\begin{equation*}
\left( \Phi \left( u\right) ,\Phi \left( v\right) \right) \in \mathbf{B}%
\text{ for all }\left( u,v\right) \in \widehat{M}(\varepsilon ,\delta
_{\varepsilon }).
\end{equation*}%
Thus, $\widehat{M}(\varepsilon ,\delta _{\varepsilon })\subset \underset{%
1\leq i\leq k,1\leq j\leq \ell }{\cup }N_{i,j}\left( \varepsilon \right) .$
\end{proof}

\begin{lemma}
\label{L5.3}Assume that $\lim\limits_{|x|\rightarrow \infty
}V_{l}(x)=\lambda _{l}$ for all $l=1,2.$ Then%
\begin{equation*}
\varepsilon ^{-N}c_{\varepsilon }=\widehat{\alpha }_{1}+\widehat{\alpha }%
_{2}~\text{ for all }\varepsilon >0.
\end{equation*}
\end{lemma}

\begin{proof}
To prove Lemma~\ref{L5.3}, we may define the test functions $u_{R}$ and $%
v_{R}$ by
\begin{eqnarray}
&&\left( u_{R}\left( x\right) ,v_{R}\left( x\right) \right)  \label{5.3} \\
&=&\left( \widehat{\omega }_{1}\left( \frac{x+Re}{\varepsilon }\right) \psi
_{R}\left( \frac{x+Re}{\varepsilon }\right) ,\widehat{\omega }_{2}\left(
\frac{x-Re}{\varepsilon }\right) \psi _{R}\left( \frac{x-Re}{\varepsilon }%
\right) \right) \,,  \notag
\end{eqnarray}%
where $R>1$, $e\in S^{N-1}=\left\{ x\in \mathbb{R}^{N}:\left\vert
x\right\vert =1\right\} $, $\widehat{\omega }_{l}$ is the unique positive
radial solution of Eq. $\left( \widehat{E}_{l}\right) $ for $l=1,2$, and the
function $\psi _{R}\in C^{1}\left( \mathbb{R}^{N},\left[ 0,1\right] \right) $
with compact support satisfies
\begin{equation*}
\psi _{R}\left( x\right) =\left\{
\begin{array}{ll}
1, & \left\vert x\right\vert <R-1, \\
0, & \left\vert x\right\vert >R,%
\end{array}%
\right.
\end{equation*}%
and $\left\vert \nabla \psi _{R}\right\vert \leq 2$ in $\mathbb{R}^{N}.$
Notice that $\widehat{I}_{l}\left( \widehat{\omega }_{l}\right) =\widehat{%
\alpha }_{1}$ for $l=1,2.$ Obviously,
\begin{equation}
\int_{\mathbb{R}^{N}}\beta (x)|u_{R}^{+}|^{q+1}|v_{R}^{+}|^{q+1}\mathrm{d}%
x=0.  \label{5.4}
\end{equation}%
By $(\ref{5.3}),$ it is easy to find two positive numbers $t_{R}$ and $s_{R}$
such that
\begin{equation}
t_{R}^{2}\int_{\mathbb{R}^{N}}\left( \varepsilon ^{2\alpha }\left\vert
(-\triangle )^{\frac{\alpha }{2}}u_{R}\right\vert
^{2}+V_{1}(x)u_{R}^{2}\right) \mathrm{d}x=\mu _{1}t_{R}^{2p+2}\int_{\mathbb{R%
}^{N}}|u_{R}^{+}|^{2p+2}\mathrm{d}x  \label{5.5}
\end{equation}%
and
\begin{equation}
s_{R}^{2}\int_{\mathbb{R}^{N}}\left( \varepsilon ^{2\alpha }\left\vert
(-\triangle )^{\frac{\alpha }{2}}v_{R}\right\vert
^{2}+V_{2}(x)v_{R}^{2}\right) \mathrm{d}x=\mu _{2}s_{R}^{2p+2}\int_{\mathbb{R%
}^{N}}|v_{R}^{+}|^{2p+2}\mathrm{d}x.  \label{5.6}
\end{equation}%
It follows from $(\ref{5.4})-(\ref{5.6})$ that $\left(
t_{R}u_{R},s_{R}v_{R}\right) \in \mathbf{N}_{\varepsilon }$ and%
\begin{equation*}
J_{\varepsilon }\left( t_{R}u_{R},s_{R}v_{R}\right) \geq c_{\varepsilon }.
\end{equation*}%
We want to show $\left( t_{R},s_{R}\right) \rightarrow \left( 1,1\right) $
as $R\rightarrow \infty .$ Using $(\ref{5.3}),(\ref{5.5})$ and change of
variables, it is easy to check that
\begin{equation*}
\int_{\mathbb{R}^{N}}\left( \left\vert (-\triangle )^{\frac{\alpha }{2}}%
\widehat{\omega }_{1}\right\vert ^{2}+V_{1}\left( \varepsilon x-Re\right)
\widehat{\omega }_{1}^{2}\right) \mathrm{d}x=\mu _{1}t_{R}^{2p}\int_{\mathbb{%
R}^{N}}|\widehat{\omega }_{1}^{+}|^{2p+2}\mathrm{d}x+o_{R}\left( 1\right) \,,
\end{equation*}%
where $o_{R}\left( 1\right) \rightarrow 0$ as $R\rightarrow \infty .$ Hence,
$t_{R}\rightarrow 1$ as $R\rightarrow \infty $ because $\lim_{|x|\rightarrow
\infty }V_{1}(x)=\lambda _{1}$ and
\begin{equation*}
\int_{\mathbb{R}^{N}}\left( \left\vert (-\triangle )^{\frac{\alpha }{2}}%
\widehat{\omega }_{1}\right\vert ^{2}+\lambda _{1}\widehat{\omega }%
_{1}^{2}\right) \mathrm{d}x=\mu _{1}\int_{\mathbb{R}^{N}}|\widehat{\omega }%
_{1}^{+}|^{2p+2}\mathrm{d}x.
\end{equation*}%
Similarly, we may obtain $s_{R}\rightarrow 1$ as $R\rightarrow \infty $.
Consequently, due to $\lim\limits_{|x|\rightarrow \infty }V_{l}(x)=\lambda
_{l}$ for $l=1,2$, we have
\begin{equation*}
\varepsilon ^{-N}J_{\varepsilon }\left( t_{R}u_{R},s_{R}v_{R}\right)
\rightarrow \widehat{\alpha }_{1}+\widehat{\alpha }_{2}\text{ as }%
R\rightarrow \infty .
\end{equation*}%
This implies%
\begin{equation*}
\varepsilon ^{-N}c_{\varepsilon }\leq \widehat{\alpha }_{1}+\widehat{\alpha }%
_{2}\text{ for all }\varepsilon >0.
\end{equation*}%
Thus, by Lemma \ref{L2.3}, we have $\varepsilon ^{-N}c_{\varepsilon }=%
\widehat{\alpha }_{1}+\widehat{\alpha }_{2}$ for all $\varepsilon >0.$
\end{proof}

\noindent \textbf{We are now ready to prove Theorem \ref{t1.2}}. Theorem \ref%
{t1.2} can be obtained directly from Lemma \ref{L2.3}, Lemma \ref{L5.2} and
Lemma \ref{L5.3}.

\end{document}